\documentclass[11pt]{article}
\usepackage[a4paper,margin=2.7cm]{geometry}
\usepackage[comma,square,sort&compress,numbers]{natbib}
\usepackage{amsmath, amsthm, amssymb}
\usepackage{graphicx}
\usepackage{caption}
\usepackage{subcaption}
\usepackage{pgf,tikz}
\usepackage{txfonts}
\usepackage[pdfpagemode=UseOutlines]{hyperref}
\hypersetup{
    colorlinks=false,
    pdfborder={0 0 0},
    pdfauthor={B. de Rijk, G. Schneider},
   pdftitle={Global existence and decay in nonlinearly coupled reaction-diffusion-advection equations with different velocities}
}
\usepackage{enumitem}

\makeatletter
\def\namedlabel#1#2{\begingroup
    #2%
    \def\@currentlabel{#2}%
    \phantomsection\label{#1}\endgroup
}
\makeatother
\newcommand{\R}{\mathbb{R}}\newcommand{\Z}{\mathbb{Z}}

\newcommand{\El}{\mathcal{L}}

\newcommand{\Es}{\mathcal{S}}

\newcommand{\U}{\mathcal{U}}

\newcommand{\N}{\mathcal{N}}

\newcommand{\B}{\mathcal{B}}

\newcommand{\V}{\mathcal{V}_{\mathrm{in}}}
\newcommand{\I}{\mathcal{I}}

\numberwithin{equation}{section}

\def\de{\mathrm{d}}

\let\epsilon\varepsilon

\let\phi\varphi
\def\xt{\zeta}
\def\vt{z} 
\def\ut{w}
\def\V{\mathcal{V}}
\def\tp{{\tilde{p}}}

\newtheorem{theorem}{Theorem}[section]

\newtheorem{proposition}[theorem]{Proposition}

\newtheorem{remark}[theorem]{Remark}

\makeatletter
\renewenvironment{proof}[1][\proofname] {\par\pushQED{\qed}\normalfont\topsep6\p@\@plus6\p@\relax\trivlist\item[\hskip\labelsep\bfseries#1\@addpunct{.}]\ignorespaces}{\popQED\endtrivlist\@endpefalse} \makeatother

\title{Global existence and decay in nonlinearly coupled reaction-diffusion-advection equations with different velocities}

\author{Bj\"orn de Rijk\footnote{\texttt{bjoern.derijk@mathematik.uni-stuttgart.de}} \and Guido Schneider\footnote{\texttt{guido.schneider@mathematik.uni-stuttgart.de}}}

\begin{document}

\maketitle

\begin{abstract}
We develop techniques to capture the effect of transport on the long-term dynamics of small, localized initial data in nonlinearly coupled reaction-diffusion-advection equations on the real line. It is well-known that quadratic or cubic nonlinearities in such systems can lead to growth of small, localized initial data and even finite time blow-up. We show that, if the components exhibit different velocities, then quadratic or cubic mix-terms, i.e.~terms with nontrivial contributions from both components, are harmless. We establish global existence and diffusive Gaussian-like decay for exponentially and algebraically localized initial conditions allowing for quadratic and cubic mix-terms. Our proof relies on a nonlinear iteration scheme that employs pointwise estimates. The situation becomes very delicate if \emph{other} quadratic or cubic terms are present in the system. We provide an example where a quadratic mix-term and a Burgers'-type coupling can compensate for a cubic term due to differences in velocities.

\end{abstract}

\section{Introduction}

We consider nonlinearly coupled reaction-diffusion-advection equations on the real line of the form
\begin{align}
\begin{split}
u_t &= d_1 u_{xx} + c_1 u_x + f_1(u,v) + \left(g_1(u,v)\right)_x, \\
v_t &= d_2 v_{xx} + c_2 v_x + f_2(u,v) + \left(g_2(u,v)\right)_x,
\end{split} \qquad t \geq 0, x \in \R, \label{GRD}
\end{align}
with diffusion coefficients $d_i > 0$, velocities $c_i \in \R$ and smooth nonlinearities $f_i,g_i \colon \R^2 \to \R$ satisfying $f_i(0), \nabla\! f_i (0) = 0$ and $g_i(0), \nabla\! g_i (0) = 0$. We are interested in the effect of the nonlinearities and the velocities in~\eqref{GRD} on the long-term dynamics of small, localized initial data. 

Reaction-diffusion-advection (or -convection) systems describe diffusive reagents that each undergo a spatial drift. They are prototype models for pattern formation in many scientific disciplines, see~\cite{BOR,KEE,KUR,MUR} for applications in ecology, physiology, chemistry and biology. It is well-known~\cite{DOELEX,ROVM,TUR} that patterns can arise when a homogeneous rest state loses its stability and species diffuse or drift at different rates. In various models the interactions within and among species are purely nonlinear and thus the dynamics is described by a system of the form~\eqref{GRD}. Such nonlinear interactions arise naturally through mass-action kinetics~\cite{ERDI,GROE}; we refer to~\cite{BAA,BISI,MATK} for examples from (reversible) chemistry and combustion theory. In the case of such nonlinear interactions, the (essential) spectrum of the linearization about the homogeneous rest state touches the imaginary axis at the origin. Consequently, the stability of the homogeneous rest state cannot be determined by a spectral analysis. In fact, the nonlinearities can be decisive for stability against small, localized perturbations. This principle is well-understood in \emph{scalar} reaction-diffusion-advection equations, see Remark~\ref{scalar}.

In this paper, we show that in reaction-diffusion-advection \emph{systems} of the form~\eqref{GRD} not only the nonlinearities, but also the occurrence of different velocities, can be decisive for the long-term dynamics of small, localized initial data. Such differences in velocities occur naturally in various reaction-diffusion-advection models. For example, in pipe flow models~\cite{BARK} turbulence is advected more slowly than the centerline velocity. Moreover, in the Klausmaier model for semi-arid ecosystems on a sloped terrain~\cite{KLAU} the flow of water is governed by advection whereas biomass spreads only diffusively. In transport-reaction problems in porous media~\cite{KRAU} one distinguishes between mobile species undergoing advection-diffusion and immobile species. Finally, in semiconductor models~\cite{ALBI,Hu15} the electric field causes species to drift proportional to their (species-dependent) charge.

\subsection{Classification of nonlinearities}

In order to develop some intuition for which nonlinearities might be decisive for the long-term dynamics of small, localized solutions to~\eqref{GRD}, we formally separate linear from nonlinear behavior. By removing the nonlinear terms in~\eqref{GRD}, the components fully decouple and we obtain the diffusion-advection system
\begin{align}
\begin{split}
\phi_t &= D\phi_{xx} + C\phi_x,
\end{split} \qquad t \geq 0, x \in \R. \label{GRDLIN}
\end{align}
where $\phi = (u,v) \in \R^2$, $D = \mathrm{diag}(d_1,d_2)$ and $C = \mathrm{diag}(c_1,c_2)$. Localized solutions to~\eqref{GRDLIN} in $L^1(\R) \cap L^\infty(\R)$ decay in $L^\infty(\R)$ with algebraic rate $t^{-1/2}$. A family of localized solutions to~\eqref{GRDLIN} exhibiting the slowest decay rate $t^{-1/2}$ are linear combinations of the drifting Gaussians
\begin{align}
\frac{e^{-\tfrac{(x+c_it)^2}{4d_i(1+t)}}}{\sqrt{1+t}} E_i, \quad i = 1,2, \label{Gaussians}
\end{align}
where $E_i \in \R^2$ is the $i$-th unit vector. For such solutions 
the linear term $\phi_t - D \phi_{xx} - C\phi_x$ decays over time with rate $t^{-3/2}$. On the other hand, nonlinear terms of the form
\begin{align} \partial_x^\gamma (u^\alpha v^\beta), \qquad \alpha,\beta \in \Z_{\geq 0},\; \gamma = 0,1, \label{Gnon1}\end{align}
decay with rate $t^{-p/2}$ where $p := \alpha + \beta + \gamma$. Thus, if $p > 3$, then, for the slowest decaying solutions to the linearized system~\eqref{GRDLIN}, the nonlinear term~\eqref{Gnon1} decays faster than the linear term $\phi_t - D \phi_{xx} - C\phi_x$. This suggests that the linear dynamics about the rest state $(u,v) = 0$ in~\eqref{GRD} is not altered by the nonlinearity. In such a case we call the nonlinearity~\eqref{Gnon1} \emph{irrelevant}. On the other hand, for $p < 3$ the nonlinear dynamics might be dominant, in which case we say the nonlinearity~\eqref{Gnon1} is \emph{relevant}. Finally, a nonlinearity~\eqref{Gnon1} with critical value $p = 3$ is called \emph{marginal}. Thus, any smooth nonlinearity in~\eqref{GRD} can be labeled relevant, marginal or irrelevant by looking at the leading-order term of its power series expansion. This classification of (smooth) nonlinearities was introduced in~\cite{BKL} and can be generalized to reaction-diffusion-advection systems in $d$ spatial dimensions by replacing the critical threshold $p = 3$ by $p = 1 + \frac{2}{d}$, see also~\cite[Section 2]{UEC9}.

In scalar reaction-diffusion-advection equations, it was proven in~\cite{PONC,ZHENG} that, if the nonlinearity is irrelevant, then small, sufficiently localized initial data decay over time with rate $t^{-1/2}$. This proof extends without effort from the scalar to the multi-component setting. On the other hand, it was shown in~\cite{ESCLEV} that every solution to the nonlinearly coupled system
\begin{align*}
\begin{split}
u_t &= u_{xx} + u^{p_1}v^{q_1},\\
v_t &= v_{xx} + u^{p_2}v^{q_2},
\end{split} \qquad t \geq 0, x \in \R, 
\end{align*}
with initial condition $(u_0,v_0)$, such that $u_0,v_0 \geq 0$ and $u_0v_0 \neq 0$ hold pointwise, blows up in finite time if the nonlinearities are relevant or marginal, i.e.~if $p_i,q_i \geq 0$ satisfy $2 \leq p_i+q_i \leq 3$ for $i = 1,2$. Thus, as in the scalar setting (see Remark~\ref{scalar}) relevant or marginal nonlinearities are not automatically controlled by the linear  dynamics in~\eqref{GRD} and can be decisive for the long-term dynamics of small, localized solutions.

\begin{remark} \label{scalar}
{\upshape The effect of nonlinearities on the long-term dynamics of small, localized solutions has been well-studied in the scalar version of~\eqref{GRD}, i.e.~in the reaction-diffusion-advection equation
\begin{align}
\begin{split}
u_t &= d u_{xx} + c u_x + f(u) + \left(g(u)\right)_x,
\end{split} \qquad t \geq 0, x \in \R, \label{GRDSC}
\end{align}
with $d > 0, c \in \R$ and $f,g \colon \R \to \R$ smooth nonlinearities with $f(0), f'(0) = 0$ and $g(0), g'(0) = 0$. The transport term and the diffusion coefficient in~\eqref{GRDSC} can be eliminated by switching to the co-moving frame $\xi = d^{-1/2}(x + ct)$, which transforms~\eqref{GRDSC} into
\begin{align}
u_t &= u_{\xi\xi} + f(u) + d^{-\frac{1}{2}} \left(g(u)\right)_\xi, \qquad t \geq 0, \xi \in \R. \label{GRDCM}
\end{align}
It is proven in~\cite{PONC,ZHENG} that any irrelevant nonlinearity in~\eqref{GRDCM} leads to diffusive decay of small, sufficiently localized initial data with rate $t^{-1/2}$. Moreover, any relevant nonlinearity has, after rescaling, $u^2$ as leading-order coefficient, which can lead to growth of small, localized initial data and blow up in finite time~\cite{FUJI}. All marginal nonlinearities have $\alpha u^3 + \beta uu_\xi$ as leading-order coefficient for some $\alpha,\beta \in \R$. Having $\alpha > 0$ can lead to growth and blow up in finite time~\cite{HAYA}, whereas it is shown in~\cite{BKL} that for $\alpha \leq 0$ all small, localized initial data in~\eqref{GRDCM} decays with rate $t^{-1/2}$.
}\end{remark}

\subsection{The effect of different velocities}

We claim that, in contrast to the scalar setting, the velocities $c_i$ in system~\eqref{GRD} play a pivotal role in determining which relevant and marginal nonlinear terms can and cannot be controlled by the linear dynamics. In order to illustrate the latter on a formal level, we introduce the notions of a \emph{nonlinear coupling}, which is a nonlinear term~\eqref{Gnon1} in the $u$-equation with $\beta \geq 1$ or in the $v$-equation with $\alpha \geq 1$, and the notion of a \emph{mix-term}, which is a term of the form~\eqref{Gnon1} with both $\alpha \geq 1$ and $\beta \geq 1$. As above, we consider the slowest decaying, localized solutions to the linear system~\eqref{GRDLIN}, which are linear combinations of the drifting Gaussians~\eqref{Gaussians}. For such solutions, mix-terms decay with \emph{exponential} rate
\begin{align*} t^{-\tfrac{p}{2}} e^{-\tfrac{\left(c_1 - c_2\right)^2 t}{4\left(d_1+d_2\right)}},\end{align*}
with $p = \alpha + \beta + \gamma$, which can be seen by completing the square. This suggests that, if components propagate with different velocities $c_1 \neq c_2$, then mix-terms do not alter the linear dynamics about the rest state $(u,v) = 0$ in~\eqref{GRD}, no matter the value of $p$.

The first result of this paper confirms this conjecture. We prove that, if all nonlinear couplings in~\eqref{GRD} are mix-terms and components exhibit different velocities, then solutions with small, algebraically or exponentially localized initial data exist globally and decay over time with rate $t^{-1/2}$. We note that the requirement that all nonlinear couplings in~\eqref{GRD} are mix-terms is equivalent to the condition that $\{u = 0\}$ and $\{v = 0\}$ are invariant subspaces for system~\eqref{GRD}.

We take our initial data from (weighted) H\"older spaces, since in such spaces local existence and uniqueness of classical solutions to~\eqref{GRD} is naturally obtained; see for instance~\cite{LUN}. Of course, one could consider more general spaces of initial data, but since our main focus lies on the effects of different velocities on the long-term dynamics, we refrain from doing so. Thus, we let $\alpha \in (0,1)$ and introduce the spaces
\begin{align*} X_\rho^\alpha := \left\{z \in C^{0,\alpha}(\R,\R^2) : \|z\rho\|_\infty < \infty\right\}, \end{align*}
of bounded, H\"older continuous initial conditions with weight $\rho \colon \R \to [1,\infty)$. We endow $X_\rho^\alpha$ with the norm $\|z\|_{\rho} := \|z\rho\|_\infty$, where $\|\cdot\|_\infty$ denotes the supremum norm, and establish the following result.

\begin{theorem} \label{mainresult1}
Let $\alpha \in (0,1)$. Suppose $d_i > 0$ and $c_1 \neq c_2$ in~\eqref{GRD} and assume that there exist constants $C> 1$ and $r_0 > 0$ such that the nonlinearities $f_i,g_i \in C^{2,\alpha}(\R^2,\R)$ in~\eqref{GRD} satisfy
\begin{align}
\begin{split}
|f_1(u,v)| \leq C\left(|u|^4 + |u||v|\right), &\qquad |g_1(u,v)| \leq C\left(|u|^2 + |u||v|\right),\\
|f_2(u,v)| \leq C\left(|v|^4 + |u||v|\right), &\qquad |g_2(u,v)| \leq C\left(|v|^2 + |u||v|\right),
\end{split} \qquad \text{for } |u|,|v| \leq r_0. \label{nonlinearbounds1}
\end{align}
Then, there exists $M_0 \geq 1$ such that, for all $M \geq M_0$, $r \geq 3$ and $\epsilon > 0$, there exists a $\delta > 0$ such that~\eqref{GRD} together with one of the following initial conditions
\begin{itemize}
\item[E)] $(u_0,v_0) \in X_{\rho_E}^\alpha$ satisfying $\|(u_0,v_0)\|_{\rho_E} < \delta$ with exponential weight $\rho_E(x) = e^{x^2/M}$;
\item[A)] $(u_0,v_0) \in X_{\rho_A}^\alpha$ satisfying $\|(u_0,v_0)\|_{\rho_A} < \delta$ with algebraic weight $\rho_A(x) = (1+|x|)^r$;
\end{itemize}
has a classical global solution $(u,v) \in C^{1,\frac{\alpha}{2}}\left([0,\infty),C^{2,\alpha}(\R,\R^2)\right)$ satisfying
\begin{align} \|(u, v)(\cdot,t)\|_\infty  \leq \frac{\epsilon}{\sqrt{1 + t}}, \qquad \|(u, v)(\cdot,t)\|_1 \leq \epsilon, \qquad \text{for } t \geq 0. \label{temporaldec}\end{align}
More specifically, for exponentially localized initial data $(u_0,v_0) \in X_{\rho_E}^\alpha$ satisfying $\|(u_0,v_0)\|_{\rho_E} < \delta$ we obtain the Gaussian decay estimates
\begin{align} \left|u(x,t)\right| \leq \epsilon \frac{e^{-\tfrac{(x+c_1t)^2}{M(1+t)}}}{\sqrt{1+t}}, \qquad \left|v(x,t)\right| \leq \epsilon\frac{e^{-\tfrac{(x+c_2t)^2}{M(1+t)}}}{\sqrt{1+t}},\label{pointdec}\end{align}
for all $x \in \R$ and $t \geq 0$.
\end{theorem}

To exploit the difference in spatial transport between components, we use pointwise estimates to prove Theorem~\ref{mainresult1}. For exponentially localized initial data, we adopt the slowest decaying, localized solutions to the linear system~\eqref{GRDLIN}, i.e.~the Gaussians~\eqref{Gaussians}, as pointwise upper bounds in a spatio-temporal nonlinear iteration scheme, which eventually yields~\eqref{pointdec}. Hence, if the nonlinearity in~\eqref{GRD} satisfies~\eqref{nonlinearbounds1}, then the obtained pointwise decay is as predicted by the linear dynamics.

Algebraically localized initial data lead to additional algebraic correction terms in the pointwise upper bounds. The associated pointwise decay estimates are technically more involved than~\eqref{pointdec} and can be found in Remark~\ref{spatiotemporalbounds1}.

The method of pointwise estimates was developed in the setting of viscous shock waves~\cite{ZUH,HOWZUM} and has proven to be a powerful tool in a large variety of nonlinear stability problems. We refer to~\cite{BECEX} for an expository article and to~\cite[Section 6]{JUN} for a straightforward application in the setting of the nonlinear heat equation $u_t = u_{xx} + u^p$ with irrelevant nonlinearity, i.e.~with $p > 3$. In~\S\ref{overview} we give an extensive, but by no means exhaustive, overview of other methods to prove global existence of (small) solutions in reaction-diffusion systems, see also~\cite[Section 14]{SUbook}. To the author's best knowledge, the effect of different velocities in~\eqref{GRD} on the long-time dynamics of small initial data has not been investigated in literature prior to this paper.

\subsection{Including irrelevant nonlinear couplings which are not of mix-type} \label{sec:mainresult2}

If \emph{nonlinear couplings which are not of mix-type} are present in system~\eqref{GRD}, i.e.~if the $u$-equation in~\eqref{GRD} has a nonlinear term~\eqref{Gnon1} with $\alpha = 0$ or the $v$-equation possesses a term~\eqref{Gnon1} with $\beta = 0$, then the analysis breaks down as the spatial localization imposed by the nonlinear iteration scheme is too restrictive. Indeed, if $v(x,t)$ is a drifting Gaussian propagating with speed $c_2$, then a $\partial_x^j (v^\beta)$-term in the $u$-equation leads to a nonlinear contribution in Duhamel's formula (or the variation of constants formula) of the form
\begin{align} \int_0^t \frac{Ce^{-\tfrac{\left(x + t c_1 + s (c_2 - c_1)\right)^2}{M(1+t)}}}{\sqrt{1+t} (1+s)^{(\beta - 1)/2} (t-s)^{j/2}} \de s, \label{drag}\end{align}
for some $C,M > 0$, which cannot be controlled by a Gaussian propagating with speed $c_1 \neq c_2$. However, by incorporating upper bounds of the form~\eqref{drag} into the nonlinear iteration scheme, we can accommodate \emph{all} irrelevant nonlinear terms in the analysis, thus in particular all irrelevant nonlinear couplings. We establish global existence and temporal decay with rate $t^{-1/2}$ for solutions to~\eqref{GRD} with small, exponentially localized initial conditions allowing for irrelevant nonlinear terms and nonlinear mix-terms.

\begin{theorem} \label{mainresult2}
Let $\alpha \in (0,1)$ and consider for $M \geq 1$ the exponential weight $\rho_E(x) = e^{x^2/M}$. Suppose $d_i > 0$ and $c_1 \neq c_2$ in~\eqref{GRD} and assume that there exist constants $C > 1$ and $r_0 > 0$ such that the nonlinearities $f_i,g_i \in C^{2,\alpha}(\R^2,\R)$ in~\eqref{GRD} satisfy
\begin{align}
\begin{split}
|f_1(u,v)| \leq C\left(|u|^4 + |u||v| + |v|^4\right), &\qquad |g_1(u,v)| \leq C\left(|u|^2 + |u||v| + |v|^3\right),\\
|f_2(u,v)| \leq C\left(|v|^4 + |u||v| + |u|^4\right), &\qquad |g_2(u,v)| \leq C\left(|v|^2 + |u||v| + |u|^3\right),
\end{split} \qquad \text{for } |u|, |v| \leq r_0. \label{nonlinearbounds2}
\end{align}
Then, there exists $M_0 \geq 1$ such that, for all $M \geq M_0$ and $\epsilon > 0$, there exists a $\delta > 0$ such that~\eqref{GRD} together with the initial condition $(u_0,v_0) \in X_{\rho_E}^\alpha$ satisfying $\|(u_0,v_0)\|_{\rho_E} < \delta$ has a classical global solution $(u,v) \in C^{1,\frac{\alpha}{2}}\left([0,\infty),C^{2,\alpha}(\R,\R^2)\right)$ satisfying
\begin{align} \|(u, v)(\cdot,t)\|_\infty  \leq \frac{\epsilon}{\sqrt{1 + t}}, \qquad \|(u, v)(\cdot,t)\|_1 \leq \epsilon, \qquad \text{for } t \geq 0. \label{temporaldec2}\end{align}
In particular, the homogeneous rest state $(0,0)$ in~\eqref{GRD} is nonlinearly stable against small, exponentially localized perturbations from $X^\alpha_{\rho_E}$.
\end{theorem}

\begin{figure}[b!]
\begin{subfigure}{.48 \textwidth}
\centering
\includegraphics[width=0.98\linewidth]{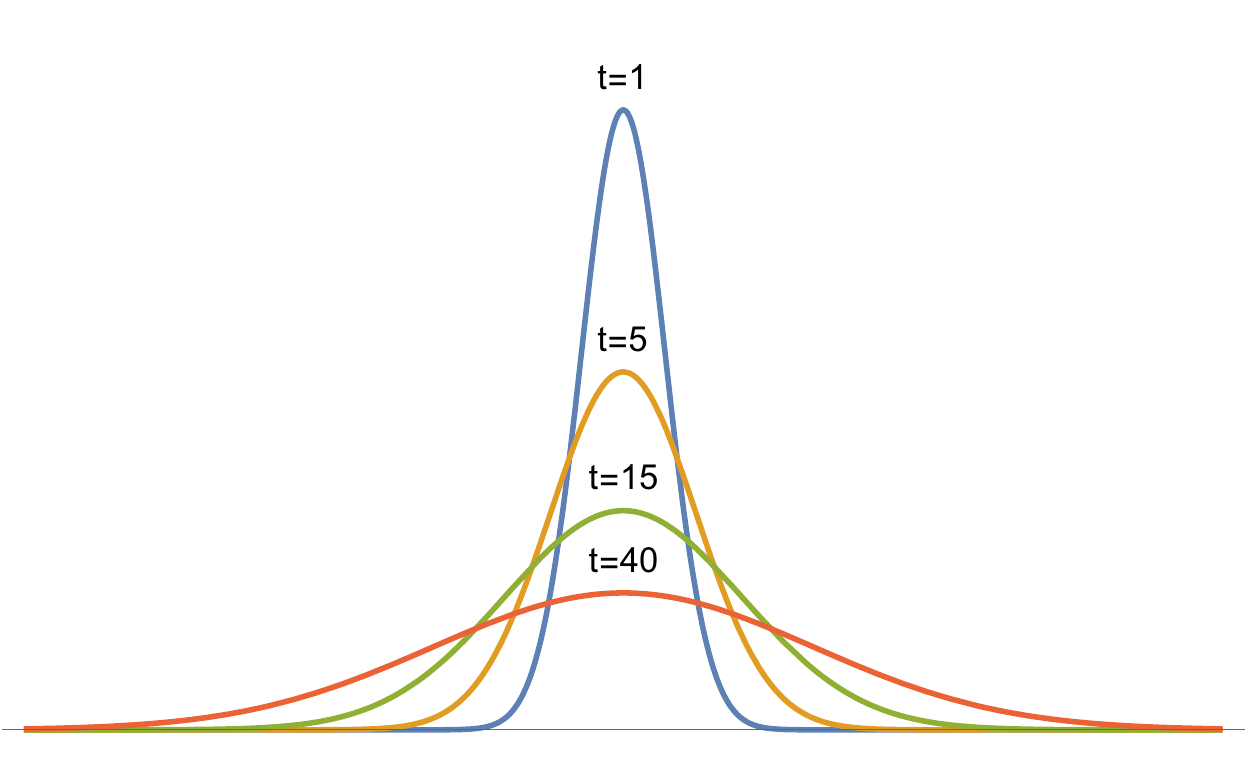}
\caption{Diffusive decay of a Gaussian $\theta(x,t) = \frac{e^{-\tfrac{x^2}{M(1+t)}}}{\sqrt{1+t}}$ with $M = 40$. }
\end{subfigure}
\hspace{.05 \textwidth}
\begin{subfigure}{.48 \textwidth}
\centering
\includegraphics[width=0.98\linewidth]{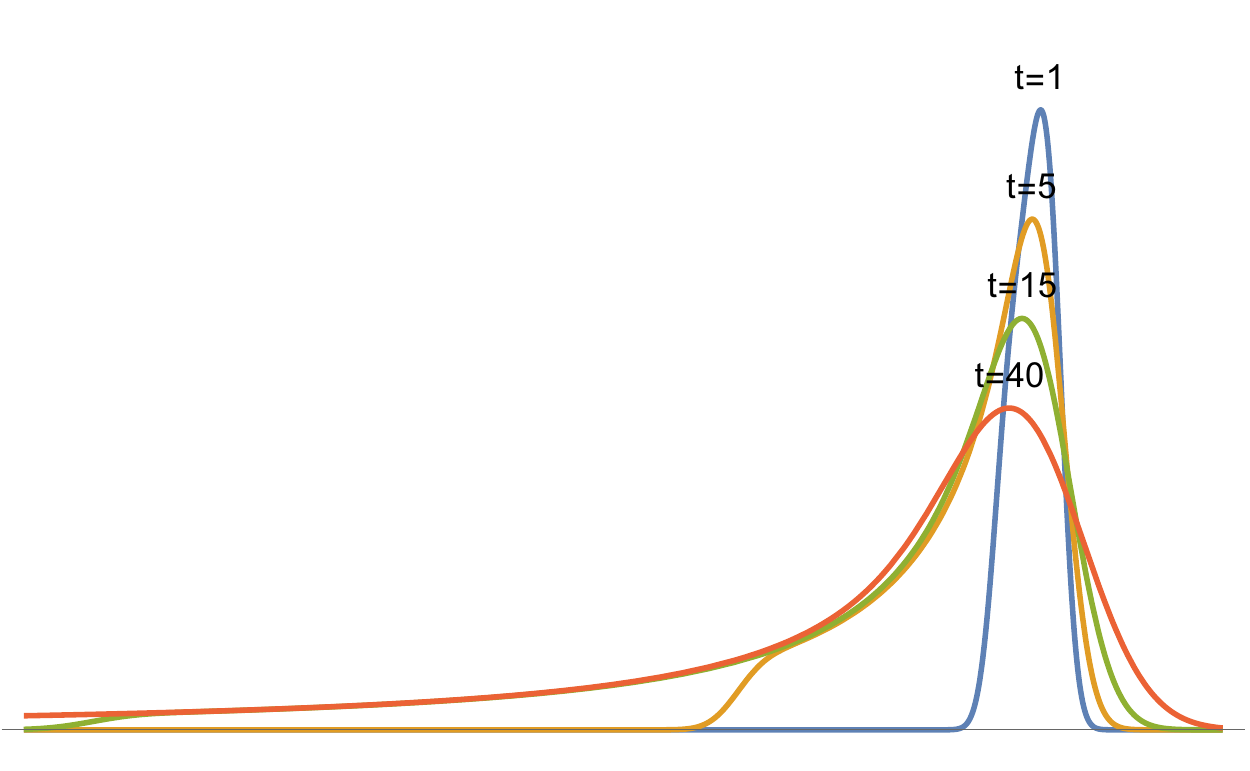}
\caption{Non-Gaussian-like decay as exhibited by the function~\eqref{drag} with $c_1 = 0$, $c_2 = 10$, $\beta = 4$, $j = 0$ and $M = 4$.}
\end{subfigure}
\caption{}\label{decaytypes}
\end{figure}

The pointwise nonlinear iteration scheme employed in the proof of Theorem~\ref{mainresult2} provides more detailed (pointwise) decay estimates on the solution than~\eqref{temporaldec2}. Due to their technicality, these decay estimates are provided in Remark~\ref{dragboundrem}. We note that the decay is no longer captured by drifting Gaussians only, as in~\eqref{pointdec}. Instead, terms of the form~\eqref{drag}, which exhibit non-Gaussian-like decay (see Figure~\ref{decaytypes}), need to be incorporated.

We prove Theorem~\ref{mainresult2} only for exponentially localized initial data. Algebraically localized initial conditions introduce additional algebraic correction terms, which would make the analysis technically more involved. Since our main goal is to capture the effect of different velocities rather than to study the precise localization properties of the initial data, we chose to work with exponential weights only in Theorem~\ref{mainresult2} for clarity of exposition.

\subsection{Including relevant or marginal couplings which are not of mix-type}

As the difference $c_2 - c_1$ is occurring in the bound~\eqref{drag}, one might expect that a difference in velocities can even be exploited to control relevant or marginal nonlinear couplings which are not of mix-type. Yet, it turns out that, in general, such relevant or marginal nonlinear couplings cannot be included in Theorem~\ref{mainresult2} and can, in fact, lead to growth of small, localized initial data. We demonstrate this by showing that any nontrivial solution with (pointwise) nonnegative initial conditions in the system
\begin{align}
\begin{split}
u_t &= d_1u_{xx} + c_1 u_x + v^2,\\
v_t &= d_2v_{xx} + c_2 v_x + u^2,
\end{split} \qquad t \geq 0, x \in \R, \label{CAS2}
\end{align}
with relevant nonlinear couplings admits lower bounds in the $L^1$- and $L^\infty$-norms that are growing over time, even in the case $c_1 \neq c_2$.

\begin{theorem} \label{mainresult3}
Let $\alpha \in (0,1)$ and let $d_i > 0$, $c_i \in \R$ in~\eqref{CAS2}. Take $(u_0,v_0) \in C^{0,\alpha}(\R,\R^2) \setminus \{0\}$ such that $u_0(x),v_0(x) \geq 0$ holds for all $x \in \R$. Then, there exists $c > 0$ such that the solution $(u,v) \in C^{1,\frac{\alpha}{2}}\left([0,T),C^{2,\alpha}(\R,\R^2)\right)$ to~\eqref{CAS2} with initial condition $(u_0,v_0)$ satisfies
\begin{align} \|(u,v)(\cdot,t)\|_\infty \geq c\log(1+t), \quad \|(u,v)(\cdot,t)\|_1 \geq ct \qquad t \in [0,T), \label{LB}\end{align}
where $[0,T)$, with $T \in (0,\infty]$, is its maximal interval of existence. In particular, the homogeneous rest state $(0,0)$ in~\eqref{GRD} is nonlinearly unstable against small, exponentially localized perturbations from $X^\alpha_{\rho_E}$ with $\rho_E(x) = e^{x^2/M}$ and $M \geq 1$.
\end{theorem}

In~\cite{ESCHER} it is shown that any nontrivial, nonnegative solution to~\eqref{CAS2} blows up in finite time in the case $d_1 = d_2$ and $c_1 = c_2$. We employ similar methods as in~\cite{ESCHER} to prove Theorem~\ref{mainresult3} for general velocities $c_1,c_2 \in \R$ and diffusion coefficients $d_1,d_2 > 0$. Thus, we iteratively feed pointwise lower bounds on the solution into Duhamel's formula and estimate from below using Jensen's inequality. These lower bounds are no longer diffusive Gaussians as in~\cite{ESCHER} but of the type~\eqref{drag}.

\subsection{Admissible relevant and marginal nonlinearities}

All in all, in the case of different velocities $c_1 \neq c_2$, nonlinearities consisting of mix-terms and irrelevant terms in~\eqref{GRD} do not alter the temporal decay dictated by the linear system~\eqref{GRDLIN} for solutions with small, localized initial conditions; thus, yielding nonlinear stability of the rest state $(u,v) = (0,0)$ in~\eqref{GRD}, see Theorem~\ref{mainresult2}. On the other hand, relevant or marginal nonlinearities in~\eqref{GRD} that are not of mix-type, even if they are couplings between the $u$- and $v$-component, can lead to growth of small, localized initial data; thus, yielding instability of the rest state $(u,v) = (0,0)$, see Theorem~\ref{mainresult3} but also Remark~\ref{scalar}.

The question which relevant or marginal nonlinear terms yield stability and which can lead to instability of the rest state in~\eqref{GRD} is rather subtle. In the scalar setting, it suffices to determine the leading-order coefficient of the power series expansion of the nonlinearity, see Remark~\ref{scalar}. On the contrary, we demonstrate that, in the case of multiple components, it is insufficient to look at the leading-order (or most dangerous) term in the nonlinearity; in general, \emph{other} (higher-order) terms need to be taken into account. As an example, we consider the system
\begin{align}
\begin{split}
u_t &= d_1u_{xx} + c_1 u_x + \alpha uv + \beta u^3,\\
v_t &= d_2v_{xx} + c_2 v_x + \gamma \left( u^2\right)_x,
\end{split} \qquad t \geq 0, x \in \R, \label{CAS3}
\end{align}
with coefficients $\alpha,\beta, \gamma \in \R$ and $c_1 \neq c_2$. Intuitively, one expects that the mix-term $\alpha uv$ is harmless, since the components propagate with different velocities. One is then inclined to neglect this seemingly unimportant mix-term such that the $u$-component in~\eqref{CAS3} decouples. We arrive at the scalar system $u_t =  d_1u_{xx} + c_1 u_x + \beta u^3$ in which the sign of $\beta$ determines whether solutions exists globally and decay over time or blow up in finite time, see Remark~\ref{scalar}. Yet, the following results shows this formal way of reasoning leads to incorrect conclusions.

\begin{theorem} \label{mainresult4}
Consider for $M \geq 1$ the exponential weight $\rho_E(x) = e^{x^2/M}$. Suppose $d_i > 0$, $c_1 \neq c_2$ and the sign condition
\begin{align}\beta - \frac{\gamma \alpha}{c_2 - c_1} < 0,\label{expreta}\end{align}
holds for the coefficients $\alpha,\beta,\gamma,c_i,d_i \in \R$ in~\eqref{CAS3}. Then, there exists $M_0 \geq 1$ such that, for all $M \geq M_0$ and $\epsilon > 0$, there exists a $\delta > 0$ such that~\eqref{GRD} together with the initial condition $(u_0,v_0) \in H^2(\R,\R^2)$ satisfying
\begin{align} \left\|(u_0,v_0)\right\|_{H^2} + \left\|v_0\rho_E\right\|_\infty + \sqrt{M\pi}\ \left\|u_0\rho_E\right\|_\infty + \left\|u_0' \rho_E\right\|_\infty < \delta, \label{initcond4}\end{align}
has a classical global solution $(u,v) \in C^1\left([0,\infty),C^2(\R,\R^2)\right)$ satisfying
\begin{align} \|(u, v)(\cdot,t)\|_\infty  \leq \frac{\epsilon}{\sqrt{1 + t}}, \qquad \|(u, v)(\cdot,t)\|_1 \leq \epsilon, \qquad t \geq 0. \label{temporaldec3}\end{align}
\end{theorem}

Thus, the seemingly unimportant mix-term $\alpha uv$ and Burgers' term $\gamma (u^2)_x$ cannot be neglected. In fact, these terms can compensate for the cubic term $\beta u^3$. Indeed, in contrast to the scalar setting, we can establish global existence and decay of small, localized initial data even for $\beta > 0$.

In the proof of Theorem~\ref{mainresult4} we first apply a normal form transform to remove the Burgers' term $\gamma (u^2)_x$ from the $v$-equation in~\eqref{CAS3}. The normal form transform introduces higher order terms to the $v$-equation which can be controlled via damping estimates by the $L^2$-norm of the solution and the $H^2$-norm of the initial data. In addition, we employ pointwise estimates to control the mix-term and we use a decomposition in Fourier space into short and long-wavelength modes to deal with the cubic term. Again, we note that our proof yields more detailed (pointwise) upper bounds than the ones provided in~\eqref{temporaldec3}; we refer the interested reader to~\S\ref{sec:proofMR4}.

It is worthwhile to emphasize that the occurrence of the difference $c_2 - c_1$ in~\eqref{expreta} underlines the fact that velocities can be decisive for the long-term dynamics. Furthermore, the expression~\eqref{expreta} suggests that, in the presence of quadratic mix-terms, a \emph{coupling} of Burgers'-type might affect the temporal decay of small, localized initial data even if the `dangerous' cubic term $\beta u^3$ is absent, whereas, in the scalar setting, a Burgers' term does not alter the temporal decay (only the limiting profile), as described in~\cite{BKL}.

In this paper, system~\eqref{CAS3} serves as an example to illustrate that the \emph{full} nonlinearity can be decisive for the global behavior of small solutions, instead of just the leading-order (or `most dangerous') term of the nonlinearity. However, its importance goes beyond the current setting: we expect that system~\eqref{CAS3} governs the critical dynamics in the stability analysis of wave trains at the Eckhaus boundary, see~\S\ref{sec:discussion}.

\subsection{Set-up}

This paper is structured as follows. First we establish local existence and uniqueness of solutions to~\eqref{GRD} in~\S\ref{sec:local}. In~\S\ref{sec:ill} we introduce the method of pointwise estimates in a simple setting and we illustrate how this method is employed to exploit differences in velocities in the proofs of our main results. The proofs of our main results, Theorems~\ref{mainresult1} and~\ref{mainresult2}, are then provided in~\S\ref{sec:proofMIX} and~\S\ref{sec:proofMR2}, respectively. Subsequently, the proof of Theorem~\ref{mainresult3}, showing growth of small initial data in system~\eqref{CAS2}, can be found in~\S\ref{sec:proofgrowth}, whereas the analysis of system~\eqref{CAS3} and the proof of Theorem~\ref{mainresult4} are the contents of~\S\ref{sec:proofMR4}. Finally, we comment on possible applications and future extensions of our work in~\S\ref{sec:discussion}.

\subsection{Other methods to prove global existence of small solutions in reaction-diffusion systems} \label{overview}

The problem of global existence of (small) solutions to nonlinear evolution equations is classical and various techniques have been developed over the past decades to address this issue. For reaction-diffusion systems on unbounded domains the first results~\cite{FUJI,FUJI2,WEISS} are obtained in the setting of the nonlinear heat equation $u_t = u_{xx} + f(u)$, where $f \colon \R \to \R$ is an irrelevant nonlinearity, typically $f(u) = u^p$. In these early papers global existence is established by iterative estimates of the linear and nonlinear terms in Duhamel's formula. In~\cite{FUJI,FUJI2} such iterative estimates are obtained using time-dependent Gaussian weights, whereas in~\cite{WEISS} the estimates are established in $L^q$-spaces. In essence, the method of pointwise estimates, used in this paper, also relies on spatio-temporal weights and the analysis in~\cite{FUJI,FUJI2} can be regarded as an early version.

It was observed in~\cite{SCHONB} that global existence and temporal decay can be established using entropy inequalities in systems of conservation laws, i.e.~in reaction-diffusion systems with nonlinearities of the form $(f(u))_x$. We emphasize that marginal Burgers'-type terms of the form $(u^2)_x$ can be included in the nonlinearity. Nowadays, there is a vast literature, see for instance~\cite{DESV,PIER} and references therein, on global existence in reaction-diffusion systems in which entropy or mass can be controlled, which sometimes even leads to a gradient structure~\cite{MIELG}.

Global existence of small solutions in general reaction-diffusion systems with irrelevant nonlinearities $f(u,u_x,u_{xx})$ was first obtained in~\cite{KLAI1} using energy estimates and the abstract Nash-Moser-H\"ormander iteration scheme. This result was improved in~\cite{PONC,ZHENG} by combining the $L^2$-energy estimates with weighted a priori estimates in $L^q$-space. The assumptions on the (irrelevant) nonlinearities were further relaxed in~\cite{CUIS}.

The first global existence results with relevant nonlinearities have been developed in the setting of the nonlinear heat equation with absorption $u_t = u_{xx} - u^p$ with $p > 1$. Using comparison principles one readily obtains that solutions decay over time~\cite{GMIR}. In fact, such comparison or maximum principles can be used to prove that solutions converge in the long time limit to self-similar solutions~\cite{GALAK,GMIR,KAMPEL}.

The renormalization group (RG) method~\cite{BK92,BKL} transforms the long time limit to a fixed time problem through an iterative scaling process, which reveals the large time (self-similar) asymptotics of solutions. In contrast to comparison principles, an extension of the RG method to \emph{systems} of reaction-diffusion equations is natural, see for instance~\cite{BK92,BKX}. In addition, the RG method has been successfully applied in multiple spatial dimensions~\cite{UEC9} and to reaction-diffusion systems with relevant or marginal nonlinear terms~\cite{BKL}. Since the renormalization procedure takes the spatio-temporal dynamics of solutions into account, we expect that, instead of the method of pointwise estimates, the RG method could also have been used in this paper to exploit the difference in velocities in~\eqref{GRD}. However, since our main interest is global existence and temporal decay of small solutions, instead of their large time asymptotics, we refrain from using (the technically more involved) RG method.

\section{Local existence and uniqueness} \label{sec:local}

Local existence and uniqueness of solutions in semilinear parabolic problems with initial data in spaces of bounded, H\"older continuous functions is well-established, see for instance~\cite{LUN}. Local existence and uniqueness has been verified for the current case of reaction-diffusion-advection systems on the real line in~\cite[Section 11.3]{ZUH} using the so-called parametrix method. In the following proposition we collect the facts which are needed to prove the main results of this paper.

\begin{proposition} \label{proplocal}
Let $\alpha \in (0,1)$. Suppose the coefficients in~\eqref{GRD} satisfy $d_i > 0$, $c_i \in \R$ and $f_i,g_i \in C^{2,\alpha}(\R^2,\R)$ with $f_i(0), g_i(0) = 0$ for $i = 1,2$. Let $(u_0,v_0) \in L^\infty(\R,\R^2) \cap C^{0,\alpha}(\R,\R^2)$. Then, there exists a classical solution $(u,v) \in C^{1,\frac{\alpha}{2}}\left([0,T),C^{2,\alpha}(\R,\R^2) \cap L^\infty(\R,\R^2)\right)$ to~\eqref{GRD} on a maximal time interval $[0,T)$, with $T \in (0,\infty]$, having initial condition $(u_0,v_0)$. If $T < \infty$, then it holds
\begin{align}\lim_{t \uparrow T} \|(u,v)(\cdot,t)\|_\infty = \infty.\label{blowuplinfty}\end{align}
\end{proposition}
\begin{proof}
By~\cite[Corollary 11.4]{ZUH} there exists a solution $(u,v) \in C^{1,\frac{\alpha}{2}}\left([0,T),C^{2,\alpha}(\R,\R^2) \cap L^\infty(\R,\R^2)\right)$ to~\eqref{GRD} on a maximal interval $[0,T)$, with $T \in (0,\infty]$, having initial condition $(u_0,v_0)$. Assume by contradiction that $T < \infty$ and $t \mapsto \|(u,v)(\cdot,t)\|_\infty$ is bounded on $[0,T)$. We observe that $z(x,t) = (u,v)(x,t)$ satisfies the linear system,
\begin{align}
z_t = Dz_{xx} + (G(x,t) z)_x + F(x,t) z, \label{linsys}
\end{align}
with $D := \text{diag}(d_1,d_2)$ and use $f_i(0), g_i(0) = 0$ for $i = 1,2$ to write
\begin{align*} G(x,t) &:= \begin{pmatrix} c_1 & 0 \\ 0 & c_2\end{pmatrix} + \int_0^1 \begin{pmatrix} \partial_u g_1 \!&\! \partial_v g_1 \\ \partial_u g_2 \!&\! \partial_v g_2 \end{pmatrix}[\gamma z(x,t)] \de\gamma,\qquad F(x,t):= \int_0^1 \begin{pmatrix} \partial_u f_1 \!&\! \partial_v f_1 \\ \partial_u f_2 \!&\! \partial_v f_2 \end{pmatrix}[\gamma z(x,t)] \de\gamma.
\end{align*}
By~\cite[Proposition 11.3]{ZUH} the Green's function $G(x,y,t,s)$ associated to~\eqref{linsys} is continuous and differentiable with respect to $x$. Moreover, it enjoys the estimate
\begin{align*} \left\|D_x^j G(x,y,t,s)\right\| \leq Ct^{-\tfrac{j+1}{2}} e^{-\tfrac{(x-y)^2}{N(t-s)}}, \qquad x,y \in \R, T_0 \leq s \leq t < T, j = 0,1.\end{align*}
for some $C,N > 0$ and $T_0 \in (0,T)$. This estimate and the fact that $t \mapsto \|z(\cdot,t)\|_\infty$ is bounded on $[0,T)$ imply that
\begin{align*} z(x,t) = \int_\R G(x,y,t,T_0) z(y,T_0)\de y, \end{align*}
can be extended from $\R \times [T_0,T)$ to $\R \times [T_0,T]$ such that $z(\cdot,T)$ is bounded and continuously differentiable on $\R$. In particular, $z(\cdot,T)$ lies in $L^\infty(\R,\R^2) \cap C^{0,\alpha}(\R,\R^2)$ and can therefore be extended by~\cite[Corollary 11.4]{ZUH} to a solution $z(\cdot,t)$ in $L^\infty(\R,\R^2) \cap C^{2,\alpha}(\R,\R^2)$ on some interval $[0,T+\tau)$ with $\tau > 0$, which contradicts the maximality of $T$. Thus, the blow-up~\eqref{blowuplinfty} must hold if $T < \infty$.
\end{proof}

\section{Illustration of the main ideas} \label{sec:ill}

In this section we give a short introduction to the method of pointwise estimates and we explain how this method can be applied to exploit differences in velocities. In particular, we aim to illustrate in a simple setting how the estimates in the proofs of our main results arise. Thus, we consider the toy model
\begin{align}
\begin{split}
u_t &= d_1u_{xx} + c_1 u_x + u^4 + uv,\\
v_t &= d_2v_{xx} + c_2 v_x,
\end{split} \qquad t \geq 0, x \in \R, \label{TOY}
\end{align}
with $d_i > 0$ and $c_i \in \R$ with $c_1 \neq c_2$. We take small, exponentially localized initial data, i.e.~we take $u_0,v_0 \in L^\infty(\R)$ satisfying
\begin{align}\left|u_0(x)e^{\tfrac{x^2}{M}}\right|, \left|v_0(x)e^{\tfrac{x^2}{M}}\right| \leq \delta, \qquad x \in \R, \label{initb}\end{align}
for some $M > 0$, where $\delta > 0$ is sufficiently small. We assume local existence and uniqueness of a classical solution $(u(t),v(t))$ to~\eqref{TOY} with initial condition $(u_0,v_0)$.

We want to derive bounds on the solution $(u(t),v(t))$ yielding global existence and decay. Such bounds follow from iterative estimates using the Duhamel formulation (or variation of constants formula). Thus, integrating the $u$-equation in~\eqref{TOY} gives
\begin{align}
u(t) &= e^{\El t} u_0 + \int_0^t e^{\El (t-s)} \N(u(s),v(s))\de s, \label{Duhamelclas}
\end{align}
where $\El$ is the differential operator $\El = d_1\partial_{xx} + c_1 \partial_x$ and $\N$ denotes the nonlinearity $\N(u,v) = u^4 + uv$. With the aid of the Fourier transform one finds
\begin{align}
\left[e^{\El t} u_0\right](x) = \int_\R \frac{e^{-\tfrac{(x-y+c_1t)^2}{4d_1t}}}{\sqrt{4\pi d_1t}} u_0(x) \de y. \label{Elop}
\end{align}

To bound the linear and nonlinear terms in~\eqref{Duhamelclas} one can first use~\eqref{Elop} to compute (or bound) the operator norm of $e^{\El t} \colon X \to Y$, where $X$ and $Y$ are suitable function spaces, and then estimate $u(t)$ in the $Y$-norm via~\eqref{Duhamelclas} using bounds on $u_0$ and $\N(u(s),v(s))$ in the $X$-norm. Thus, one obtains
\begin{align}
\|u(t)\|_Y \leq \|e^{\El t}\|_{X \to Y} \|u_0\|_X + \int_0^t \|e^{\El(t-s)}\|_{X \to Y} \|\N(u(s),v(s))\|_X \de s. \label{opnorm}
\end{align}
For the nonlinear heat equation $u_t = u_{xx} + u^p$ with $p > 3$ suitable spaces are for instance $L^1(\R)$ and $L^\infty(\R)$, see~\cite{MSU} and~\cite[Section 14]{SUbook}. We note that the renormalization group method and the approaches using bounds in $L^q$-spaces, which are mentioned in~\S\ref{overview}, are based on estimates of the form~\eqref{opnorm}.

On the contrary, the method of pointwise estimates does not rely on the operator norm of $e^{\El t}$ between suitable function spaces. Instead, the functions $e^{\El t} u_0$ and $e^{\El (t-s)} \N(u(s),v(s))$ in~\eqref{Duhamelclas} are estimated \emph{pointwise} using~\eqref{Elop}. In the case of exponential weights, all these estimates utilize the integral identity
\begin{align}
\int_\R e^{-ay^2 + by + c} \de y = \frac{\sqrt{\pi} e^{\tfrac{b^2 + 4ac}{4a}}}{\sqrt{a}}, \label{integralid1}
\end{align}
with $a > 0$ and $b, c \in \R$, which follows from the standard Gaussian integral by completing the square.

To bound the linear term $e^{\El t} u_0$ in~\eqref{Duhamelclas}, we take $M \geq 4d_1$, plug the bound~\eqref{initb} into~\eqref{Elop} and use~\eqref{integralid1} to yield
\begin{align}
\left|\left[e^{\El t} u_0\right](x)\right| \leq \delta \int_{\R} \frac{e^{-\tfrac{(x-y+c_1t)^2}{Mt}-\tfrac{y^2}{M}}}{\sqrt{4 \pi d_1 t}}\de y
= \frac{\delta \sqrt{M} e^{-\tfrac{(x+c_1t)^2}{M (1+t)}}}{2 \sqrt{d_1 (1+t)}}. \label{linboundTOY}
\end{align}
Thus, without the presence of the nonlinear terms in~\eqref{TOY}, solutions with exponentially localized initial conditions decay as
\begin{align} |u(x,t)| \leq C \frac{e^{-\tfrac{(x+c_1t)^2}{M (1+t)}}}{\sqrt{1+t}}, \qquad |v(x,t)| \leq C \frac{e^{-\tfrac{(x+c_2t)^2}{M (1+t)}}}{\sqrt{1+t}}, \label{assump} \end{align}
where $C>0$ is some $x$- and $t$-independent constant.

The method of pointwise estimates now employs the bounds~\eqref{assump} as spatio-temporal weights in the nonlinear iteration. Thus, one assumes that~\eqref{assump} holds and obtains
\begin{align*}\left|\N(u(y,s),v(y,s))\right| \leq C\left(\frac{e^{-\tfrac{(y+c_1s)^2}{M (1+s)}}}{(1+s)^2} + \frac{e^{-\tfrac{(y+c_1s)^2}{M (1+s)} - \tfrac{(y+c_2s)^2}{M (1+s)}}}{1+s}\right),\end{align*}
where $C> 0$ is some $y$- and $s$-independent constant. We take $M \geq 8d_1$, use~\eqref{Elop} to plug the latter estimate into $e^{\El (t-s)} \N(u(s),v(s))$ and establish
\begin{align*}
\left|\left[e^{\El (t-s)} \N(u(s),v(s)) \right](x)\right| \leq C \int_{\R} \left(\frac{e^{-\tfrac{(x-y+c_1(t-s))^2}{M(t-s)} -\tfrac{(y+c_1s)^2}{M (1+s)}}}{\sqrt{4\pi d_1(t-s)}(1+s)^2} + \frac{e^{-2\tfrac{(x-y+c_1(t-s))^2}{M(t-s)}-\tfrac{(y+c_1s)^2}{M (1+s)} - \tfrac{(y+c_2s)^2}{M (1+s)}}}{\sqrt{4\pi d_1(t-s)}(1+s)}\right) \de y.
\end{align*}
Although the obtained integrals over $y$ look complicated at first sight, they are of the form~\eqref{integralid1} with constants $a,b$ and $c$ depending on $s,x$ and $t$, see Remark~\ref{mathematica}. Hence, using~\eqref{integralid1}, we compute
\begin{align}
\int_{\R} \frac{e^{-\tfrac{(x-y+c_1(t-s))^2}{M(t-s)} - \tfrac{(y + c_1s)^2}{M (1+s)}}}{\sqrt{4\pi d_1(t-s)}(1+s)^2} \de y &= \frac{\sqrt{M} e^{-\tfrac{(x+c_1t)^2}{M (1+t)}}}{\sqrt{d_1(1+t)}(1+s)^{3/2}}, \label{eval1}
\end{align}
and
\begin{align}
\int_{\R} \frac{e^{-2\tfrac{(x-y+c_1(t-s))^2}{M(t-s)}-\tfrac{(y+c_1s)^2}{M (1+s)} - \tfrac{(y+c_2s)^2}{M (1+s)}}}{\sqrt{4\pi d_1(t-s)}(1+s)} \de y = \frac{\sqrt{M} e^{- \tfrac{(x+c_1t)^2}{M(1+t)} - \tfrac{s^2 (t-s) (c_1-c_2)^2}{2 M (1+s) (1+t)} - \tfrac{(x + c_1t + s(c_2-c_1))^2}{M (1+t)}}}{2\sqrt{2 d_1(1+t)(1+s)}}. \label{latter}
\end{align}
We emphasize that, in order to obtain the exponentially decaying term $e^{- \frac{s^2 (t-s) (c_1-c_2)^2}{2 M (1+s) (1+t)}}$ on the right hand side of~\eqref{latter}, taking $M \geq 8d_1$, instead of $M \geq 4d_1$, is crucial. Indeed, one readily verifies that, without the $2$ before the coefficient $\tfrac{(x-y+c_1(t-s))^2}{M(t-s)}$ in~\eqref{latter}, such exponential decay in time is not achieved uniformly in space. Intuitively, this $2$ balances the \emph{two} exponentials coming from the pointwise bound on the mix-term $uv$, whereas in~\eqref{eval1} only one exponential arises when bounding the $u^4$-term.

To close the nonlinear iteration scheme and apply continuous induction, we want to recover our original assumption~\eqref{assump} after one iteration of Duhamel's formula. Thus, in addition to the linear term $e^{\El t} u_0$, we also want to bound the nonlinear term $\int_0^t e^{\El (t-s)} \N(u(s),v(s)) \de s$ in~\eqref{Duhamelclas} by an $x$- and $t$-independent multiple of the drifting Gaussian
\begin{align} \frac{e^{-\tfrac{(x+c_1t)^2}{M (1+t)}}}{\sqrt{1+t}}. \label{weight}\end{align}
Integrating~\eqref{eval1} over $s$ clearly yields such a bound, since the integral $\int_0^t (1+s)^{-3/2} \de s$ is uniformly bounded for $t \geq 0$. On the other hand, exploiting the difference in velocities, we establish
\begin{align*}
\begin{split}
\int_0^t \frac{e^{- \tfrac{s^2 (t-s) (c_1-c_2)^2}{2 M (1+s) (1+t)}}}{\sqrt{1+s}} \de s  &\leq C\left(\int_0^1 \de s + \int_{1}^{\frac{t}{2}} \frac{\sqrt{1+s}(1+t)}{s^2 (t-s)} \de s + \int_{\frac{t}{2}}^{t} \frac{e^{-\tfrac{t^2 (t - s) (c_1 - c_2)^2}{8 M (1 + t)^2}}}{\sqrt{1+t}} \de s\right),
\end{split}
\end{align*}
for some $t$-independent constant $C > 0$, where the right hand side is uniformly bounded for $t \geq 2$. It is not hard to see that the left hand side is also bounded for $t \in [0,2]$. So, we obtain
\begin{align*} \left|\int_0^t \left[e^{\El (t-s)} \N(u(s),v(s))\right](x) \de s\right| \leq C \frac{e^{-\tfrac{(x+c_1s)^2}{M (1+t)}}}{\sqrt{1+t}},
\end{align*}
for some $x$- and $t$-independent constant $C> 0$. Combining the latter with~\eqref{linboundTOY}, we have recovered our original assumption~\eqref{assump} after an iteration with Duhamel's formula. Hence, we can expect to close the nonlinear iteration scheme and apply continuous induction to establish global existence and Gaussian-like decay. We will make the latter rigorous in the proofs of our main results in the upcoming sections.

\begin{remark}{\upshape \label{remspatiotemp}
In the case of polynomially localized initial conditions, one obtains a different pointwise bound than~\eqref{linboundTOY} on the linear term $e^{\El t} u_0$ in Duhamel's formula. Therefore, also the chosen spatio-temporal weight is different for polynomially localized initial data, see~\S\ref{estexplinM} in the proof of Theorem~\ref{mainresult1}.

However, it is not only the pointwise bound on the linear term $e^{\El t} u_0$ that determines the spatio-temporal weight. For instance, the presence of a $v^4$-term in the $u$-component of~\eqref{TOY} would lead to the contribution
\begin{align*}
\int_{\R} \frac{e^{-\tfrac{(x-y+c_1(t-s))^2}{M(t-s)} - \tfrac{(y + c_2s)^2}{M (1+s)}}}{\sqrt{4\pi d_1(t-s)}(1+s)^2} \de y = \frac{\sqrt{M} e^{-\tfrac{(x+c_1t + s(c_2-c_1))^2}{M (1+t)}}}{\sqrt{d_1(1+t)}(1+s)^{3/2}},
\end{align*}
in $\left[e^{\El (t-s)}\N(u(s),v(s))\right](x)$. Due to this contribution, the integral $\int_0^t e^{\El (t-s)}\N(u(s),v(s)) \de s$ in~\eqref{Duhamelclas} can no longer be bounded by an $x$- and $t$-independent multiple of a drifting Gaussian~\eqref{weight}, see Figure~\ref{decaytypes}. Consequently, one cannot close the nonlinear iteration scheme and needs to adapt the spatio-temporal weight. Thus, our choice of spatio-temporal weight is inspired by the pointwise bounds on both the linear terms and the nonlinear terms. For instance, the presence of nonlinear coupling terms, which are not of mix-type, leads, in the proof of Theorem~\ref{mainresult2}, to an additional contribution in the chosen spatio-temporal weight.
}\end{remark}

\begin{remark}{\upshape \label{mathematica}
In the upcoming proofs, the pointwise estimation of $e^{\El (t-s)} \N(u(s),v(s))$ in~\eqref{Duhamelclas} relies, as in the above, on the integral identity~\eqref{integralid1} in the case of exponentially localized initial data. As additional contributions are introduced to the spatio-temporal weight in the proof of Theorem~\ref{mainresult2}, the expressions for $a,b$ and $c$ in~\eqref{integralid1} become more involved. However, since the integral identity~\eqref{integralid1} is standard, our calculations can be (and have been) verified using a symbolic computation program such as \textsf{Mathematica}.
}\end{remark}

\section{Proof of Theorem~\ref{mainresult1}} \label{sec:proofMIX}

In this proof, $C>1$ denotes a constant, which is independent of $\delta,x$ and $t$ and that will be taken larger if necessary.

\subsection{Plan of proof} \label{sec:planM}

Take $M_0 = \max\{16d_1,16d_2,1\}$ and let $M \geq M_0$ and $r \geq 3$. By Proposition~\ref{proplocal} there exists a classical solution $(u,v) \in C^{1,\frac{\alpha}{2}}\left([0,T),C^{2,\alpha}(\R,\R^2)\right)$ to~\eqref{GRD} on a maximal time interval $[0,T)$, with $T \in (0,\infty]$, having initial condition $(u_0,v_0) \in X_{\rho_i}^\alpha \subset C^{0,\alpha}(\R,\R^2) \cap L^\infty(\R,\R^2)$ for $i = E,A$. Therefore, the functions $\eta_E, \eta_A \colon [0,T) \to [0,\infty)$ given by
\begin{align*}
 \eta_E(t) = \sup_{\begin{smallmatrix} x \in \R \\ 0 \leq s \leq t \end{smallmatrix}} &\sqrt{1+s} \left(\left|u(x,s)\right|e^{\tfrac{(x+c_1s)^2}{M(1+s)}} + \left|v(x,s)\right|e^{\tfrac{(x+c_2s)^2}{M(1+s)}}\right),
 \end{align*}
and
\begin{align*}
 \eta_A(t) = \sup_{\begin{smallmatrix} x \in \R \\ 0 \leq s \leq t \end{smallmatrix}} &\left(\left|u(x,s)\right| \left[\frac{1}{\left(1 + |x+c_1s| + \sqrt{s}\right)^{r}} + \frac{e^{-\tfrac{(x+c_1s)^2}{M(1+s)}}}{\sqrt{1+s}}\right]^{-1} \right.\\
& \left. \qquad \qquad\qquad\qquad\qquad + \left|v(x,s)\right| \left[\frac{1}{\left(1 + |x+c_2s| + \sqrt{s}\right)^{r}} + \frac{e^{-\tfrac{(x+c_2s)^2}{M(1+s)}}}{\sqrt{1+s}}\right]^{-1}\right),
\end{align*}
are continuous. In addition, if $T < \infty$, then $\eta_A(t)$ and $\eta_E(t)$ must blow up as $t \uparrow T$ by~\eqref{blowuplinfty}. Our aim is to show that, if we have $\|(u_0,v_0)\|_{\rho_i} < \delta$ and $t \in [0,T)$ is such that $\eta_i(t)$ is bounded by the constant $r_0$ in Theorem~\ref{mainresult1}, then $\eta_i(t)$ satisfies an inequality of the form
\begin{align} \eta_i(t) \leq C\left(\delta + \eta_i(t)^2\right), \quad i = E,A. \label{etaest}\end{align}
Since $\eta_i$ must be continuous as long as it remains bounded, we can apply continuous induction using~\eqref{etaest}. Thus, taking $\delta < \min\{\frac{1}{4C^2},\frac{r_0}{2C}\}$, it follows $\eta_i(t) \leq 2C\delta \leq r_0$ for \emph{all} $t \geq 0$, which proves global existence, i.e.~$T = \infty$. Moreover, taking $\delta = \min\{\epsilon/(2C(1+\sqrt{M\pi})),1/(4C^2),r_0/(2C)\}$ it holds
\begin{align}\eta_i(t) \leq 2C\delta = \frac{\epsilon}{1+\sqrt{M\pi}} \leq \epsilon, \label{etaest2} \end{align}
for all $t \geq 0$, which yields the desired temporal decay in~\eqref{temporaldec} in the $L^\infty$-norm. The bound in~\eqref{temporaldec} of the $L^1$-norm follows from~\eqref{etaest2} after integration of the associated spatio-temporal weight:
\begin{align} \int_\R \frac{e^{-\tfrac{x^2}{M(1+s)}}}{\sqrt{1+s}}\de x = \sqrt{M \pi}, \qquad \int_\R \frac{1}{\left(1 + |x+c_1s| + \sqrt{s}\right)^{r}} \de x = \frac{2}{(r-1)\left(1 + \sqrt{s}\right)^{r - 1}} \leq 1, \label{L1integral}\end{align}
for $t \geq 0$. Finally,~\eqref{etaest2} proves the pointwise decay estimates~\eqref{pointdec} in the case of exponentially localized initial data.

Take $z(x,t) = (u,v)(x,t)$ and denote
\begin{align*} \theta(x,t) := \frac{e^{-\tfrac{(x+c_1t)^2}{4d_1t}}}{\sqrt{4\pi d_1t}}.\end{align*}
Integrating the $u$-equation in~\eqref{GRD} and applying integration by parts we obtain
\begin{align}
\begin{split}
u(x,t) &= \int_{\R} \theta(x-y,t) u_0(y)\de y\\ &\qquad + \int_0^t  \int_{\R} \left[\theta(x-y,t-s) f_1(z(y,s)) + \theta_x(x-y,t-s)g_1(z(y,s))\right]\de y\de s,
\end{split} \label{Duhamel}
\end{align}
for $x \in \R$ and $t \in [0,T)$. A similar integral formulation can be obtained for the $v$-component in~\eqref{GRD}. Our plan is to prove the key inequality~\eqref{etaest}, provided $\eta_i(t) \leq r_0$, by estimating the linear and nonlinear terms in the integral formulations for the $u$- and $v$-component.

\begin{remark} \label{spatiotemporalbounds1}
{\upshape
For algebraically localized initial data $(u_0,v_0) \in X_{\rho_A}^\alpha$ satisfying $\|(u_0,v_0)\|_{\rho_A} < \delta$,~\eqref{etaest2} yields
\begin{align*} \left|u(x,t)\right| &\leq \epsilon\left[\frac{1}{\left(1 + |x+c_1t| + \sqrt{t}\right)^{r}} + \frac{e^{-\tfrac{(x+c_1t)^2}{M(1+t)}}}{\sqrt{1+t}}\right], \quad \left|v(x,t)\right| \leq \epsilon\left[\frac{1}{\left(1 + |x+c_2t| + \sqrt{t}\right)^{r}} + \frac{e^{-\tfrac{(x+c_2t)^2}{M(1+t)}}}{\sqrt{1+t}}\right],
\end{align*}
for all $x \in \R$ and $t \geq 0$. Hence, the spatio-temporal decay of the solutions to~\eqref{GRD} is, as in~\eqref{pointdec}, componentwise controlled by drifting Gaussians, which exhibit the decay as predicted by the linear system, and by traveling algebraically localized correction terms, which exhibit faster temporal decay.
}\end{remark}

\subsection{Linear estimates} \label{estexplinM}

Using $M \geq 4d_1$ and taking exponentially localized initial conditions $(u_0,v_0) \in X_{\rho_E}^\alpha$ the linear term in~\eqref{Duhamel} enjoys the bound
\begin{align}
\left|\int_{\R}  \theta(x-y,t) u_0(y)\de y\right| &\leq C\|u_0\|_{\rho_E} \int_{\R} \frac{e^{-\tfrac{(x-y+c_1t)^2}{Mt}-\tfrac{y^2}{M}}}{\sqrt{t}}\de y \leq  C \delta \frac{e^{-\tfrac{(x+c_1 t)^2}{M(1+t)}}}{\sqrt{1+t}}.\label{linexp} \end{align}
For algebraically localized initial conditions $(u_0,v_0) \in X_{\rho_A}^\alpha$ we apply~\cite[Corollary 6.16]{JUN} and obtain
\begin{align}
\left|\int_{\R}  \theta(x-y,t) u_0(y)\de y\right| &\leq C\|u_0\|_{\rho_A} \int_{\R} \frac{e^{-\tfrac{(x-y+c_1t)^2}{Mt}}}{(1+|y|)^r \sqrt{t}}\de y \leq  C \delta \left[\frac{1}{\left(1 + |x+c_1t| + \sqrt{t}\right)^{r}} + \frac{e^{-\tfrac{(x+c_1t)^2}{M(1+t)}}}{\sqrt{1+t}}\right]. \label{linalg}\end{align}

\subsection{Nonlinear estimates for exponential weights} \label{estexpM}

Provided $\eta_E(t) \leq r_0$, the nonlinear term in~\eqref{Duhamel} can be estimated, using~\eqref{nonlinearbounds1}, $M \geq 8d_1$, $\eta_E(t) \leq r_0$ and the boundedness of $x \mapsto xe^{-x^2}$ on $\R$, as follows
\begin{align}
\begin{split}
&\left|\int_0^t  \int_{\R} \left(\theta(x-y,t-s) f_1(z(y,s)) + \theta_x(x-y,t-s)g_1(z(y,s))\right)\de y\de s\right| \\
&\qquad \leq C\int_0^t\int_\R \left[\theta(x-y,t-s) |u(y,s)|^4 + |\theta_x(x-y,t-s)| |u(y,s)|^2\right.\\
&\qquad \phantom{\leq C\int_0^t\int_\R}\left.\phantom{|u_x(y,s)|^4} + \left(\theta(x-y,t-s) + |\theta_x(x-y,t-s)|\right)|u(y,s)||v(y,s)|\right] \de y \de s\\
&\qquad \leq C\eta_E(t)^2 \left(I_* + I_0 + I_1\right),
\end{split} \label{estM3}
\end{align}
with
\begin{align}
I_* &= \int_0^t\int_{\R} \frac{e^{-\tfrac{(x-y+c_1(t-s))^2}{M(t-s)} - \tfrac{(y+c_1s)^2}{M(1+s)}}}{1+s} \left(\frac{1}{(1+s)\sqrt{t-s}} + \frac{1}{t-s}\right)\de y\de s, \label{defI*}\\
I_j &= \int_0^t\int_{\R} \frac{e^{-\tfrac{2(x-y+c_1(t-s))^2}{M(t-s)} - \tfrac{(y+c_1s)^2}{M(1+s)} - \tfrac{(y+c_2s)^2}{M(1+s)}}}{(1+s)(t-s)^{(1+j)/2}}\de y\de s, \quad j = 0,1. \label{defIj}
\end{align}

\paragraph*{Estimate on $I_*$.} Using~\eqref{integralid1}, we compute the inner integral in $I_*$ and establish
\begin{align}
I_* &\leq C \frac{e^{-\tfrac{(x+c_1t)^2}{M(1+t)}}}{\sqrt{1+t}}\int_0^t \left( \frac{1}{(1+s)^{3/2}} + \frac{1}{\sqrt{1+s}\sqrt{t-s}}\right) \de s \leq C\frac{e^{-\tfrac{(x+c_1t)^2}{M(1+t)}}}{\sqrt{1+t}}. \label{estM2}
\end{align}

\paragraph*{Estimate on $I_j$.} Using~\eqref{integralid1}, we calculate the integral over $y$ in $I_j$ and obtain
\begin{align}
I_j &\leq C \frac{e^{-\tfrac{(x+c_1t)^2}{M(1+t)}}}{\sqrt{1+t}}\int_0^t \frac{e^{-\tfrac{s^2 (t-s) (c_1-c_2)^2}{2 M (1+s) (1+t)}-\tfrac{(x + c_1t + s(c_2-c_1))^2}{M (1+t)}}}{\sqrt{1+s}(t-s)^{j/2}} \de s\leq C\frac{e^{-\tfrac{(x+c_1t)^2}{M(1+t)}}}{\sqrt{1+t}}\int_0^t \frac{e^{-\tfrac{s^2 (t - s) (c_1 - c_2)^2}{2 M (1 + s) (1 + t)}}}{\sqrt{1+s}(t-s)^{j/2}} \de s, \label{estEM1}
\end{align}
for $j = 0, 1$. For $t \geq 2$, we bound the integral in~\eqref{estEM1} using $c_1 \neq c_2$ and the integral identity
\begin{align}
\int_0^\infty \frac{e^{-z^2 a}}{\sqrt{z}}\de z &= \frac{2 \Gamma \left(\frac{5}{4}\right)}{\sqrt[4]{a}}, \qquad a > 0, \label{integralID5}
\end{align}
where $\Gamma$ denotes the Gamma function. Thus, we establish
\begin{align}
\begin{split}
\int_0^t \frac{e^{-\tfrac{s^2 (t - s) (c_1 - c_2)^2}{2 M (1 + s) (1 + t)}}}{\sqrt{1+s}(t-s)^{j/2}} \de s &\leq C\left(\int_0^1 \frac{1}{(t-s)^{j/2}} \de s + \int_{1}^{\frac{t}{2}} \frac{\sqrt{1+s}(1+t)}{s^2 (t-s)^{1+j/2}} \de s\right.\\ &\quad \left. + \int_{\frac{t}{2}}^{t-1} \frac{e^{-\tfrac{t^2 (t - s) (c_1 - c_2)^2}{8 M (1 + t)^2}}}{\sqrt{1+t}(t-s)^{j/2}} \de s + \int_{t-1}^t \frac{1}{\sqrt{1+t}(t-s)^{j/2}} \de s\right) \leq \frac{C}{(1+t)^{j/2}},
\end{split} \label{innerest}
\end{align}
for $j = 0,1$ and $t \geq 2$. On the other hand, for $t \leq 2$ we establish the bound
\begin{align*} \int_0^t \frac{e^{-\tfrac{s^2 (t - s) (c_1 - c_2)^2}{2 M (1 + s) (1 + t)}}}{\sqrt{1+s}(t-s)^{j/2}} \de s \leq \int_0^2 \frac{1}{(t-s)^{j/2}} \de s \leq C.\end{align*}
Thus, we obtain
\begin{align} I_j \leq C\frac{e^{-\tfrac{(x+c_1t)^2}{M(1+t)}}}{(1+t)^{(1+j)/2}}, \qquad j = 0,1. \label{estM4}\end{align}

\paragraph*{Final nonlinear estimate.} Combining~\eqref{estM3},~\eqref{estM2} and~\eqref{estM4} yields the bound
\begin{align}\left|\int_0^t  \int_{\R} \left(\theta(x-y,t-s) f_1(z(y,s)) + \theta_x(x-y,t-s)g_1(z(y,s))\right)\de y\de s\right| \leq C\eta_E(t)^2\frac{e^{-\tfrac{(x+c_1t)^2}{M(1+t)}}}{\sqrt{1+t}},\label{exponentialNL}
\end{align}
on the nonlinear term in~\eqref{Duhamel}.

\subsection{Nonlinear estimates for algebraic weights} \label{sec:nonlalg}

Using~\eqref{nonlinearbounds1}, $M \geq 16d_1$, $r \geq 3$, $\eta_A(t) \leq r_0$ and the fact that $x \mapsto xe^{-x^2}$ is bounded on $\R$, the nonlinear term in~\eqref{Duhamel} enjoys the bound
\begin{align}
\begin{split}
&\left|\int_0^t  \int_{\R} \left(\theta(x-y,t-s) f_1(z(y,s)) + \theta_x(x-y,t-s)g_1(z(y,s))\right)\de y\de s\right| \\
&\qquad \leq C\int_0^t\int_\R \left[\theta(x-y,t-s) |u(y,s)|^4 + |\theta_x(x-y,t-s)| |u(y,s)|^2\right.\\
&\qquad \phantom{\leq C\int_0^t\int_\R}\left.\phantom{|u_x(y,s)|^4} + \left(\theta(x-y,t-s) + |\theta_x(x-y,t-s)|\right)|u(y,s)||v(y,s)|\right] \de y \de s\\
&\qquad \leq C\eta_A(t)^2 \left(I_* + I_0 + I_1 + II_0 + II_1 + III_0 + III_1 + IV_0 + IV_1\right),
\end{split} \label{estM5}
\end{align}
provided $\eta_A(t) \leq r_0$, where $I_*,I_0$ and $I_1$ are defined in~\eqref{defI*} and~\eqref{defIj} and
\begin{align*}
II_j &= \int_0^t\int_{\R} \frac{e^{-\tfrac{4(x-y+c_1(t-s))^2}{M(t-s)}}}{\left(1 + |y+c_1s| + \sqrt{s}\right)^{2r} (t-s)^{(j+1)/2}} \de y\de s,\\
III_j &= \int_0^t \int_{\R} \frac{e^{-\tfrac{4(x-y+c_1(t-s))^2}{M(t-s)}}}{\left(1 + |y+c_1s| + \sqrt{s}\right)^{r} \left(1 + |y+c_2s| + \sqrt{s}\right)^{r} (t-s)^{(j+1)/2}} \de y\de s,\\
IV_j &= \int_0^t \int_{\R} \frac{e^{-\tfrac{4(x-y+c_1(t-s))^2}{M(t-s)} - \tfrac{(y+c_2s)^2}{M(1+s)}}}{\sqrt{1+s} \left(1 + |y+c_1s| + \sqrt{s}\right)^{r}(t-s)^{(j+1)/2}} \de y\de s,
\end{align*}
with $j = 0,1$. Estimates on $I_*$, $I_j$ have already been obtained in~\eqref{estM2} and~\eqref{estM4}, respectively.

\paragraph*{Estimates on $II_j$ and $III_j$.} We estimate $III_j$ for arbitrary $c_1,c_2 \in \R$, which also provides an estimate on $II_j$ by taking $c_1=c_2$. For $j = 0,1$, we split the integral $III_j$ as follows
\begin{align*} III_j &= \int_0^{\tfrac{t}{2}} \int_{|y+c_1s| \geq \frac{|x+c_1t|}{2}} \ldots \de y \de s + \int_0^{t} \int_{|y+c_1s| \leq \frac{|x+c_1t|}{2}} \ldots \de y \de s + \int_{\tfrac{t}{2}}^t \int_{|y+c_1s| \geq \frac{|x+c_1t|}{2}} \ldots \de y \de s\\ 
&= \widetilde{I}_j + \widetilde{II}_j + \widetilde{III}_j, \end{align*}
and estimate the terms separately. Thus, using $r \geq 3$, we obtain
\begin{align*}
\widetilde{I}_j &\leq \frac{C}{\sqrt{t}} \int_0^{\tfrac{t}{2}} \int_{|y+c_1s| \geq \frac{|x+c_1t|}{2}} \frac{1}{\left(1+|y+c_1s|\right)^r \left(1+\sqrt{s}\right)^{r}(t-s)^{j/2}} \de y \de s \\
&\qquad \qquad \qquad \qquad \leq \frac{C}{\sqrt{t}} \int_0^{\tfrac{t}{2}} \frac{1}{\left(1+\sqrt{s}\right)^{r}(t-s)^{j/2}} \int_{\R}\frac{1}{\left(1+|y+c_1s|\right)^r} \de y\de s \leq \frac{C}{\sqrt{t}},
\end{align*}
but also
\begin{align*}
\widetilde{I}_j &\leq \frac{C}{\left(1 + |x+c_1t|\right)^{r}} \int_0^{\tfrac{t}{2}} \int_{|y| \geq \frac{|x+c_1t|}{2}} \frac{e^{-\tfrac{(x-y+c_1(t-s))^2}{M(t-s)}}}{\left(1+\sqrt{s}\right)^{r}(t-s)^{(1+j)/2}} \de y \de s\\
&\leq \frac{C}{\left(1 + |x+c_1t|\right)^{r}} \int_0^{\tfrac{t}{2}} \frac{1}{\left(1+\sqrt{s}\right)^{r}(t-s)^{j/2}} \int_{\R} \frac{e^{-\tfrac{(x-y+c_1(t-s))^2}{M(t-s)}}}{\sqrt{t-s}} \de y \de s \leq \frac{C}{\left(1 + |x +c_1t|\right)^{r}}.
\end{align*}
Thus, applying the inequality derived in~\cite[Corollary 6.16]{JUN}, we establish
\begin{align}
 \widetilde{I}_j \leq C \min\left\{\frac{1}{\sqrt{t}}, \frac{1}{\left(1+|x+c_1t|\right)^r}\right\} \leq C \left[\frac{1}{\left(1+|x+c_1t|+\sqrt{t}\right)^{r}} + \frac{e^{-\tfrac{(x+c_1t)^2}{M(1+t)}}}{\sqrt{1+t}}\right], \qquad j = 0,1. \label{JUNeq}
\end{align}
Subsequently, using $r \geq 3$, we estimate
\begin{align*}
&\widetilde{II}_j \leq Ce^{-\tfrac{(x+c_1t)^2}{M(1+t)}} \int_0^t \int_{|y+c_1s| \leq \frac{|x+c_1t|}{2}} \frac{e^{-\tfrac{2(x-y+c_1(t-s))^2}{M(t-s)}}}{\left(1 + |y+c_1s|\right)^{r} \left(1+\sqrt{s}\right)^{r}(t-s)^{(1+j)/2}} \de y \de s\\
 &\leq Ce^{-\tfrac{(x+c_1t)^2}{M(1+t)}} \int_0^t \frac{1}{\left(1+\sqrt{s}\right)^{r}\sqrt{t-s}} \left(\int_{\R} \frac{e^{-\tfrac{2(x-y+c_1(t-s))^2}{M(t-s)}}}{\sqrt{t-s}} \de y + \int_\R \frac{1}{\left(1 + |y+c_1s|\right)^{r}} \de y\right) \de s \leq C \frac{e^{-\tfrac{(x+c_1t)^2}{M(1+t)}}}{\sqrt{1+t}}.
\end{align*}
Finally, using $r \geq 3$, it holds
 \begin{align*}
\widetilde{III}_j &\leq \frac{C}{\left(1 + |x+c_1t| + \sqrt{t}\right)^{r}} \int_{\tfrac{t}{2}}^t \int_{|y+c_1s| \geq \frac{|x+c_1t|}{2}} \frac{e^{-\tfrac{(x-y+c_1(t-s))^2}{M(t-s)}}}{\left(1 + \sqrt{s}\right)^{r}(t-s)^{(1+j)/2}} \de y \de s\\
 &\leq \frac{C}{\left(1 + |x+c_1t| + \sqrt{t}\right)^{r}} \int_{\tfrac{t}{2}}^t \frac{1}{\left(1+\sqrt{s}\right)^{r}(t-s)^{j/2}} \int_{\R} \frac{e^{-\tfrac{(x-y+c_1(t-s))^2}{M(t-s)}}}{\sqrt{t-s}} \de y \de s \leq \frac{C}{\left(1 + |x+c_1t| + \sqrt{t}\right)^{r}}.
 \end{align*}
This concludes the estimation of the integral $III_j$ (and thus $II_j$ by taking $c_1 = c_2$ in the above). We have obtained
\begin{align} II_j + III_j \leq C \left[\frac{1}{\left(1+|x+c_1t|+\sqrt{t}\right)^{r}} + \frac{e^{-\tfrac{(x+c_1t)^2}{M(1+t)}}}{\sqrt{1+t}}\right], \qquad j = 0,1. \label{estL2}\end{align}

\paragraph*{Estimate on $IV_j$.} We continue with estimating the integral $IV_j$, which we again split as follows
\begin{align*} IV_j &= \int_0^{\tfrac{t}{2}} \int_{|y+c_1s| \geq \frac{|x+c_1t|}{2}} \ldots \de y \de s + \int_0^{t} \int_{|y+c_1s| \leq \frac{|x+c_1t|}{2}} \ldots \de y \de s + \int_{\tfrac{t}{2}}^t \int_{|y+c_1s| \geq \frac{|x+c_1t|}{2}} \ldots \de y \de s \\&= \widehat{I}_j + \widehat{II}_j + \widehat{III}_j, \end{align*}
for $j = 0,1$. Using~\eqref{L1integral}, we obtain
\begin{align*}
\widehat{I}_j &\leq \frac{C}{\sqrt{t}} \int_0^{\tfrac{t}{2}} \frac{1}{\sqrt{1+s}(t-s)^{j/2}} \int_{|y+c_1s| \geq \frac{|x+c_1t|}{2}} \frac{1}{\left(1+|y+c_1s|+\sqrt{s}\right)^r } \de y \de s \\&\qquad\qquad\qquad\qquad\qquad\qquad\qquad\qquad \leq \frac{C}{\sqrt{t}} \int_0^{\frac{t}{2}} \frac{1}{\left(1+\sqrt{s}\right)^{r}(t-s)^{j/2}} ds \leq \frac{C}{\sqrt{t}}.
\end{align*}
On the other hand,~\eqref{integralid1} yields
\begin{align}
\int_{\R}  e^{-\tfrac{(x-y+c_1(t-s))^2}{M(t-s)} - \tfrac{(y+c_2s)^2}{M(1+s)}}dy = e^{-\tfrac{(x + c_1t + (c_2-c_1)s)^2}{M (1 + t)}} \sqrt{\frac{\pi M (1 + s) (t-s)}{1 + t}}. \label{integralID3}
\end{align}
for $s \in [0,t]$. So, we also have the estimate
\begin{align*}
\widehat{I}_j &\leq \frac{C}{\left(1 + |x+c_1t|\right)^{r}} \int_0^{\tfrac{t}{2}} \int_{|y+c_1s| \geq \frac{|x+c_1t|}{2}}  \frac{e^{-\tfrac{(x-y+c_1(t-s))^2}{M(t-s)} - \tfrac{(y+c_2s)^2}{M(1+s)}}}{\sqrt{1+s}(t-s)^{(1+j)/2}} \de y \de s\\
&\leq \frac{C}{\left(1 + |x+c_1t|\right)^{r}} \left(\int_\R \frac{e^{-\tfrac{(x+c_1 t + (c_2-c_1)s)^2}{M(1+t)}}}{\sqrt{1+t}}ds + \int_{\max\{t-1,0\}}^t \frac{1}{(t-s)^{j/2}} \de s\right) \leq \frac{C}{\left(1 + |x+c_1t|\right)^{r}}.
\end{align*}
Thus, as in~\eqref{JUNeq} we establish
\begin{align*}
\widehat{I}_j \leq C \left[\frac{1}{\left(1+|x+c_1t|+\sqrt{t}\right)^{r}} + \frac{e^{-\tfrac{(x+c_1t)^2}{M(1+t)}}}{\sqrt{1+t}}\right], \qquad j = 0,1.
\end{align*}
Subsequently, using $r \geq 3$ and~\eqref{L1integral}, we estimate
\begin{align*}
\widehat{II}_j &\leq C e^{-\tfrac{(x+c_1t)^2}{M(1+t)}} \int_0^t \int_{|y+c_1s| \leq \frac{|x+c_1t|}{2}} \frac{e^{-\tfrac{2(x-y+c_1(t-s))^2}{M(t-s)}}}{\left(1 + |y+c_1s| + \sqrt{s}\right)^{r} \sqrt{1+s} (t-s)^{(1+j)/2}} \de y \de s\\
&\leq Ce^{-\tfrac{(x+c_1t)^2}{M(1+t)}} \int_0^t \frac{1}{\sqrt{1+s}\sqrt{t-s}} \left(\int_{\R} \frac{e^{-\tfrac{2(x-y+c_1(t-s))^2}{M(t-s)}}}{\sqrt{t-s}\left(1+\sqrt{s}\right)^r} \de y + \int_\R \frac{1}{\left(1 + |y+c_1s| + \sqrt{s}\right)^{r}} \de y\right) \de s \\ &\qquad\qquad\qquad\qquad\qquad\qquad\qquad\qquad\quad\leq C e^{-\tfrac{(x+c_1t)^2}{M(1+t)}} \int_0^t \frac{1}{\sqrt{t-s} \left(1+\sqrt{s}\right)^{r}} ds \leq C \frac{e^{-\tfrac{(x+c_1t)^2}{M(1+t)}}}{\sqrt{1+t}}.
\end{align*}
Finally, using $r \geq 3$ and~\eqref{integralID3}, it holds
\begin{align*}
&\widehat{III}_j \leq \frac{C}{\left(1 + |x+c_1t| + \sqrt{t}\right)^{r}} \int_{\tfrac{t}{2}}^t \int_{|y+c_1s| \geq \frac{|x+c_1t|}{2}} \frac{e^{-\tfrac{(x-y+c_1(t-s))^2}{M(t-s)} - \tfrac{(y+c_2s)^2}{M(1+s)}}}{\sqrt{1+s}(t-s)^{(1+j)/2} }\de y\de s\\
&\leq \frac{C}{\left(1 + |x+c_1t| + \sqrt{t}\right)^{r}} \left(\int_\R \frac{e^{-\tfrac{(x+c_1 t + (c_2-c_1)s)^2}{M(1+t)}}}{\sqrt{1+t}}ds + \int_{\max\{t-1,0\}}^t \frac{1}{(t-s)^{j/2}} \de s\right) \leq \frac{C}{\left(1 + |x+c_1t| + \sqrt{t}\right)^{r}}.
\end{align*}
This concludes the estimation of the last integral $IV_j$. We have obtained
\begin{align} IV_j \leq C \left[\frac{1}{\left(1+|x+c_1t|+\sqrt{t}\right)^{r}} + \frac{e^{-\tfrac{(x+c_1t)^2}{M(1+t)}}}{\sqrt{1+t}}\right], \qquad j = 0,1. \label{estL3}\end{align}

\paragraph*{Final nonlinear estimate.} Combining~\eqref{estM2},~\eqref{estM4},~\eqref{estM5},~\eqref{estL2} and~\eqref{estL3} yields the bound
\begin{align}\begin{split}&\left|\int_0^t  \int_{\R} \left(\theta(x-y,t-s) f_1(z(y,s)) + \theta_x(x-y,t-s)g_1(z(y,s))\right)\de y\de s\right|\\ &\qquad \leq C\eta_A(t)^2\left[\frac{1}{\left(1+|x+c_1t|+\sqrt{t}\right)^{r}} + \frac{e^{-\tfrac{(x+c_1t)^2}{M(1+t)}}}{\sqrt{1+t}}\right],\end{split} \label{algebraicNL}\end{align}
on the nonlinear term in~\eqref{Duhamel}.

\subsection{Conclusion}

For exponentially localized initial conditions $(u_0,v_0) \in X_{\rho_E}^\alpha$ and $t \in [0,T)$ such that $\eta_E(t) \leq r_0$, estimates~\eqref{linexp} and~\eqref{exponentialNL} on the linear and nonlinear terms in~\eqref{Duhamel} yield
\begin{align*}
\left|u(x,t)e^{\tfrac{(x+c_1t)^2}{M(1+t)}}\right| \leq C\left(\delta + \eta_E(t)^2\right), \qquad x \in \R.
\end{align*}
Analogously, one obtains
\begin{align*}
\left|v(x,t)e^{\tfrac{(x+c_2t)^2}{M(1+t)}}\right| \leq C\left(\delta + \eta_E(t)^2\right), \qquad x \in \R,
\end{align*}
and we conclude that~\eqref{etaest} holds for $i = E$, which proves, as explained in~\S\ref{sec:planM}, Theorem~\ref{mainresult1} for exponentially localized initial data. Similarly, we obtain from the estimates~\eqref{linalg} and~\eqref{algebraicNL} on the linear and nonlinear terms in~\eqref{Duhamel} that~\eqref{etaest} holds for $i = A$, which yields Theorem~\ref{mainresult1} for algebraically localized initial data $(u_0,v_0) \in X_{\rho_A}^\alpha$. $\hfill \Box$

\section{Proof of Theorem~\ref{mainresult2}} \label{sec:proofMR2}

In this proof, $C>1$ denotes a constant, which is independent of $\delta,x$ and $t$ and that will be taken larger if necessary.

\subsection{Plan of proof} \label{sec:planIRR}

In contrast to Theorem~\ref{mainresult1}, we allow in Theorem~\ref{mainresult2} for (irrelevant) nonlinear couplings which are not of mix-type. As explained in~\S\ref{sec:mainresult2} and Remark~\ref{remspatiotemp}, we have to incorporate non-Gaussian upper bounds of the form~\eqref{drag} into our spatio-temporal weight to accommodate such nonlinear terms.

Thus, in contrast to Theorem~\ref{mainresult1}, we define our spatio-temporal weight $\eta \colon [0,T) \to [0,\infty)$ this time by
\begin{align*}
 \eta(t) &= \sup_{\begin{smallmatrix} x \in \R \\ 0 \leq s \leq t \end{smallmatrix}} \left(\left|u(x,s)\right| \left[\frac{e^{-\tfrac{(x+c_1s)^2}{M(1+s)}}}{\sqrt{1+s}} + \int_0^s \frac{e^{-\tfrac{\left(x + s c_1 + r (c_2 - c_1)\right)^2}{M(1+s)}}}{\sqrt{1+s} (1+r)}\left(\frac{\sqrt[4]{r+1}}{\sqrt{r}} + \frac{1}{\sqrt{s-r}}\right)\de r\right]^{-1}\right.\\
 &\qquad \qquad \qquad + \left.\left|v(x,s)\right| \left[\frac{e^{-\tfrac{(x+c_2s)^2}{M(1+s)}}}{\sqrt{1+s}} + \int_0^s \frac{e^{-\tfrac{\left(x + s c_2 + r (c_1 - c_2)\right)^2}{M(1+s)}}}{\sqrt{1+s} (1+r)}\left(\frac{\sqrt[4]{r+1}}{\sqrt{r}} + \frac{1}{\sqrt{s-r}}\right)\de r \right]^{-1}\right).
 \end{align*}
The further set-up is the same as in the proof of Theorem~\ref{mainresult1} and is skipped to avoid unnecessary repetitions. Thus, as in Theorem~\ref{mainresult1}, the result follows by establishing the key inequality
\begin{align} \eta(t) \leq C\left(\delta + \eta(t)^2\right), \label{etaest3}\end{align}
for all $t \in [0,T)$ with $\eta(t) \leq r_0$. We prove~\eqref{etaest3} by estimating the nonlinear terms in integral formulation~\eqref{Duhamel} for the $u$-component in~\eqref{GRD} and, analogously, for the $v$-component. The estimate~\eqref{linexp} on the linear term in~\eqref{Duhamel} has already been obtained in~\S\ref{estexplinM}.

\begin{remark} \label{dragboundrem}
{\upshape
We note that the proof of Theorem~\ref{mainresult2} yields more detailed, spatio-temporal estimates than the bounds~\eqref{temporaldec2}. For exponentially localized initial data $(u_0,v_0) \in X_{\rho_E}^\alpha$ satisfying $\|(u_0,v_0)\|_{\rho_E} < \delta$ we infer
\begin{align*} \left|u(x,t)\right| &\leq \epsilon \left[\frac{e^{-\tfrac{(x+c_1t)^2}{M(1+t)}}}{\sqrt{1+t}} + \int_0^t \frac{e^{-\tfrac{\left(x + t c_1 + s (c_2 - c_1)\right)^2}{M(1+t)}}}{\sqrt{1+t} (1+s)}\left(\frac{\sqrt[4]{s+1}}{\sqrt{s}} + \frac{1}{\sqrt{t-s}}\right)\de s\right], \\
\left|v(x,t)\right| &\leq \epsilon \left[\frac{e^{-\tfrac{(x+c_2t)^2}{M(1+t)}}}{\sqrt{1+t}} + \int_0^t \frac{e^{-\tfrac{\left(x + t c_2 + s (c_1 - c_2)\right)^2}{M(1+t)}}}{\sqrt{1+t} (1+s)}\left(\frac{\sqrt[4]{s+1}}{\sqrt{s}} + \frac{1}{\sqrt{t-s}}\right)\de s\right],\end{align*}
for all $x \in \R$ and $t \geq 0$. Hence, the solutions to~\eqref{GRD} are, as in~\eqref{pointdec}, controlled by drifting Gaussians, which exhibit the diffusive decay as predicted by the linear system~\eqref{GRDLIN}, and by terms of the form~\eqref{drag}. which exhibit algebraic, non-Gaussian-like decay, see Figure~\ref{decaytypes}. We emphasize that the occurrence of such non-Gaussian upper bounds is not artificial. In fact, bounds of the form~\eqref{drag} can be attained. For instance,
\begin{align*} u(x,t) = \frac{e^{-\tfrac{(x+c_1t)^2}{4(1+t)}}}{\sqrt{4 \pi (1+t)}}, \qquad v(x,t) = \int_0^t \frac{e^{-\tfrac{(x+ tc_2 + s(c_1-c_2))^2}{1+t}}}{16 \pi ^2 (1+s)^{3/2} \sqrt{1+t}} \de s,\end{align*}
is an exact solution with exponentially localized initial data to the nonlinearly coupled system
\begin{align*}
\begin{split}
u_t &= u_{xx} + c_1 u_x, \\
v_t &= \tfrac{1}{4} v_{xx} + c_2 v_x + u^4,
\end{split} \qquad t \geq 0, x \in \R.
\end{align*}
}\end{remark}

\begin{remark} 
{\upshape
Terms of the form~\eqref{drag} also occur in the estimates on page 340 in the nonlinear stability analysis~\cite{HOWZUM} of viscous under-compressive shocks satisfying a system of conservation laws. Since, on the linear level, different components might exhibit different velocities in such systems, upper bounds of the form~\eqref{drag} arise in the pointwise estimates. Instead of incorporating such upper bounds into the spatio-temporal weights as in the current analysis, terms of the form~\eqref{drag} are estimated in~\cite{HOWZUM} by a sum of drifting Gaussians and algebraic correction terms. In particular, each component is estimated by the same upper bound and, therefore, differences in velocities are not exploited in~\cite{HOWZUM}.
}\end{remark}

\subsection{Nonlinear estimates}

Assume $t \in [0,T)$ is such that $\eta(t) \leq r_0$. By~\eqref{nonlinearbounds2} the nonlinear term in~\eqref{Duhamel} can be estimated as follows
\begin{align}
&\left|\int_0^t  \int_{\R} \left(\theta(x-y,t-s) f_1(z(y,s)) + \theta_x(x-y,t-s)g_1(z(y,s))\right)\de y\de s\right| \nonumber \\
&\ \leq C\int_0^t\int_\R \left[\theta(x-y,t-s)|v(y,s)|^4 + |\theta_x(x-y,t-s)| |v(y,s)|^3 + \theta(x-y,t-s)|u(y,s)|^4 \right. \label{estI3}\\
&\phantom{\int_0^t\int_\R} \left.\phantom{|u_x|^4}  + |\theta_x(x-y,t-s)| |u(y,s)| \left(|u(y,s)| + |v(y,s)|\right) + \theta(x-y,t-s)|u(y,s)| |v(y,s)|\right] \de y \de s. \nonumber
\end{align}
In the following, we estimate the convolutions in~\eqref{estI3} in three steps. First, we estimate the integrals coming from irrelevant couplings which are not of mix-type. Then, we estimate all remaining irrelevant terms and marginal terms in divergence form. Finally, we estimate the remaining relevant and marginal mix-terms, where we exploit the difference in velocities.

\subsubsection{Irrelevant couplings which are not of mix-type} \label{sec:irrelevantcoupling}
First, we use $M \geq 4d_1$, $\eta(t) \leq r_0$ and the boundedness of $x \mapsto xe^{-x^2}$ on $\R$ and we observe
\begin{align}
\int_0^t\int_\R \theta(x-y,t-s)|v(y,s)|^4 + |\theta_x(x-y,t-s)| |v(y,s)|^3 \de y \de s \leq C\eta(t)^2\left(J_1 + J_2\right), \label{estI5}
\end{align}
with
\begin{align*}
J_1 &= \int_0^t\int_{\R} \frac{e^{-\tfrac{(x-y+c_1(t-s))^2}{M(t-s)} - \tfrac{(y+c_2s)^2}{M(1+s)}}}{(1+s)^{3/2}} \left(\frac{1}{\sqrt{1+s}\sqrt{t-s}} + \frac{1}{t-s}\right)\de y\de s, \\
J_2 &= \int_0^t\int_0^s\int_{\R} \frac{e^{-\tfrac{(x-y+c_1(t-s))^2}{M(t-s)} - \tfrac{(y+c_2s+r(c_1-c_2))^2}{M(1+s)}}}{(1+s)^{3/2}(1+r)} \left(\frac{\sqrt[4]{1+r}}{\sqrt{r}} + \frac{1}{\sqrt{s-r}}\right)\left(\frac{1}{\sqrt{1+s}\sqrt{t-s}} + \frac{1}{t-s}\right)\de y\de r\de s.
\end{align*}

\paragraph*{Estimate on $J_1$.} Using~\eqref{integralid1}, we calculate the inner integral in $J_1$ and obtain
\begin{align}
\begin{split}
J_1 &\leq C \int_0^t \frac{e^{-\tfrac{(x+c_1t + s(c_2-c_1))^2}{M(1+t)}}}{\sqrt{1+t}(1+s)} \left(\frac{1}{\sqrt{1+s}} + \frac{1}{\sqrt{t-s}}\right)\de y\de s. 
 \end{split} \label{estJ1}
\end{align}

\paragraph*{Estimate on $J_2$.} We use~\eqref{integralid1} again to compute the inner integral in $J_2$ and establish
\begin{align*}
J_2 &\leq \int_0^t\int_0^s \frac{e^{-\tfrac{(x+c_1t + (s-r)(c_2-c_1))^2}{M(1+t)}}}{\sqrt{1+t}(1+r)(1+s)} \left(\frac{\sqrt[4]{1+r}}{\sqrt{r}} + \frac{1}{\sqrt{s-r}}\right)\left(\frac{1}{\sqrt{1+s}} + \frac{1}{\sqrt{t-s}}\right)\de r\de s \\
 &\leq C\int_0^t \frac{e^{-\tfrac{(x+c_1t + \tilde{r}(c_2-c_1))^2}{M(1+t)}}}{\sqrt{1+t}(1+\tilde{r})} \left[\int_{\tilde{r}}^t \frac{1}{(1+s-\tilde{r})^{3/4} \sqrt{s-\tilde{r}}} \left(\frac{1}{\sqrt{1+\tilde{r}}} + \frac{1}{\sqrt{t-s}}\right)\de s\right.\\ &\qquad \qquad \qquad \qquad \qquad \qquad +\left. \frac{1}{\sqrt{\tilde{r}}} \int_{\tilde{r}}^t \frac{1}{(1+s-\tilde{r})}\left(\frac{1}{\sqrt{1+s}} + \frac{1}{\sqrt{t-s}}\right) \de s\right]\de \tilde{r}.
\end{align*}
Employing the estimate
\begin{align*}
\int_{\tilde{r}}^t \frac{(1+s-\tilde{r})^{-3/4}}{\sqrt{t-s}\sqrt{s-\tilde{r}}} \de s \leq C\left(\int_{\tilde{r}}^{\frac{t+\tilde{r}}{2}} \frac{(1+s-\tilde{r})^{-3/4}}{\sqrt{t-r}\sqrt{s-\tilde{r}}} \de s + \int_{\frac{t+\tilde{r}}{2}}^t \frac{(1+s-\tilde{r})^{-3/4}}{\sqrt{t-s}\sqrt{}} \de s\right) \leq \frac{C}{\sqrt{t-\tilde{r}}},
\end{align*}
in the above yields
\begin{align}
J_2 &\leq C\left[\int_0^t \frac{e^{-\tfrac{\left(x + t c_1 + \tilde{r} (c_2 - c_1)\right)^2}{M(1+t)}}}{\sqrt{1+t} (1+\tilde{r})}\left(\frac{1}{\sqrt{\tilde{r}}} + \frac{1}{\sqrt{t-\tilde{r}}}\right)\de \tilde{r}\right].\label{estJ2}
 \end{align}

\subsubsection{Other irrelevant terms and marginal terms in divergence form}
We use $M \geq 4d_1$, $\eta(t) \leq r_0$ and the boundedness of $x \mapsto xe^{-x^2}$ on $\R$ and we continue with estimating the second nonlinear term in~\eqref{estI3}:
\begin{align}
\begin{split}
\int_0^t\int_\R & \left[\theta(x-y,t-s)|u(y,s)|^4 + |\theta_x(x-y,t-s)| |u(y,s)| \left(|u(y,s)| + |v(y,s)|\right)\right] \de y \de s\\ &\leq C \eta(t)^2 \left(I_* + J_3\right), \end{split}\label{estI4}
\end{align}
with $I_*$ defined in~\eqref{defI*} and
\begin{align*}
J_3 &= \int_0^t\int_0^s\int_{\R} \frac{e^{-\tfrac{(x-y+c_1(t-s))^2}{M(t-s)} - \tfrac{(y+c_1s+r(c_2-c_1))^2}{M(1+s)}}}{(1+s)(1+r)} \left(\frac{\sqrt[4]{1+r}}{\sqrt{r}} + \frac{1}{\sqrt{s-r}}\right)\left(\frac{1}{(1+s)\sqrt{t-s}} + \frac{1}{t-s}\right)\de y\de r\de s.
\end{align*}
Note that an estimate on $I_*$ has been obtained in~\eqref{estM2}.

\paragraph*{Estimate on $J_3$.} Employing~\eqref{integralid1}, we compute the inner integral in $J_3$ and we establish the bound
\begin{align}
\begin{split}
J_3 &\leq C\int_0^t\int_0^s \frac{e^{-\tfrac{(x + c_1t + r (c_2-c_1))^2}{M(1+t)}}}{\sqrt{1+t}(1+r)\sqrt{1+s}} \left(\frac{\sqrt[4]{1+r}}{\sqrt{r}} + \frac{1}{\sqrt{s-r}}\right)\left(\frac{1}{(1+s)} + \frac{1}{\sqrt{t-s}}\right)\de r\de s \\
&= C\int_0^t \frac{e^{-\tfrac{(x + c_1t + r (c_2-c_1))^2}{M(1+t)}}}{\sqrt{1+t}} \left[\frac{1}{(1+r)^{3/4}\sqrt{r}} \int_r^t \left(\frac{1}{(1+s)^{3/2}} + \frac{1}{\sqrt{1+s} \sqrt{t-s}}\right) \de s\right. \\
&\quad \left. + \frac{1}{(1+r)^{3/2}} \int_r^t \left(\frac{1}{(1+s)\sqrt{s-r}} + \frac{1}{\sqrt{t-s}\sqrt{s-r}}\right)\de s\right]\de r \ \leq \ C\int_0^t \frac{e^{-\tfrac{\left(x + t c_1 + r (c_2 - c_1)\right)^2}{M(1+t)}}}{\sqrt{1+t} \sqrt{r} (1+r)^{3/4}} \de r.
\end{split}
\label{estJ3}
\end{align}

\subsubsection{Relevant and marginal mix-terms}
We use $M \geq 8d_1$ and estimate the remaining nonlinear term in~\eqref{estI3} as follows
\begin{align}
\int_0^t\int_\R \theta(x-y,t-s) |v(y,s)||u(y,s)| \de y \de s \leq C\eta(t)^2\left(I_0 + J_4 + J_5 + J_6\right), \label{estI6}
\end{align}
with $I_0$ defined in~\eqref{defIj} and
\begin{align*}
J_4 &= \int_0^t \int_0^s \int_\R \frac{e^{-\tfrac{2(x-y+c_1(t-s))^2}{M(t-s)} - \tfrac{(y+c_1s+r(c_2-c_1))^2}{M(1+s)}-\tfrac{(y+c_2s)^2}{M(1+s)}}}{(1+s)\sqrt{t-s}(1+r)} \left(\frac{\sqrt[4]{1+r}}{\sqrt{r}} + \frac{1}{\sqrt{s-r}}\right)\de y \de r \de s,\\
J_5 &= \int_0^t \int_0^s \int_0^s \int_\R \frac{e^{-\tfrac{2(x-y+c_1(t-s))^2}{M(t-s)} - \tfrac{(y+c_1s+r(c_2-c_1))^2}{M(1+s)}-\tfrac{(y+c_2s+p(c_1-c_2))^2}{M(1+s)}}}{(1+s)\sqrt{t-s}(1+r)(1+p)} \left(\frac{\sqrt[4]{1+r}}{\sqrt{r}} + \frac{1}{\sqrt{s-r}}\right) \\
& \qquad \qquad \qquad \qquad  \cdot \left(\frac{\sqrt[4]{1+p}}{\sqrt{p}} + \frac{1}{\sqrt{s-p}}\right)\de y \de p \de r \de s,\\
J_6 &= \int_0^t \int_0^s \int_\R \frac{e^{-\tfrac{2(x-y+c_1(t-s))^2}{M(t-s)} - \tfrac{(y+c_1s)^2}{M(1+s)}-\tfrac{(y+c_2s + r(c_1 -c_2))^2}{M(1+s)}}}{(1+s)\sqrt{t-s}(1+r)} \left(\frac{\sqrt[4]{1+r}}{\sqrt{r}} + \frac{1}{\sqrt{s-r}}\right)\de y \de r \de s.
\end{align*}
Note that an estimate on $I_0$ has been obtained in~\eqref{estM4}. In the following, we estimate $J_4,J_5$ and $J_6$.

\paragraph*{Estimate on $J_4$.} Using~\eqref{integralid1} and Young's inequality, we calculate the inner integral in $J_4$ and obtain
\begin{align}
\begin{split}
J_4 &\leq C \int_0^t \int_0^s \frac{e^{-\tfrac{\left(x + c_1 t + r (c_2-c_1)\right)^2}{2 M(1+t)}-\tfrac{\left(x+ c_1 t + s (c_2-c_1)\right)^2}{2 M (1+t)} - \tfrac{\left(x+c_1t+\frac{1}{2} (r+s) (c_2-c_1)\right)^2}{M (1+t)}-\tfrac{(s-r)^2 (2t + 1 - s) (c_1-c_2)^2}{4 M (1+s) (1+t)}}}{\sqrt{1+t}\sqrt{1+s}(1+r)}\\
&\qquad \qquad \qquad \qquad\qquad\qquad\qquad\qquad\qquad \cdot \left(\frac{\sqrt[4]{1+r}}{\sqrt{r}} + \frac{1}{\sqrt{s-r}}\right) \de r \de s \leq C(J_{41} + J_{42}),
\end{split} \label{estJ42}
\end{align}
with
\begin{align*}
J_{41} &= \int_0^t \frac{e^{-\tfrac{\left(x + c_1 t + r (c_2-c_1)\right)^2}{M(1+t)}}}{\sqrt{1+t}(1+r)} \int_r^t \frac{e^{-\tfrac{(s-r)^2 (c_1-c_2)^2}{4 M (1+s)}}}{\sqrt{1+s}}\left(\frac{\sqrt[4]{1+r}}{\sqrt{r}} + \frac{1}{\sqrt{s-r}}\right) \de s \de r,\\
J_{42} &= \int_0^t \frac{e^{-\tfrac{\left(x + c_1 t + s (c_2-c_1)\right)^2}{M(1+t)}}}{\sqrt{1+t}\sqrt{1+s}} \int_0^s \frac{e^{-\tfrac{(s-r)^2 (c_1-c_2)^2}{4 M (1+s)}}}{1+r}\left(\frac{\sqrt[4]{1+r}}{\sqrt{r}} + \frac{1}{\sqrt{s-r}}\right) \de r \de s.
\end{align*}
We compute
\begin{align}
\int_r^\infty \frac{e^{-\tfrac{(s-r)^2 a}{1+s}}}{\sqrt{1+s}}\de s &= \frac{\sqrt{\pi } \left(e^{4 a (r+1)} \text{erfc}\left(\sqrt{4 a (r+1)}\right)+1\right)}{2 \sqrt{a}}, \label{integralID4}
\end{align}
for $a > 0$ and $r \geq 0$, where $\text{erfc}(x) = 1 - \text{erf}(x)$ denotes the complementary error function. By applying the Chernoff bound to the Gaussian distribution, we find $\text{erfc}(x) \leq 2e^{-x^2}$ for $x \geq 0$. Hence, using~\eqref{integralID5},~\eqref{integralID4} and $c_1 \neq c_2$, we bound
\begin{align} \int_r^\infty  \frac{e^{-\tfrac{(s-r)^2 (c_1-c_2)^2}{4 M (1+s)}}}{\sqrt{1+s}\sqrt{s-r}}\de s \leq \int_r^{2r}  \frac{e^{-\tfrac{(s-r)^2 (c_1-c_2)^2}{4 M (1+2r)}}}{\sqrt{1+r}\sqrt{s-r}} \de s + \int_{2r}^\infty  \frac{e^{-\tfrac{(s-r)^2 (c_1-c_2)^2}{4 M (1+s)}}}{\sqrt{1+s}\sqrt{r}}\de s \leq C \frac{\sqrt[4]{1+r}}{\sqrt{r}}, \label{integralID6} \end{align}
for $0 \leq r \leq s$. So, with the aid of~\eqref{integralID4} and~\eqref{integralID6} we obtain
\begin{align}
J_{41} \leq C\int_0^t \frac{e^{-\tfrac{\left(x + c_1 t + r (c_2-c_1)\right)^2}{M(1+t)}}}{\sqrt{1+t} (1+r)^{3/4} \sqrt{r}} \de r. \label{estJ41}
\end{align}
We proceed by estimating the inner integral in $J_{42}$. First, the fact that $y^be^{-y} \leq 1$ for $y \geq 0$ implies
\begin{align} e^{-\frac{a x^2}{1+s}} 
\leq \frac{(1+s)^b}{\max\left\{1,a x^{2}\right\}^b} \leq \frac{2^b (1+s)^b}{\left(1 + a x^2\right)^b} \leq 2^{b} \max\left\{1,a^{-b}\right\} \frac{(1+s)^b}{(1+x)^{2b}} ,\label{expest}\end{align}
for $b \in (0,1], a > 0$ and $x, s \geq 0$. Hence, using $c_1 \neq c_2$,~\eqref{integralID5} and~\eqref{expest}, we obtain
\begin{align*}
\int_0^s &\frac{e^{-\tfrac{(s-r)^2 (c_1-c_2)^2}{4 M (1+s)}}}{1+r}\left(\frac{\sqrt[4]{1+r}}{\sqrt{r}} + \frac{1}{\sqrt{s-r}}\right) \de r \leq C\int_0^{\frac{s}{2}} \frac{e^{-\tfrac{s^2 (c_1-c_2)^2}{16 M (1+s)}}}{1+r}\left(\frac{\sqrt[4]{1+r}}{\sqrt{r}} + \frac{1}{\sqrt{s-r}}\right) \de r \\&\qquad\qquad\qquad\qquad\qquad\qquad\qquad+ C\int_{\frac{s}{2}}^s \frac{e^{-\tfrac{(s-r)^2 (c_1-c_2)^2}{4 M (1+s)}}}{1+s}\left(\frac{\sqrt[4]{1+s}}{\sqrt{s}} + \frac{1}{\sqrt{s-r}}\right) \de r \leq C \frac{1}{\sqrt[4]{1+s} \sqrt{s}},
\end{align*}
for $r \geq 0$. Thus, $J_{42}$ enjoys the bound
\begin{align*}
J_{42} \leq C\int_0^t \frac{e^{-\tfrac{\left(x + c_1 t + s (c_2-c_1)\right)^2}{M(1+t)}}}{\sqrt{1+t} (1+s)^{3/4} \sqrt{s}} \de s.
\end{align*}
Plugging the latter and~\eqref{estJ41} into~\eqref{estJ42} yields
\begin{align}
J_4 \leq C\int_0^t \frac{e^{-\tfrac{\left(x + c_1 t + s (c_2-c_1)\right)^2}{M(1+t)}}}{\sqrt{1+t} (1+s)^{3/4} \sqrt{s}} \de s. \label{estJ4}
\end{align}

\paragraph*{Estimate on $J_5$.} Using~\eqref{integralid1} and Young's inequality, we compute the inner integral in $J_5$ and establish
\begin{align}
\begin{split}
J_5 &\leq C \int_0^t \int_0^s \int_0^s \frac{e^{-\tfrac{\left(x + c_1 t + r (c_2-c_1)\right)^2}{2 M(1+t)}-\tfrac{\left(x+ c_1 t + (s-p) (c_2-c_1)\right)^2}{2 M (1+t)} - \tfrac{\left(x + c_1 t - \frac{1}{2} (c_1-c_2) (r+s-p)\right)^2}{M (1+t)}-\tfrac{(p+r-s)^2 (2t+1-s) (c_1-c_2)^2 }{4 M (1+s) (1+t)}}}{\sqrt{1+t}\sqrt{1+s}(1+r)(1+p)}\cdot\\
&\qquad \qquad \qquad \qquad \qquad \left(\frac{\sqrt[4]{1+r}}{\sqrt{r}} + \frac{1}{\sqrt{s-r}}\right)\left(\frac{\sqrt[4]{1+p}}{\sqrt{p}} + \frac{1}{\sqrt{s-p}}\right) \de p \de r \de s \leq C(J_{51} + J_{52}),
\end{split} \label{estJ52}
\end{align}
with
\begin{align*}
J_{51}\!&=\!\int_0^t \frac{e^{-\tfrac{\left(x + c_1 t + r (c_2-c_1)\right)^2}{M(1+t)}}}{\sqrt{1+t}(1+r)} \int_r^t\! \int_0^s \frac{e^{-\tfrac{(p+r-s)^2 (c_1-c_2)^2}{4 M (1+s)}}}{(1+p)\sqrt{1+s}}\left(\frac{\sqrt[4]{1+r}}{\sqrt{r}} + \frac{1}{\sqrt{s-r}}\right)\left(\frac{\sqrt[4]{1+p}}{\sqrt{p}} + \frac{1}{\sqrt{s-p}}\right) \de p \de s \de r,\\
J_{52}\!&=\! \int_0^t \frac{e^{-\tfrac{\left(x + c_1 t + \tp (c_2-c_1)\right)^2}{M(1+t)}}}{\sqrt{1+t}} \int_{\tp}^t \! \int_0^s \frac{e^{-\tfrac{(r-\tp)^2 (c_1-c_2)^2}{4 M (1+s)}}}{(1+r)(1+s-\tp)\sqrt{1+s}}\left(\frac{\sqrt[4]{1+r}}{\sqrt{r}} + \frac{1}{\sqrt{s-r}}\right)\cdot \\
&\qquad \qquad \qquad \qquad \qquad \qquad \qquad \qquad \qquad \qquad \qquad \qquad  \left(\frac{\sqrt[4]{1+s-\tp}}{\sqrt{s-\tp}} + \frac{1}{\sqrt{\tp}}\right) \de r \de s \de \tp.
\end{align*}
We split the inner integral in $J_{51}$ in five parts. First, as in~\eqref{integralID4} and~\eqref{integralID6}, we estimate
\begin{align*}
\int_r^t \int_0^{\frac{s-r}{2}} & \frac{e^{-\tfrac{(p+r-s)^2 (c_1-c_2)^2}{4 M (1+s)}}}{(1+p)\sqrt{1+s}}\left(\frac{\sqrt[4]{1+r}}{\sqrt{r}} + \frac{1}{\sqrt{s-r}}\right)\left(\frac{\sqrt[4]{1+p}}{\sqrt{p}} + \frac{1}{\sqrt{s-p}}\right) \de p \de s\\ &\leq C\int_r^t \frac{e^{-\tfrac{(s-r)^2 (c_1-c_2)^2}{16 M (1+s)}}}{\sqrt{1+s}} \left(\frac{\sqrt[4]{1+r}}{\sqrt{r}} + \frac{1}{\sqrt{s-r}}\right) \de s \leq C\frac{\sqrt[4]{1+r}}{\sqrt{r}},
\end{align*}
for $r \in [0,t]$. Second, using $c_1 \neq c_2$, we obtain
\begin{align*}
\int_r^t \int_{\frac{s-r}{2}}^s \frac{e^{-\tfrac{(p+r-s)^2 (c_1-c_2)^2}{4 M (1+s)}}}{\sqrt{1+s}(1+p)^{3/4} \sqrt{p}} \de p \de s &\leq C\int_r^t \int_\R \frac{e^{-\tfrac{(p+r-s)^2 (c_1-c_2)^2}{4 M (1+s)}}}{\sqrt{1+s}(1+s-r)^{3/4}\sqrt{s-r}} \de p \de s
\leq C,
\end{align*}
for $r \in [0,t]$. Third, using~\eqref{expest} and $c_1 \neq c_2$, we establish
\begin{align*}
&\int_r^t \int_{\frac{s-r}{2}}^s \frac{e^{-\tfrac{(p+r-s)^2 (c_1-c_2)^2}{4 M (1+s)}} (1+s)^{-1/2} }{(1+p) \sqrt{s-p}} \de p \de s \leq C\int_r^t \left(\int_{\frac{s-r}{2}}^{s-r} \frac{(1+s)^{-1/8}}{(1+s-r-p)^{3/4} (1+s-r) \sqrt{s-r-p} } \de p \right.\\&\qquad \qquad \qquad \qquad \left. + \int_{s-r}^s \frac{(1+s)^{-1/4}}{\sqrt{1+p-(s-r)} (1+s-r)\sqrt{s-p}} \de p \right)\de s \leq C\int_r^t \frac{1}{(1+s-r)^{9/8}} \de s \leq C,
\end{align*}
for $r \in [0,t]$. Fourth, using~\eqref{expest} and $c_1 \neq c_2$ again, we bound
\begin{align*}
&\int_r^t \int_{\frac{s-r}{2}}^s \frac{e^{-\tfrac{(p+r-s)^2 (c_1-c_2)^2}{4 M (1+s)}}(1+s)^{-1/2}}{(1+p)^{3/4} \sqrt{p} \sqrt{s-r}} \de p \de s \leq C\int_r^t \left(\int_{0}^{s-r} \frac{(1+s)^{-1/4}}{(1+s-r)^{3/4} \sqrt{1+s-r-p} \sqrt{p}} \de p \right.\\
&\qquad \qquad\qquad \qquad \left. + \int_{s-r}^s \frac{(1+s)^{-1/4}}{(1+s-r)^{5/8} \sqrt{p-(s-r)} (1+p-(s-r))^{5/8}} \de p \right)\frac{1}{\sqrt{s-r}}\de s
\\& \qquad \leq C\int_r^t \frac{(1+r)^{-1/4}}{(1+s-r)^{5/8}\sqrt{s-r}} \de s \leq \frac{C}{\sqrt[4]{1+r}}.
\end{align*}
for $r \in [0,t]$. Finally, we estimate the fifth part, using~\eqref{expest} and $c_1 \neq c_2$:
\begin{align*}
&\int_r^t \int_{\frac{s-r}{2}}^s \frac{e^{-\tfrac{(p+r-s)^2 (c_1-c_2)^2}{4 M (1+s)}}(1+s)^{-1/2}}{(1+p) \sqrt{s-p} \sqrt{s-r}} \de p \de s \leq C\int_r^t \left(\int_0^{s-r} \frac{(1+s)^{-3/8}}{(1+s-r)^{3/4} \sqrt[4]{1+p} (s-r-p)^{3/4}} \de p \right.\\&\qquad \qquad \qquad \qquad \left. + \int_{s-r}^s \frac{(1+s)^{-3/8}}{\sqrt{1+p-s+r} (1+s-r)^{3/4} \sqrt{s-p}} \de p \right)\frac{1}{\sqrt{s-r}}\de s
\\& \qquad \leq C\int_r^t \frac{(1+r)^{-3/8}}{(1+s-r)^{3/4}\sqrt{s-r}} \de s\leq \frac{C}{(1+r)^{5/8}},
\end{align*}
for $r \in [0,t]$. We conclude $J_{51}$ is estimated by
\begin{align} J_{51} \leq C\int_0^t \frac{e^{-\tfrac{\left(x + c_1 t + r (c_2-c_1)\right)^2}{M(1+t)}}}{\sqrt{1+t}(1+r)^{3/4} \sqrt{r}} \de r. \label{estJ51}\end{align}
Subsequently, we split the inner integral in $J_{52}$ in four parts. First, using~\eqref{expest} and $c_1 \neq c_2$, we estimate
\begin{align*}
\int_{\tp}^t \int_0^{\frac{\tp}{2}} & \frac{e^{-\tfrac{(r-\tp)^2 (c_1-c_2)^2}{4 M (1+s)}}}{(1+r)(1+s-\tp)\sqrt{1+s}}\left(\frac{\sqrt[4]{1+r}}{\sqrt{r}} + \frac{1}{\sqrt{s-r}}\right) \left(\frac{\sqrt[4]{1+s-\tp}}{\sqrt{s-\tp}} + \frac{1}{\sqrt{\tp}}\right) \de r \de s \\
&\leq C\int_{\tp}^t \frac{e^{-\frac{\tp^2 (c_1-c_2)^2}{16 M (1+s)}}}{(1+s-\tp)\sqrt{1+s}} \left(\frac{\sqrt[4]{1+s-\tp}}{\sqrt{s-\tp}} + \frac{1}{\sqrt{\tp}}\right) \de s\\
&\leq C\left(\int_{\tp}^{2\tp} \frac{e^{-\frac{\tp^2 (c_1-c_2)^2}{16 M (1+2\tp)}}}{(1+s-\tp)\sqrt{1+s}} \left(\frac{\sqrt[4]{1+s-\tp}}{\sqrt{s-\tp}} + \frac{1}{\sqrt{\tp}}\right) \de s + \int_{2\tp}^\infty \frac{1}{(1+s-\tp)^{3/4}(1+\tp)\sqrt{s-\tp}} \de s\right. \\ &\qquad \qquad \qquad \qquad\qquad\qquad\qquad\quad\left. + \int_{2\tp}^\infty \frac{1}{(1+s-\tp)^{5/4}\sqrt{1+\tp}\sqrt{\tp}} \de s \right) \leq \frac{C}{(1+\tp)^{3/4}\sqrt{\tp}},
\end{align*}
for $\tp \in [0,t]$. Second, we establish
\begin{align*}
\int_{\tp}^t \int_{\frac{\tp}{2}}^s \frac{e^{-\tfrac{(r-\tp)^2 (c_1-c_2)^2}{4 M (1+s)}}(1+r)^{-3/4} r^{-1/2}}{ (1+s-\tp)^{3/4} \sqrt{s-\tp} \sqrt{1+s}} \de r \de s
&\leq C\int_{\tp}^t \int_\R \frac{e^{-\tfrac{(r-\tp)^2 (c_1-c_2)^2}{4 M (1+s)}}(1+\tp)^{-3/4} {\tp}^{-1/2} }{(1+s-\tp)^{3/4} \sqrt{s-\tp} \sqrt{1+s}} \de r \de s\\
\leq C\int_{\tp}^t \frac{(1+\tp)^{-3/4} {\tp}^{-1/2}}{(1+s-\tp)^{3/4} \sqrt{s-\tp}} \de s &\leq \frac{C}{(1+\tp)^{3/4}\sqrt{\tp}}.
\end{align*}
for $\tp \in [0,t]$, using $c_1 \neq c_2$. Third, using~\eqref{expest} and $c_1 \neq c_2$, we bound
\begin{align*}
\int_{\tp}^t \int_{\frac{\tp}{2}}^s  \frac{e^{-\tfrac{(r-\tp)^2 (c_1-c_2)^2}{4 M (1+s)}}(1+r)^{-3/4} {r}^{-1/2}}{ (1+s-\tp) \sqrt{\tp} \sqrt{1+s}} \de r \de s
&\leq C\int_{\tp}^t \int_{\frac{\tp}{2}}^s \frac{(1+\tp)^{-3/4}{\tp}^{-1/2}}{(1+s-\tp) (1+|\tp - r|)^{3/4} \sqrt{r} \sqrt[8]{1+s}} \de r \de s \\
\leq C\int_{\tp}^t \frac{(1+\tp)^{-3/4} {\tp}^{-1/2}}{(1+s-\tp)^{9/8}} \de s &\leq \frac{C}{(1+\tp)^{3/4}\sqrt{\tp}},
\end{align*}
for $\tp \in [0,t]$. Finally, the fourth term is, using~\eqref{integralID5},~\eqref{expest} and $c_1 \neq c_2$, estimated by
\begin{align*}
 \int_{\tp}^t &\int_{\frac{\tp}{2}}^s  \frac{e^{-\tfrac{(r-\tp)^2 (c_1-c_2)^2}{4 M (1+s)}}(1+r)^{-1}}{ (1+s-\tp) \sqrt{s-r} \sqrt{1+s}} \left(\frac{\sqrt[4]{1+s-\tp}}{\sqrt{s-\tp}} + \frac{1}{\sqrt{\tp}}\right) \de r \de s\\
&\leq C\int_{\tp}^t\frac{(1+\tp)^{-1}}{1+s-\tp} \left(\int_{\frac{\tp}{2}}^{\tp} \frac{e^{-\tfrac{(r-\tp)^2 (c_1-c_2)^2}{4 M (1+s)}}}{ \sqrt{\tp-r} \sqrt{1+s}}\de r + \int_{\tp}^s \frac{1}{\sqrt{s-r} \sqrt{r-\tp} \sqrt[4]{1+s}} \de r\right) \left(\frac{\sqrt[4]{1+s-\tp}}{\sqrt{s-\tp}} + \frac{1}{\sqrt{\tp}}\right)\de s \\
&\leq C\int_{\tp}^t \frac{(1+\tp)^{-1}}{(1+s-\tp)\sqrt[4]{1+s}} \left(\frac{\sqrt[4]{1+s-\tp}}{\sqrt{s-\tp}} + \frac{1}{\sqrt{\tp}}\right) \de s \leq \frac{C}{(1+\tp)^{3/4}\sqrt{\tp}},
\end{align*}
for $\tp \in [0,t]$. We conclude $J_{52}$ is estimated by
\begin{align*} J_{52} \leq C\int_0^t \frac{e^{-\tfrac{\left(x + c_1 t + \tp (c_2-c_1)\right)^2}{M(1+t)}}}{\sqrt{1+t}(1+\tp)^{3/4} \sqrt{\tp}} \de \tp.\end{align*}
Substituting the latter and~\eqref{estJ51} into~\eqref{estJ52} yields
\begin{align} J_5 \leq C\int_0^t \frac{e^{-\tfrac{\left(x + c_1 t + s (c_2-c_1)\right)^2}{M(1+t)}}}{\sqrt{1+t}(1+s)^{3/4} \sqrt{s}} \de s. \label{estJ53}\end{align}

\paragraph*{Estimate on $J_6$.} All that remains is establishing a bound on $J_6$. Using~\eqref{integralid1},~\eqref{innerest},~\eqref{integralID5} and $c_1 \neq c_2$, we calculate the inner integral in $J_6$ and estimate
\begin{align}
\begin{split}
&J_6 \leq C \frac{e^{-\tfrac{(x+c_1t)^2}{M(1+t)}}}{\sqrt{1+t}} \int_0^t \int_0^s \frac{e^{-\tfrac{(x+c_1 t + (s-r) (c_2-c_1))^2}{M (1+t)}-\tfrac{(s-r)^2 (t-s) (c_1-c_2)^2}{2 M (1+s) (1+t)}}}{\sqrt{1+s}(1+r)} \left(\frac{\sqrt[4]{1+r}}{\sqrt{r}} + \frac{1}{\sqrt{s-r}}\right) \de r \de s\\
&\ \leq C\frac{e^{-\tfrac{(x+c_1t)^2}{M(1+t)}}}{\sqrt{1+t}} \int_0^t \left(\int_0^{\frac{s}{2}} \frac{e^{-\tfrac{s^2 (t-s) (c_1-c_2)^2}{8 M (1+s) (1+t)}}}{\sqrt{1+s}(1+r)} \left(\frac{\sqrt[4]{1+r}}{\sqrt{r}} + \frac{1}{\sqrt{s-r}}\right) \de r\right. \\ & \qquad \qquad  \qquad \qquad  \qquad \left. + \int_{\frac{s}{2}}^s \frac{e^{-\tfrac{(s-r)^2 (t-s) (c_1-c_2)^2}{2 M (1+s) (1+t)}}}{\sqrt{1+s}(1+s)} \left(\frac{\sqrt[4]{1+s}}{\sqrt{s}} + \frac{1}{\sqrt{s-r}}\right)\de r \right) \de s\\
&\ \leq C\frac{e^{-\tfrac{(x+c_1t)^2}{M(1+t)}}}{\sqrt{1+t}} \int_0^t \left(\frac{e^{-\tfrac{s^2 (t-s) (c_1-c_2)^2}{8 M (1+s) (1+t)}}}{\sqrt{1+s}} + \frac{\sqrt{1+t}}{(1+s)^{3/4} \sqrt{s}\sqrt{t-s}} + \frac{\sqrt[4]{1+t}}{(1+s)^{5/4}\sqrt[4]{t-s}}\right) \de s \leq \frac{Ce^{-\tfrac{(x+c_1t)^2}{M(1+t)}}}{\sqrt{1+t}}.
\end{split}\label{estJ6}
\end{align}

\subsubsection{Final nonlinear estimate}
Finally, we are in the position to bound the nonlinear term in~\eqref{Duhamel}. By~\eqref{estI3},~\eqref{estI5},~\eqref{estI4} and~\eqref{estI6} the convolutions in~\eqref{Duhamel} are bounded by
\begin{align*}\left|\int_0^t  \int_{\R} \left(\theta(x-y,t-s) f_1(z(y,s)) + \theta_x(x-y,t-s)g_1(z(y,s))\right)\de y\de s\right| \leq C\eta(t)^2\left(I_* + I_0 + \sum_{i = 1}^6 J_i\right).\end{align*}
We bound the terms $I_*,I_0, J_i, i = 1,\ldots,6$ by~\eqref{estM2},~\eqref{estM4},~\eqref{estJ1},~\eqref{estJ2},~\eqref{estJ3},~\eqref{estJ4},~\eqref{estJ53} and~\eqref{estJ6}, respectively, and establish
\begin{align}\begin{split} &\left|\int_0^t  \int_{\R} \left(\theta(x-y,t-s) f_1(z(y,s)) + \theta_x(x-y,t-s)g_1(z(y,s))\right)\de y\de s\right|\\
&\qquad \leq C\eta(t)^2\left[\frac{e^{-\tfrac{(x+c_1t)^2}{M(1+t)}}}{\sqrt{1+t}} + \int_0^t \frac{e^{-\tfrac{\left(x + t c_1 + s (c_2 - c_1)\right)^2}{M(1+t)}}}{\sqrt{1+t} (1+s)}\left(\frac{\sqrt[4]{1+s}}{\sqrt{s}} + \frac{1}{\sqrt{t-s}}\right)\de s\right].\end{split}\label{irrelevantNL} \end{align}

\subsection{Conclusion}

For initial conditions $(u_0,v_0) \in X_{\rho_E}^\alpha$ and $t \in [0,T)$ such that $\eta(t) \leq r_0$, estimates~\eqref{linexp} and~\eqref{irrelevantNL} on the linear and nonlinear terms in~\eqref{Duhamel} provide the bound
\begin{align*}
\left|u(x,t)\right|\left[\frac{e^{-\tfrac{(x+c_1t)^2}{M(1+t)}}}{\sqrt{1+t}} + \int_0^t \frac{e^{-\tfrac{\left(x + t c_1 + s (c_2 - c_1)\right)^2}{M(1+t)}}}{\sqrt{1+t} (1+s)}\left(\frac{\sqrt[4]{1+s}}{\sqrt{s}} + \frac{1}{\sqrt{t-s}}\right)\de s\right]^{-1} \leq C\left(\delta + \eta(t)^2\right),
\end{align*}
for $x \in \R$. Analogously, one establishes for the other component
\begin{align*}
\left|v(x,t)\right|\left[\frac{e^{-\tfrac{(x+c_2t)^2}{M(1+t)}}}{\sqrt{1+t}} + \int_0^t \frac{e^{-\tfrac{\left(x + t c_2 + s (c_1 - c_2)\right)^2}{M(1+t)}}}{\sqrt{1+t} (1+s)}\left(\frac{\sqrt[4]{1+s}}{\sqrt{s}} + \frac{1}{\sqrt{t-s}}\right)\de s\right]^{-1} \leq C\left(\delta + \eta(t)^2\right),
\end{align*}
and we conclude that~\eqref{etaest3} holds, which proves, as explained in~\S\ref{sec:planIRR}, Theorem~\ref{mainresult2}. $\hfill \Box$

\section{Proof of Theorem~\ref{mainresult3}} \label{sec:proofgrowth}

Take initial conditions $(u_0,v_0) \in C^{0,\alpha}(\R,\R^2) \setminus \{0\}$ such that $u_0(x),v_0(x) \geq 0$ for all $x \in \R$. Without loss of generality, we assume $u_0 \neq 0$. By Proposition~\ref{proplocal} there exists a classical solution $(u,v) \in C^{1,\frac{\alpha}{2}}\left([0,T),C^{2,\alpha}(\R,\R^2)\right)$ to~\eqref{CAS2} on some maximal time interval $[0,T)$, with $T \in (0,\infty]$, having initial condition $(u_0,v_0)$.  We integrate~\eqref{CAS2} and obtain the Duhamel formulation
\begin{align}
u(x,t) &= \int_{\R} \frac{e^{-\tfrac{(x-y+c_1 t)^2}{4 d_1 t}}}{\sqrt{4\pi d_1 t}} u_0(y)\de y + \int_0^t  \int_{\R} \frac{e^{-\tfrac{(x-y+c_1 (t-s))^2}{4 d_1 (t-s)}}}{\sqrt{4\pi d_1 (t-s)}} v(y,s)^2 \de y\de s, \label{DuhamelU}\\
v(x,t) &= \int_{\R} \frac{e^{-\tfrac{(x-y+c_2 t)^2}{4 d_2 t}}}{\sqrt{4\pi d_2 t}} v_0(y)\de y + \int_0^t  \int_{\R} \frac{e^{-\tfrac{(x-y+c_2 (t-s))^2}{4 d_2 (t-s)}}}{\sqrt{4\pi d_2 (t-s)}} u(y,s)^2 \de y\de s, \label{DuhamelV}
\end{align}
for $x \in \R$ and $t \in [0,T)$.

We exploit that~\eqref{CAS2} is translational invariant in time and space, i.e.~that $(u,v)(\cdot+X,\cdot+T)$ is also a solution to~\eqref{CAS2} for each fixed $X, T \in \R$, to relax our assumption on the initial datum $u_0$. First, due to translational invariance in space and since $u_0 \neq 0$ is pointwise nonnegative and continuous, there exist, without loss of generality, some $r,\nu > 0$ such that $u_0(x) \geq \nu$ for $|x| \leq r$. Feeding this lower bound into~\eqref{DuhamelU} yields
\begin{align*}
u(x,t_0) \geq \nu \int_{-r}^r \frac{e^{-\tfrac{(x-y+c_1 t_0)^2}{4 d_1 t_0}}}{\sqrt{4\pi d_1 t_0}} \de y \geq \nu e^{-\tfrac{x^2}{2d_1 t_0}} \int_{-r}^r \frac{e^{-\tfrac{(y-c_1t)^2}{2d_1 t_0}}}{\sqrt{4\pi d_1 t_0}} \de y,
\end{align*}
for $x \in \R$ and $t_0 \in (0,T)$. Hence, there exist ($t_0$-dependent) constants $\nu_0,\alpha > 0$ such that $u(x,t_0) \geq \nu_0 e^{-\alpha x^2}$ for all $x \in \R$. Thus, due to translational invariance in time, we may, without loss of generality, assume $u_0(x) \geq \nu_0e^{-\alpha x^2}$ for all $x \in \R$.

We denote $d_m := \min\{d_1,d_2\} > 0$. Feeding this lower bound into~\eqref{DuhamelU} and integrating with the aid of~\eqref{integralid1} yields
\begin{align*}
u(x,t) \geq \frac{\nu_0 \sqrt{d_m}}{\sqrt{d_1}} \int_\R \frac{e^{-\tfrac{(x-y+c_1 t)^2}{4 d_m t} - \alpha y^2}}{\sqrt{4\pi d_m t}} \de y = \frac{\nu_0 \sqrt{d_m} e^{-\tfrac{\alpha (x + c_1 t)^2}{1 + 4 \alpha d_m t}}}{\sqrt{d_1} \sqrt{1+ 4 \alpha d_m t}},
\end{align*}
for $x \in \R$ and $t \in [0,T)$. Substituting the latter into~\eqref{DuhamelV} and employing~\eqref{integralid1} again gives
\begin{align}
v(x,t) \geq \frac{\nu_0^2 d_m^{3/2}}{d_1 \! \sqrt{d_2}} \int_0^t \int_\R \frac{e^{-\tfrac{(x-y+c_2 (t-s))^2}{2 d_m (t-s)} - 2\tfrac{\alpha (y+c_1s)^2}{1+4\alpha d_m s}}\de y \de s}{\sqrt{4\pi d_m(t-s)}(1+4\alpha d_m s)} = \frac{\nu_0^2 d_m^{3/2}}{d_1\!\sqrt{2 d_2}} \int_0^t \frac{e^{-\tfrac{2 \alpha \left(x + c_2 t + s (c_1-c_2)\right)^2}{1 + 4 \alpha d_m t}}\de s}{\!\sqrt{(1 + 4 \alpha d_m s) (1 + 4 \alpha d_m t)}}, \label{LB1}
\end{align}
for $x \in \R$ and $t \in [0,T)$. We apply Jensen's inequality to~\eqref{DuhamelU} and establish
\begin{align*}
u(x,t) \geq \int_0^t \left(\int_\R \frac{e^{-\tfrac{(x-y+c_1 (t-s))^2}{4 d_1 (t-s)} - \alpha y^2}}{\sqrt{4\pi d_1 (t-s)}}  v(y,s)  \de y\right)^2 \de s,
\end{align*}
Plugging the lower bound~\eqref{LB1} into the latter we use~\eqref{integralid1} again to evaluate the integral over $y$ and obtain
\begin{align}
\begin{split}
u(x,t) &\geq \frac{\nu_0^4 d_m^4}{2 d_1^3 d_2} \int_0^t \left( \int_0^s \int_\R \frac{e^{-\tfrac{(x-y+c_1 (t-s))^2}{2 d_m (t-s)} - \frac{2 \alpha \left(x + c_2 s + r (c_1-c_2)\right)^2}{1 + 4 \alpha d_m s}}}{\sqrt{4\pi d_m (t-s) (1 + 4 \alpha d_m r) (1 + 4 \alpha d_m s)} } \de y \de r \right)^2 \de s\\
&\geq \frac{\nu_0^4 d_m^4}{4 d_1^3 d_2} \int_0^t \left(\int_0^s \frac{e^{-\tfrac{2 \alpha (x + c_1 t + (s-r)(c_2 - c_1))^2}{1+4 \alpha d_m t}}}{\sqrt{(1 + 4 \alpha d_m s) (1 + 4 \alpha d_m t)}} \de r \right)^2 \de s,
\end{split} \label{keyLB}
\end{align}
for $x \in \R$ and $t \in [0,T)$.

\paragraph*{Lower bound in the $L^1$-norm.} We establish the lower bound in~\eqref{LB} in the $L^1$-norm. By~\eqref{keyLB} and~\eqref{integralid1} it holds
\begin{align*}
\|u(\cdot,t)\|_1 &\geq \frac{\nu_0^4 d_m^4}{4 d_1^3 d_2} \int_\R \int_0^t \int_0^s \int_0^s \frac{e^{-\tfrac{2 \alpha (x + c_1 t + (s-r)(c_2 - c_1))^2}{1+4 \alpha d_m t}-\tfrac{2 \alpha (x + c_1 t + (s-p)(c_2 - c_1))^2}{1+4 \alpha d_m t}}}{(1 + 4 \alpha d_m s) (1 + 4 \alpha d_m t)} \de r \de p \de s \de x\\
&= \frac{\nu_0^4 d_m^4 \sqrt{\pi}}{8 d_1^3 d_2\sqrt{\alpha}} \int_0^t \int_0^s \int_0^s \frac{e^{-\tfrac{\alpha (c_1-c_2)^2 (r-p)^2}{1+4 \alpha d_m t}}}{(1+ 4 \alpha d_m s) \sqrt{1+4 \alpha d_m t}}\de r \de p \de s =: J_*,
\end{align*}
for $t \in [0,T)$. So, if $c_1 \neq c_2$, we use that $z\text{erf}(z) + \pi^{-1/2} e^{-z^2}$ is a primitive of the error function $\text{erf}(z)$ to compute
\begin{align}
\|u(\cdot,t)\|_1 &\geq J_* =  \frac{\nu_0^4 d_m^4 \sqrt{\pi}}{8 d_1^3 d_2 \sqrt{\alpha}} \int_0^t \frac{\left(e^{-\tfrac{\alpha (c_1-c_2)^2 s^2}{1+ 4 \alpha d_m t}} - 1\right)\sqrt{1+4\alpha d_m t} + \sqrt{\alpha \pi} |c_1-c_2| s \ \text{erf}\left(\frac{\sqrt{\alpha} |c_1-c_2| s}{\sqrt{1+4 \alpha d_m t}}\right)}{\alpha (c_1-c_2)^2(1 + 4 \alpha d_m s)} \de s \nonumber\\
&\geq \frac{\nu_0^4 d_m^4 \sqrt{\pi}}{16 d_1^3 d_2\alpha \sqrt{\alpha}}  \int_{\frac{t}{2}}^t \frac{-2\sqrt{1+4\alpha d_m t} + \sqrt{\alpha \pi} |c_1-c_2| t \ \text{erf}\left(\frac{\sqrt{\alpha} |c_1-c_2| t}{2\sqrt{1+4 \alpha d_m t}}\right)}{(c_1-c_2)^2(1 + 4 \alpha d_m s)} \de s \label{L1c1nc2}\\
&\geq \frac{\nu_0^4 d_m^4 \sqrt{\pi} \ t \left(-2\sqrt{1+4\alpha d_m t} + \sqrt{\alpha \pi} |c_1-c_2| t \ \text{erf}\left(\frac{\sqrt{\alpha} |c_1-c_2| t}{2\sqrt{1+4 \alpha d_m t}}\right)\right)}{32 d_1^3 d_2 \alpha^2 (c_1-c_2)^2 \sqrt{\alpha} (1+4\alpha d_m t)},\nonumber
\end{align}
for $t \in [0,T)$. The latter clearly yields the lower bound~\eqref{LB} in the $L^1$-norm for $c_1 \neq c_2$ as the right hand side of~\eqref{L1c1nc2} grows linearly with $t$ as $t \to \infty$. On the other hand, if $c_1 = c_2$, we have
\begin{align}
\|u(\cdot,t)\|_1 &\geq J_* \geq \frac{\nu_0^4 d_m^4\sqrt{\pi} }{8 d_1^3 d_2 \sqrt{\alpha }} \int_{\frac{t}{2}}^t \frac{s^2}{(1+ 4 \alpha d_m s) \sqrt{1+4 \alpha d_m t}} \de s
\geq \frac{\nu_0^4 d_m^4\sqrt{\pi} \ t^3}{64 d_1^3 d_2 \sqrt{\alpha} (1+4\alpha d_m t)^{3/2}}, \label{L1c1c2}
\end{align}
for $t \in [0,T)$. This proves the lower bound~\eqref{LB} in the $L^1$-norm for $c_1 = c_2$.

\paragraph*{Lower bound in the $L^\infty$-norm.}  Next we prove the lower bound in~\eqref{LB} in the $L^\infty$-norm. Thus, if $c_1 \neq c_2$, we employ~\eqref{keyLB} to estimate
\begin{align}
\begin{split}
\|u(\cdot,t)\|_\infty &\geq u(-c_1 t, t) \geq \frac{\nu_0^4 d_m^4 \pi}{32 \alpha d_1^3 d_2} \int_0^t \frac{\text{erf}\left(s \sqrt{\frac{2\alpha }{1 + 4 \alpha d_m t}} \left|c_2-c_1\right|\right)^2}{(c_1-c_2)^2(1 + 4 \alpha d_m s)} \de s\\
&\geq \frac{\nu_0^4 d_m^4 \pi}{32 \alpha d_1^3 d_2} \int_{\sqrt{t}}^t \frac{\text{erf}\left( \sqrt{\frac{2\alpha t}{1 + 4 \alpha d_m t}} \left|c_2-c_1)\right|\right)^2}{(c_1-c_2)^2(1 + 4 \alpha d_m s)} \de s\\
&= \frac{\nu_0^4 d_m^4 \pi \log\left(\frac{1 + 4\alpha t}{1+4\alpha \sqrt{t}}\right)\text{erf}\left( \sqrt{\frac{2\alpha t}{1 + 4 \alpha d_m t}} \left|c_2-c_1)\right|\right)^2}{128 \alpha^2 d_1 \sqrt{d_1 d_2} (c_1-c_2)^2},\end{split} \label{Linftyc1nc2}
\end{align}
for $t \in [0,T)$. This proves the lower bound~\eqref{LB} in the $L^\infty$-norm for $c_1 \neq c_2$ as the right hand side of~\eqref{Linftyc1nc2} grows logarithmically with $t$ as $t \to \infty$. On the other hand, if $c_1 = c_2$, then~\eqref{keyLB} yields
\begin{align}
\|u(\cdot,t)\|_\infty &\geq u(-c_1 t, t) \geq \frac{\nu_0^4 d_m^4}{4 d_1^3 d_2} \int_{\frac{t}{2}}^t \frac{s^2}{(1 + 4 \alpha d_m s) (1 + 4 \alpha d_m t)} \de s
\geq \frac{\nu_0^4 d_m^4 t^3}{32 d_1^3 d_2 (1+4\alpha d_m t)^2}, \label{Linftyc1c2}
\end{align}
for $t \in [0,T)$. This proves the bound~\eqref{LB} in the $L^\infty$-norm for $c_1 = c_2$ and concludes the proof. $\hfill \Box$

\begin{remark}
{\upshape
We expect that, as in~\cite{ESCHER}, one can establish a priori upper bounds on global solutions to~\eqref{CAS2}, which are incompatible with the lower bounds~\eqref{LB}. This would imply by Proposition~\ref{proplocal} that the solutions in Theorem~\ref{mainresult3} blow up in finite time, i.e.~it holds $T < \infty$.
}\end{remark}

\begin{remark}{
\upshape
It is open whether Theorem~\ref{mainresult3} can be extended to establish growth, possibly with other lower bounds than~\eqref{LB}, of nontrivial solutions with nonnegative initial conditions in systems of the form
\begin{align*}
\begin{split}
u_t &= d_1u_{xx} + c_1 u_x + v^p,\\
v_t &= d_2v_{xx} + c_2 v_x + u^q,
\end{split} \qquad t \geq 0, x \in \R,
\end{align*}
allowing for different velocities $c_1 \neq c_2$ and, as in~\cite{ESCHER}, for both relevant and marginal nonlinearities, i.e.~having $2 \leq p,q \leq 3$. The obtained lower bounds in the $L^1$- and $L^\infty$-norms for $p = q = 2$ are less strong when the velocities are different; indeed, compare~\eqref{L1c1nc2} with~\eqref{L1c1c2} and~\eqref{Linftyc1nc2} with~\eqref{Linftyc1c2}. On the one hand, this might indicate that our result does not extend to the case of marginal nonlinearities and different velocities. On the other hand, it could be that the case of different velocities just requires more iterations of estimates via Duhamel's formula to obtain the desired lower bounds.
}\end{remark}

\section{Proof of Theorem~\ref{mainresult4}} \label{sec:proofMR4}

In this proof, $C>1$ denotes a constant, which is independent of $\delta,x$ and $t$ and that will be taken larger if necessary.

\subsection{Plan of proof} \label{sec:planMR4}

In order to eliminate the Burgers' term $\gamma (u^2)_x$ from the $v$-equation in~\eqref{CAS3}, we introduce the comoving coordinate $\xt = x + c_1 t$ and make the normal form transform $\vt = v + \frac{\gamma}{c} u^2$, where we denote $c := c_2 - c_1 \neq 0$. Thus, any solution $(u,v)$ to~\eqref{CAS3} yields a solution
\begin{align} \ut(\xt,t) = u(\xt - c_1t,t), \qquad \vt(\xt,t) = v(\xt-c_1t,t) + \frac{\gamma}{c} u(\xt-c_1t,t)^2, \label{defutvt} \end{align}
satisfying system
\begin{align}
\begin{split}
\ut_t &= d_1\ut_{\xt\xt} + \alpha \ut\vt - \mu \ut^3,\\
\vt_t &= d_2\vt_{\xt\xt} + c \vt_\xt + \frac{\gamma}{c}\left((d_1-d_2)\left(\ut^2\right)_{\xt\xt} - 2d_1 \ut_\xt^2 + 2\alpha \ut^2 \vt - 2\mu \ut^4\right).
\end{split} \label{CAS32}
\end{align}
where $\mu := \frac{\gamma \alpha}{c} - \beta > 0$.

Proposition~\ref{proplocal} provides local existence and uniqueness on some maximal time interval $[0,T)$, with $T \in (0,\infty]$, of a classical solution $(u,v) \in C^{1,\frac{1}{4}}([0,T),C^{2,\frac{1}{2}}(\R,\R^2))$ to~\eqref{CAS3} having initial condition $(u_0,v_0) \in H^2(\R,\R^2) \subset C^{0,\frac{1}{2}}(\R,\R^2) \cap L^\infty(\R,\R^2)$ by Morrey's inequality. Clearly, this yields a solution $(\ut,\vt) \in C^{1,\frac{1}{4}}\left([0,T),C^{2,\frac{1}{2}}(\R,\R^2)\right)$ to~\eqref{CAS32} given by~\eqref{defutvt} with initial condition $(u_0,v_0 + \frac{\gamma}{c} u_0^2)$. Since $\|(u,v)(\cdot,t)\|_\infty$ blows up as $t \uparrow T$ by~\eqref{blowuplinfty}, $\|(\ut,\vt)(\cdot,t)\|_\infty$ must also blow up as $t \uparrow T$ by~\eqref{defutvt}.

In order to exploit the negative sign of the cubic term in the $\ut$-equation in~\eqref{CAS32}, we decompose the $\ut$-variable into an explicit, leading-order Gaussian term and a remainder whose integral vanishes. Thus, we define $R(\xt,t)$ by
\begin{align}
R(\xt,t) = \ut(\xt,t) - \sigma(\xt,t) \int_\R \ut(y,t) \de y, \qquad \sigma(\xt,t) = \frac{e^{-\tfrac{\xt^2}{4d_1(1+t)}}}{\sqrt{4\pi d_1(1+t)}}, \label{defR}
\end{align}
and observe that $\hat{R}(0,t) = 0$ holds for all $t \in [0,T)$, where~~$\hat{ }$~~denotes the Fourier transform
\begin{align}
\hat{g}(k) = \int_\R e^{-\mathrm{i} k\xt} g(\xt) \de \xt, \label{defFourier}
\end{align}
for $g \in L^2(\R)$.

We aim to control both $\ut(\xt,t)$ and $R(\xt,t)$ in the nonlinear iteration and show that the remainder $R(\xt,t)$ exhibits stronger temporal decay rate. Thus, we take $M_0 = \max\{8d_1,8d_2,1\}$, let $M \geq M_0$ and define the spatio-temporal weight $\eta \colon [0,T) \to [0,\infty)$ by
\begin{align*}
 \eta(t) &=\! \sup_{\begin{smallmatrix} \xt \in \R \\ 0 \leq s \leq t \end{smallmatrix}} \left(\sqrt{1+s}e^{\tfrac{\xt^2}{M(1+s)}}\left(\left|\ut(\xt,s)\right| + \sqrt{s}\left|\ut_\xt(\xt,s)\right|\right) + \log^{\frac{3}{4}}(s+2) \left(\left\|\hat{R}(\cdot,s)\right\|_\infty  + \sqrt{1+s}\left\|\hat{R}(\cdot,s)\right\|_1\right) \hspace{-0.2cm}\phantom{\frac{e^{\tfrac{\xt^2}{M}}}{\sqrt{s}}}\!\right.\\
 &\qquad \qquad \ \left. + \ \|\ut(\cdot,s)\|_{H^2} + \left|\vt(\xt,s)\right|\left[\frac{e^{-\tfrac{(\xt+cs)^2}{M(1+s)}}}{\sqrt{1+s}} + \int_0^s \frac{e^{-\tfrac{\left(\xt + (s-r) c)\right)^2}{M(1+s)}}}{\sqrt{1+s}\sqrt{1+r}\sqrt{r}}\left(\frac{\sqrt[8]{1+r}}{r^{3/8}} + \frac{1}{\sqrt{s-r}}\right)\de r\right]^{-1} \right).
 \end{align*}
As in the proof of Theorem~\ref{mainresult1}, see~\S\ref{sec:planM}, the result follows by showing that, if $t \in [0,T)$ is such that $\eta(t)$ is bounded by a sufficiently small constant $\eta_0 > 0$, then $\eta(t)$ satisfies the key inequality
\begin{align} \eta(t) \leq C\left(\sqrt{\delta} + \eta(t)^2\right). \label{etaest5}\end{align}
%
%

\subsection{Damping estimate}
Let $t \in [0,T)$ such that $\eta(t) \leq \eta_0$, where we will take $\eta_0 > 0$ smaller if necessary. The normal form transform introduces second derivatives to the system, which we control via an $L^2$-damping estimate. Assume the solution $\phi(s) := (\ut,\vt)(s)$ to~\eqref{CAS32} lies in $H^4(\R,\R^2)$ for $s \in [0,T)$. Taking the $L^2$-inner product on both sides in~\eqref{CAS32} with $\psi(s) = (\psi_1,\psi_2)(s) = \sum_{j = 0}^2 (-1)^j \partial_\xt^{2j} \phi(s)$ and integrating by parts yields
\begin{align}
\begin{split}
\frac{1}{2} \frac{\de}{\de s} \left(\|\phi(s)\|_{H^2}^2\right) = \langle \phi_s,\psi \rangle_2 &\leq -\theta \left\|\partial_\xt^3 \phi\right\|_2^2 + C\|\phi\|_{H^2}\|\phi\|_{H^3}\\
&\qquad + C\left[\left|\langle \partial_\xt^2 (\ut^2), \psi_2 \rangle_2\right| + \left|\langle \ut_\xt^2, \psi_2 \rangle_2\right| + \sum_{j = 1}^2 \sum_{\begin{smallmatrix} p, q \in \Z_{\geq 0} \\ 2 \leq p+q\leq 4\end{smallmatrix}} \left|\langle \ut^p\vt^q, \psi_j \rangle_2\right| \right],
\end{split}\label{eest1}
\end{align}
for $s \in [0,t]$, with $D = \text{diag}(d_1,d_2)$ and $\theta = \min\{d_1,d_2\} > 0$, suppressing the $s$-dependency of the right hand side. We have now estimated the inner products in~\eqref{eest1} coming from the linear terms in~\eqref{CAS32}. We proceed by bounding the inner products originating from nonlinear contributions. Using $\eta(s) \leq \eta(t) \leq \eta_0 \leq 1$ and the Sobolev embedding $H^2(\R) \hookrightarrow W^{1,\infty}(\R)$, we integrate by parts to obtain
\begin{align*}
\begin{split}
\left|\langle \ut^p\vt^q, \psi_j \rangle_2\right| &\leq C\|\phi\|_\infty^{p+q-1} \|\phi\|_{H^1}\|\phi\|_{H^3} \leq C\eta_0 \|\phi\|_{H^3}^2,\\
\left|\langle \partial_\xt^2 (\ut^2), \psi_2 \rangle_2\right| &\leq C\left(\|\ut\|_\infty\|\ut\|_{H^3} \|\phi\|_{H^3} + \|\ut\|_{W^{1,\infty}}\|\ut\|_{H^2}\|\phi\|_{H^3}\right)
\leq C\eta_0 \|\phi\|_{H^3}^2,\\
\left|\langle \ut_\xt^2, \psi_2 \rangle_2\right| &\leq C\|\ut\|_{W^{1,\infty}}\|\phi\|_{H^2}\|\phi\|_{H^3} \leq C\eta_0\|\phi\|_{H^3}^2,
\end{split}
\end{align*}
for $j = 0,1$ and $p,q \in \Z_{\geq 0}$ such that $2 \leq p+q \leq 4$. Now, we apply the latter estimates to~\eqref{eest1}, choose $\eta_0 > 0$ so small such that $C\eta_0 \leq \theta/4$ and we use Young's inequality to bound $C\|\phi\|_{H^2}\|\phi\|_{H^3} \leq \frac{\theta}{4} \|\partial_\xt^3 \phi\|_{2}^2 + \tilde{C}_1 \|\phi\|_{H^2}^2$ for some constant $\tilde{C}_1 > 1$, to infer
\begin{align}
\frac{1}{2} \frac{\de}{\de s} \left(\|\phi(s)\|_{H^2}^2\right) \leq -\theta \left\|\partial_\xt^3 \phi\right\|_2^2 + C\left(\|\phi\|_{H^2}\|\phi\|_{H^3} + \eta_0 \|\phi\|_{H^3}^2\right) \leq -\frac{\theta}{2}\left\|\partial_\xt^3 \phi\right\|_2^2 + \left(\tilde{C}_1 + C\eta_0\right) \|\phi\|_{H^2}^2, \label{eest3}
\end{align}
for $s \in [0,t]$. By Sobolev interpolation and Young's inequality it holds $(\frac{1}{2} + \tilde{C}_1 + C\eta_0)\|\phi\|_{H^2}^2 \leq \frac{\theta}{2} \|\partial_\xt^3 \phi\|_2^2 + \tilde{C}_2 \|\phi\|_2^2$ for some constant $\tilde{C}_2 > 1$. Thus,~\eqref{eest3} boils down to
\begin{align*}
\frac{1}{2} \frac{\de }{\de s} \left(\|\phi(s)\|_{H^2}^2\right)  \leq -\frac{1}{2}\left\|\phi(s)\right\|_{H^2}^2 + \tilde{C}_2 \|\phi(s)\|_2^2.
\end{align*}
for $s \in [0,t]$. We multiply by $e^s$, integrate from $0$ to $t$, establish
\begin{align*}
\|\phi(t)\|_{H^2}^2 \leq e^{-t} \|\phi(0)\|_{H^m}^2 + C \int_0^t e^{-(t-s)} \|\phi(s)\|_2^2 \de s,
\end{align*}
and we arrive at the damping estimate
\begin{align}
 \|\ut(t)\|_{H^2}, \|\vt(t)\|_{H^2} \leq C \left(\delta + \sup_{0 \leq s \leq t} \|\phi(s)\|_2\right). \label{boundderiv}
\end{align}
In fact,~\eqref{boundderiv} holds even without the assumption that the solution $(\ut,\vt)(s)$ to~\eqref{CAS32} lies in $H^4(\R,\R^2)$ for $s \in [0,T)$. Indeed, we know that it holds $(\ut,\vt) \in C^{1,\frac{1}{4}}\left([0,T),C^{2,\frac{1}{2}}(\R,\R^2)\right)$ and $H^4(\R,\R^2)$ lies dense in $C^{2,\frac{1}{2}}(\R,\R^2)$.

\subsection{Estimates on the \texorpdfstring{$\vt$}{ }-component}
We denote by $\V$ the drifting Gaussian
\begin{align*} \V(\xt,t) := \frac{e^{-\tfrac{(\xt+c t)^2}{4d_2t}}}{\sqrt{4\pi d_2t}},\end{align*}
and let $z_0 := v_0 + \frac{\gamma}{c} u_0^2$. We integrate the $\vt$-equation in~\eqref{CAS32} and obtain, after integration by parts, the Duhamel formulation
\begin{align}
\begin{split}
\vt(\xt,t) &= \int_{\R} \V(\xt-y,t) \vt_0(y)\de y + \frac{2\gamma(d_1-d_2)}{c} \int_0^t \int_\R \V_\xt(\xt-y,t-s) \ut_\xt(y,s)\ut(y,s)\de y \de s\\
&\qquad - \frac{2\gamma}{c} \int_0^t  \int_{\R} \V(\xt-y,t-s) \left(d_1(\ut_\xt(y,s))^2 - \alpha (\ut(y,s))^2 \vt(y,s) + \mu (\ut(y,s))^4 \right)\de y \de s,
\end{split} \label{Duhamelv}
\end{align}
for $\xt \in \R$ and $t \in [0,T)$. We estimate the linear and nonlinear terms in~\eqref{Duhamelv} one by one.

\paragraph*{Linear estimates.} For $(u_0,v_0) \in H^2(\R,\R^2)$ satisfying~\eqref{initcond4}, we proceed as in~\eqref{linexp} and use $M \geq 4d_2$ and $\delta \leq 1$ to estimate the linear term in~\eqref{Duhamelv} as
\begin{align}
\left|\int_{\R} \V(\xt-y,t) \vt_0(y)\de y \right| \leq C \delta \frac{e^{-\tfrac{(\xt+c t)^2}{M(1+t)}}}{\sqrt{1+t}}.  \label{linestv}
\end{align}

\paragraph*{Nonlinear estimates.} The first convolution in~\eqref{Duhamelv} can be estimated using the boundedness of $x \mapsto xe^{-x^2}$ on $\R$ and identity~\eqref{integralid1} to compute the integral over $y$, which yields
\begin{align}
\begin{split}
\left|\int_0^t \V_\xt(\xt-y,t-s) \ut_\xt(y,s)\ut(y,s)\de y \de s\right| &\leq C\eta(t)^2 \int_0^t\int_{\R} \frac{e^{-\tfrac{(\xt-y+c(t-s))^2}{M(t-s)} - \tfrac{y^2}{M(1+s)}}}{(t-s)(1+s)\sqrt{s}} \de y\de s\\
&\leq C\eta(t)^2 \int_0^t \frac{e^{-\tfrac{(\xt+c(t-s))^2}{M(1+t)}}}{\sqrt{1+t}\sqrt{t-s}\sqrt{1+s}\sqrt{s}} \de s.
\end{split} \label{vnl1}
\end{align}

For the following nonlinear term in~\eqref{Duhamelv} we first use the Sobolev embedding $H^2(\R) \hookrightarrow W^{1,\infty}(\R)$ to estimate
\begin{align*}|\ut_\xt(y,s)|^2 \leq |\ut_\xt(y,s)| \|\ut_\xt(\cdot,s)\|_{\infty}^{3/4} \|\ut(\cdot,s)\|_{H^2}^{1/4} &\leq C\eta(t)^2 \frac{e^{-\tfrac{y^2}{M(1+s)}}}{\left[(1+s)s\right]^{7/8}}, \end{align*}
for $s \in [0,t]$ and $y \in \R$. So, applying~\eqref{integralid1} and using $M \geq 4d_2$ and $\eta(t) \leq \eta_0$, we establish the bound
\begin{align}
\begin{split}
&\left|\int_0^t \V(\xt-y,t-s) \left(d_1 (\ut_\xt(y,s))^2 + \mu (\ut(y,s))^4\right)\de y \de s\right|\\
&\ \ \leq C\eta(t)^2 \int_0^t \int_{\R} \frac{e^{-\tfrac{(\xt-y+c(t-s))^2}{M(t-s)} - \tfrac{y^2}{M(1+s)}}}{\sqrt{t-s}\left[(1+s)s\right]^{7/8}} \de y\de s \leq C \eta(t)^2 \int_0^t \frac{e^{-\tfrac{(\xt+c(t-s))^2}{M(1+t)}}}{\sqrt{1+t}(1+s)^{3/8}s^{7/8}} \de s.
\end{split} \label{vnl2}
\end{align}

Using $M \geq 8d_2$, the remaining term in~\eqref{Duhamelv} is estimated by
\begin{align}
\left| \int_0^t  \int_{\R} 2\V(\xt-y,t-s) (\ut(y,s))^2 \vt(y,s) \de y \de s\right| \leq C\eta(t)^2(I_0 + K_1), \label{vnl3}
\end{align}
where
\begin{align*}
K_1 = \int_0^t \int_0^s \int_\R \frac{e^{-\tfrac{2(\xt-y+c(t-s))^2}{M(t-s)} - \tfrac{(y+c(s-r))^2}{M(1+s)}-\tfrac{y^2}{M(1+s)}}}{(1+s)^{3/2}\sqrt{t-s}\sqrt{1+r}\sqrt{r}}\left(\frac{\sqrt[8]{1+r}}{r^{3/8}} + \frac{1}{\sqrt{s-r}}\right) \de y \de r \de s,
\end{align*}
and $I_0$ is as in~\eqref{defIj} with $c_1 = c$ and $c_2 = 0$. We note that an estimate on $I_0$ has already been obtained in~\eqref{estM4}. Using~\eqref{integralid1} we calculate the inner integral in $K_1$ and obtain
\begin{align*}
\begin{split}
K_1 &\leq C \int_0^t \int_r^t \frac{e^{-\tfrac{\left(\xt + c(t - r)\right)^2}{M(1+t)} - \tfrac{\left(\xt+ c(t - s)\right)^2}{M (1+t)}-\tfrac{(s-r)^2 (t - s) c^2}{4 M (1+s) (1+t)}}}{\sqrt{1+t}(1+s)\sqrt{1+r}\sqrt{r}}\left(\frac{\sqrt[8]{1+r}}{r^{3/8}} + \frac{1}{\sqrt{s-r}}\right) \de s \de r\\
&\leq C\int_0^t \frac{e^{-\tfrac{\left(\xt + c(t - r)\right)^2}{M(1+t)}}}{\sqrt{1+t}(1+r)^{3/8} r^{7/8}} \left(\int_r^{\max\{r,\frac{t}{2}\}} \frac{e^{-\tfrac{(s-r)^2 t c^2}{8 M (1+s) (1+t)}}}{\sqrt{1+s}} \de s + \int_{\frac{t}{2}}^t \frac{1}{1+t} \de s + \int_r^t \frac{(1+s)^{-3/4}}{\sqrt{s-r}} \de s\right) \de r.
\end{split}
\end{align*}
We use $c \neq 0$ and~\eqref{integralID4} to estimate the first integral over $s$ in the above for $t \geq 1$ (for $0 \leq t \leq 1$ the estimate is trivial) and establish
\begin{align} K_1 &\leq C\int_0^t \frac{e^{-\tfrac{\left(\xt + c(t - r)\right)^2}{M(1+t)}}}{\sqrt{1+t}(1+r)^{3/8} r^{7/8}} \de r. \label{vnl4}\end{align}

\paragraph*{Final estimate.} Finally, applying the estimates~\eqref{estM4},~\eqref{linestv},~\eqref{vnl1},~\eqref{vnl2},~\eqref{vnl3} and~\eqref{vnl4} on the linear and nonlinear terms in~\eqref{Duhamelv} we establish
\begin{align}
|\vt(\xt,t)| \left[\frac{e^{-\tfrac{(\xt+ct)^2}{M(1+t)}}}{\sqrt{1+t}} + \int_0^t \frac{e^{-\tfrac{\left(\xt + (t-s) c)\right)^2}{M(1+t)}}}{\sqrt{1+t}\sqrt{1+s}\sqrt{s}}\left(\frac{\sqrt[8]{1+s}}{s^{3/8}} + \frac{1}{\sqrt{t-s}}\right)\de s\right]^{-1} \leq C\left(\delta + \eta(t)^2\right). \label{finalv}
\end{align}

\subsection{Estimates on the \texorpdfstring{$\ut$}{ }-component}

As explained in~\S\ref{sec:planMR4}, we control the cubic term in~\eqref{CAS32} by decomposing the $\ut$-variable into an explicit, leading-order Gaussian term and a remainder $R(\xt,t)$ whose integral vanishes, see~\eqref{defR}. Applying the Fourier transform~\eqref{defFourier} to~\eqref{defR} yields
\begin{align}
\hat{R}(k,t) = \hat{\ut}(k,t) - \hat{\sigma}(k,t) A(t), \qquad k \in \R, \label{defRFourier}
\end{align}
with $A(t) := \hat{\ut}(0,t)$ and $\hat{\sigma}(k,t) := e^{-d_1k^2 (t+1)}$. The $\ut$-equation in~\eqref{CAS32} reads in Fourier space
\begin{align}
\hat{\ut}_t &= -d_1k^2 \hat{\ut} + \frac{\alpha}{2\pi} \ \hat{\ut} \ast \hat{\vt} - \frac{\mu}{4\pi^2} \ \hat{\ut}^{\ast 3}.\label{utFourier}
\end{align}
Using~\eqref{defRFourier} and~\eqref{utFourier}, we derive the following equations for $A(t)$ and $\hat{R}(k,t)$
\begin{align}
\frac{\de A}{\de t}(t)   &= -\frac{\nu}{1+t} A(t)^3 + g(t) + B(0,t), \label{Aeq}\\
\frac{\de \hat{R}}{\de t}(k,t) &= -d_1k^2 \hat{R}(k,t) + N(k,t) - N(0,t)\hat{\sigma}(k,t) + B(k,t) - B(0,t)\hat{\sigma}(k,t), \label{Req}
\end{align}
for $k \in \R$, with $\nu := \frac{\mu}{4\sqrt{3}d_1 \pi} > 0$ and
\vspace{-0.2cm}
\begin{align*}
g(t) &:= -\frac{\mu}{4\pi^2} \left[\hat{R} \ast \left(3A \left(\hat{\sigma} \ast \hat{R}\right) + 3A^2 \hat{\sigma}^{\ast 2} + \hat{R}^{\ast 2}\right)\right](0,t),\\
N(k,t) &:= -\frac{\mu}{4\pi^2} \ \hat{\ut}^{*3}(k,t), \qquad B(k,t) := \frac{\alpha}{2\pi} \left(\hat{\ut} \ast \hat{\vt}\right)(k,t).
\end{align*}
In the following we bound the $\ut$-component by estimating the leading order part $\sigma(\xt,t)A(t)$ and the remainder $R(\xt,t)$ in~\eqref{defR} separately. By exploiting the negative sign of $\nu$ in the nonlinear ODE~\eqref{Aeq}, we derive that $A(t)$ decays logarithmically over time, whereas the fact that the nonlinearity in~\eqref{Req} vanishes at $k = 0$ yields an additional algebraic decay factor for the remainder term $R(\xt,t)$. All in all, we gain a logarithmic decay factor for $\ut(\xt,t)$ which is enough to control the cubic term $-\mu \ut^3$ in~\eqref{CAS32}.

\subsubsection{Analysis of the \texorpdfstring{$A$}{ }-equation}  Define the set
\begin{align*}
 \Es = \left\{s \in [0,T) : |A(s)| \sqrt{\left(\delta + \eta(s)^2\right)^{-2} + 2 \nu \log(1+s)} \leq 1\right\}.
\end{align*}
Clearly, $0 \in \Es$ as $|A(0)| \leq \|\hat{\ut}(\cdot,0)\|_\infty \leq \|\ut(\cdot,0)\|_1 = \|u_0\|_1 \leq \delta$ by~\eqref{initcond4}. We aim to show that, in fact, we have $t \in \Es$, provided $\eta_0 > 0$ sufficiently small. Assume by contradiction that $t \in [0,T) \setminus \Es$ and let $t_1 = \inf\{s \in [0,T) : \text{ for all } r \in [s,t] \text{ we have } r \notin \Es\}$. Then, by continuity of $A(t)$ (recall $A(t) = \int_\R \ut(\xt,t)\de \xt$ and $\ut \in C^{1,\frac{\alpha}{2}}([0,T),C^{2,\alpha}(\R))$), it holds
\begin{align} 0 < A(s)^{-2} \leq \left(\delta + \eta(s)^2\right)^{-2} + 2 \nu \log(1+s), \qquad A(t_1)^{-2} = \left(\delta + \eta(t_1)^2\right)^{-2} + 2 \nu \log(1+t_1), \label{Aid}
\end{align}
for $s \in [t_1,t]$. So, by~\eqref{Aeq} we obtain for all $s \in [t_1,t]$
\begin{align*}
\frac{1}{2} \frac{\de }{\de s}\left(A(s)^{-2}\right) &= \frac{\nu}{1+s} - A(s)^{-3} g(s) - A(s)^{-3} B(0,s).
\end{align*}
Integrating the latter from $t_1$ to $t$ and using~\eqref{Aid} yields
\begin{align}
A(t)^{-2} = \left(\delta + \eta(t_1)^2\right)^{-2} + 2\nu \log(t+1) - 2\int_{t_1}^t A(s)^{-3} \left(g(s) + B(0,s)\right) \de s. \label{AeqINT}
\end{align}
To bound the integral over $A(s)^{-3}g(s)$ in~\eqref{AeqINT} we first estimate using Young's convolution inequality
\begin{align*}
|\hat{R}^{\ast 3}(0,s)| &\leq \|\hat{R}^{\ast 3}(\cdot,s)\|_\infty \leq \|\hat{R}(\cdot,s)\|_1^2 \|\hat{R}(\cdot,s)\|_\infty,\\
|(\hat{\sigma}^{\ast j} \ast \hat{R}^{\ast (3-j)})(0,s)| &\leq \|\hat{R}(\cdot,s)\|_1^{3-j} \|\hat{\sigma}(\cdot,s)\|_\infty \|\hat{\sigma}(\cdot,s)\|_1^{j-1} \leq (1+s)^{-(j-1)/2} \|\hat{R}(\cdot,s)\|_1^{3-j}, \qquad j = 1,2,
\end{align*}
for $s \in [t_1,t]$. Thus, using $\eta(t) \leq \eta_0 \leq 1$ and identity~\eqref{Aid}, we obtain
\begin{align}\begin{split}
\int_{t_1}^t &\left|A(s)^{-3} \hat{R}^{\ast 3}(0,s)\right| \de s \leq C \int_{t_1}^t \eta(s)^3 \frac{\left(\left(\delta + \eta(s)^2\right)^{-2} + \log(1+s)\right)^{3/2}}{(1+s)\log^{\frac{9}{4}}(s + 2)} \de s\\
&\qquad \leq C\eta(t)\left(\left(\delta + \eta(t_1)^2\right)^{-2} + \log(1+t)\right) \int_{t_1}^t \frac{\left(\delta + \eta(s)^2\right)^{-1} \eta(s)^2 + \log^{\frac{1}{2}}(1+s)}{(1+s)\log^{\frac{9}{4}}(s + 2)} \de s \\
&\qquad \leq C\eta_0\left(\left(\delta + \eta(t_1)^2\right)^{-2} + 2\nu\log(1+t)\right).
\end{split}\label{Aest1}\end{align}
Similarly, we bound
\begin{align}\begin{split}
\int_{t_1}^t \left|A(s)^{-2} \left(\hat{\sigma} \ast \hat{R}^{\ast 2}\right)(0,s)\right| \de s &\leq C \int_{t_1}^t \eta(s)^2 \frac{\left(\left(\delta + \eta(s)^2\right)^{-2} + \log(1+s)\right)}{(1+s)\log^{\frac{3}{2}}(s + 2)} \de s\\
&\leq C\eta_0\left(\left(\delta + \eta(t_1)^2\right)^{-2} + 2\nu\log(1+t)\right),
\end{split}\label{Aest2}\end{align}
and
\begin{align}\begin{split}
\int_{t_1}^t \left|A(s)^{-1} \left(\hat{\sigma}^{\ast 2} \ast \hat{R}\right)(0,s)\right|& \de s \leq C \int_{t_1}^t \eta(s) \frac{\sqrt{\left(\delta + \eta(s)^2\right)^{-2} + \log(1+s)}}{(1+s)\log^{\frac{3}{4}}(s + 2)} \de s\\
&\leq C\eta_0\left(\left(\delta + \eta(t_1)^2\right)^{-2} + 2\nu\log(1+t)\right).
\end{split}\label{Aest3}\end{align}

Recall that we derived~\eqref{finalv} for all $t \in [0,T)$ such that $\eta(t) \leq \eta_0$. For $s \in [t_1,t]$ it holds $\eta(s) \leq \eta(t) \leq \eta_0$ and, hence, we can use~\eqref{finalv} to estimate the integral over $A(s)^{-3}B(0,s)$ in~\eqref{AeqINT}:
\begin{align*}
\left|B(0,s)\right| &\leq \left\|B(\cdot,s)\right\|_{\infty} \leq C\left\|\left(\ut \cdot \vt\right)(\cdot,s)\right\|_{1}\\
 &\leq C \eta(s) \left(\delta + \eta(s)^2\right) \int_\R \frac{e^{-\tfrac{\xt^2}{M(1+s)}}}{1+s} \left(e^{-\tfrac{(\xt+cs)^2}{M(1+s)}} + \int_0^s \frac{e^{-\tfrac{\left(\xt + (s-r) c\right)^2}{M(1+s)}}}{\sqrt{1+r}\sqrt{r}}\left(\frac{\sqrt[8]{1+r}}{r^{3/8}} + \frac{1}{\sqrt{s-r}}\right) \de r\right)\\
 &\leq C \eta(s) \left(\delta + \eta(s)^2\right) \left(\frac{e^{-\frac{c^2 s^2}{2M(1+s)}}}{\sqrt{1+s}} + \int_0^s \frac{e^{-\frac{c^2(s-r)^2}{2M(1+s)}}}{\sqrt{1+s}\sqrt{1+r}\sqrt{r}}\left(\frac{\sqrt[8]{1+r}}{r^{3/8}} + \frac{1}{\sqrt{s-r}}\right) \de r\right).
\end{align*}
We estimate the integral over $r$ in the above by splitting it in two integrals over $[0,\frac{s}{2}]$ and $[\frac{s}{2},s]$ and by subsequently using~\eqref{integralID5}. Thus, we obtain
\begin{align}
\begin{split}
\left|B(0,s)\right| &\leq \frac{C \eta(s) \left(\delta + \eta(s)^2\right)}{(1+s)^{5/4}},
\end{split} \label{mixint}
\end{align}
for $s \in [t_1,t]$.
This leads to the estimate
\begin{align}\begin{split}
\int_{t_1}^t \left|A(s)^{-3} B(0,s)\right| \de s &\leq C \int_{t_1}^t \eta(s) \frac{\left(\left(\delta + \eta(s)^2\right)^{-2} + \log(1+s)\right)^{3/2}\left(\delta + \eta(s)^2\right)}{(1+s)^{5/4}} \de s\\
&\leq C\eta(t)\left(\left(\delta + \eta(t_1)^2\right)^{-2} + \log(1+t)\right) \int_{t_1}^t \frac{1 + \log^{\frac{1}{2}}(1+s)}{(1+s)^{5/4}} \de s \\
&\leq C\eta_0\left(\left(\delta + \eta(t_1)^2\right)^{-2} + 2\nu\log(1+t)\right).
\end{split}\label{Aest4}\end{align}
Combining~\eqref{Aest1},~\eqref{Aest2},~\eqref{Aest3} and~\eqref{Aest4} we can estimate the integral in~\eqref{AeqINT} and we obtain, provided $\eta_0$ is sufficiently small, the lower bound
\begin{align*}
A(t)^{-2} \geq \left(1-C\eta_0\right)\left(\left(\delta + \eta(t_1)^2\right)^{-2} + 2\nu \log(t+1)\right) \geq \frac{1}{2}\left(\left(\delta + \eta(t)^2\right)^{-2} + 2\nu \log(t+1)\right),
\end{align*}
which contradicts $t \notin \Es$. We conclude
\begin{align}
|A(t)| \sqrt{\left(\delta + \eta(t)^2\right)^{-2} + 2 \nu \log(1+t)} \leq 1. \label{Aestfin}
\end{align}

\subsubsection{Analysis of the \texorpdfstring{$R$}{ }-equation}
We integrate~\eqref{Req} and obtain the Duhamel formulation
\begin{align}
\begin{split}
\hat{R}(k,t) &= e^{-d_1k^2 t} \hat{R}(k,0) + \int_0^t e^{-d_1 k^2 (t-s)} \left(N(k,s) - N(0,s)\hat{\sigma}(k,s) + B(k,s) - B(0,s)\hat{\sigma}(k,s)\right)\de s,
\end{split} \label{DuhamelR}
\end{align}
for $k \in \R$ and $t \in [0,T)$. We estimate the linear and nonlinear terms in~\eqref{DuhamelR} one by one.

\paragraph*{Linear estimates.} First, by~\eqref{initcond4},~\eqref{defRFourier}, using $\ut(\cdot,0) = u_0$ and the Sobolev embedding $H^2(\R) \hookrightarrow \left\{f : \hat{f} \in L^1(\R)\right\}$, we obtain
\begin{align*}
\left\|\hat{R}(\cdot,0)\right\|_{1} &\leq \left\|\hat{\ut}(\cdot,0)\right\|_{1} + |A(0)|\left\|\hat{\sigma}(\cdot,0)\right\|_{1}
\leq C\left(\|u_0\|_{H^2} + \int_\R |u_0(x)| \de x \right) \leq C\delta,\\
\left\|\partial_k^j \hat{R}(\cdot,0)\right\|_{\infty} &\leq \left\|\partial_k^j \hat{\ut}(\cdot,0)\right\|_{\infty} + |A(0)|\left\|\partial_k^j \hat{\sigma}(\cdot,0)\right\|_{\infty}
\leq C\left(\int_\R \left|x^j u_0(x)\right|\de x + \int_\R |u_0(x)| \de x \right) \leq C\delta,
\end{align*}
for $j = 0,1$. We use $\hat{R}(0,0) = 0$ and employ the estimates
\begin{align*}
L_1 &\leq \int_\R |k| e^{-d_1k^2 t} \left\|\partial_k \hat{R}(\cdot,0)\right\|_\infty \de k  \leq \frac{C\delta}{t}, \qquad& L_1 &\leq \left\|\hat{R}(\cdot,0)\right\|_1 \leq C\delta,\\
L_\infty &\leq \left\|\partial_k \hat{R}(\cdot,0)\right\|_\infty \sup_{k \in \R} |k| e^{-d_1k^2 t}  \leq \frac{C\delta}{\sqrt{t}}, \qquad& L_\infty &\leq \left\|\hat{R}(\cdot,0)\right\|_\infty  \leq C\delta.
\end{align*}
Hence, the linear term in~\eqref{DuhamelR} enjoys the following bounds
\begin{align}
L_1 := \int_\R e^{-d_1k^2 t} \left|\hat{R}(k,0)\right|\de k \leq \frac{C\delta}{1+t}, \qquad L_\infty := \sup_{k \in \R} e^{-d_1k^2 t} \left|\hat{R}(k,0)\right| \leq \frac{C\delta}{\sqrt{1+t}}. \label{linRbounds}
\end{align}

\paragraph*{Nonlinear estimates.}  We start by bounding the nonlinear term
\begin{align*}
\N(k,t) := \int_0^t e^{-d_1 k^2 (t-s)} \left(N(k,s) - N(0,s)\hat{\sigma}(k,s) \right)\de s,
\end{align*}
in~\eqref{DuhamelR}. First, we estimate the integrand in $\N(k,t)$. On the one hand, we have
\begin{align*}
\left\|\hat{\ut}(\cdot,t)\right\|_1 &\leq |A(t)| \left\|\hat{\sigma}(\cdot,t)\right\|_1 + \left\|\hat{R}(\cdot,t)\right\|_1 \leq C\left(\frac{\min\left\{\eta(t)^2 + \delta, \log^{-\frac{1}{2}}(1+t)\right\}}{\sqrt{1+t}} + \frac{\eta(t)}{\sqrt{1+t}\log^{\frac{3}{4}}(t+2)}\right),
\end{align*}
by~\eqref{defRFourier} and~\eqref{Aestfin}. In addition, we establish
\begin{align*}
\left\|\partial_k^j \hat{\ut}(\cdot,t)\right\|_\infty &\leq \int_\R \left|\xt^j\ut(\xt,t)\right| \de \xt \leq C\eta(t) (1+t)^{j/2}, \qquad j = 0,1.
\end{align*}
Thus, with the aid of the mean value theorem and Young's convolution inequality we bound the integrand in $\N(k,t)$ as
\begin{align}
\begin{split}
&\left|N(k,s) - N(0,s)\hat{\sigma}(k,s)\right| \leq \left|N(k,s) - N(0,s)\right| + \left(1-e^{-d_1 k^2 (1+s)}\right) \left|N(0,s)\right|\\
&\leq C|k|\left(\|\partial_k N(\cdot,s)\|_\infty + \sqrt{1+s} \|N(\cdot,s)\|_\infty\right)\leq C\left\|\hat{\ut}(\cdot,s)\right\|_1^2 |k| \left(\left\|\partial_k \hat{\ut}(\cdot,s)\right\|_\infty + \sqrt{1+s} \left\|\hat{\ut}(\cdot,s)\right\|_\infty\right)\\
&\leq C|k| \eta(s) \frac{\sqrt{\delta} + \eta(s)}{\sqrt{1+s} \log^{\frac{3}{4}}(1+s)},
\end{split} \label{ANL}
\end{align}
for $s \in [0,t]$. Using $\eta(t) \leq \eta_0 \leq 1$, we also establish the second bound
\begin{align}
\begin{split}
\left|N(k,s) - N(0,s)\hat{\sigma}(k,s)\right| &\leq C\left\|N(\cdot,s)\right\|_\infty \leq C \left\|\hat{\ut}(\cdot,s)\right\|_\infty \left\|\hat{\ut}(\cdot,s)\right\|_1^2 \leq C \eta(s) \frac{\sqrt{\delta} + \eta(s)}{(1+s) \log^{\frac{3}{4}}(1+s)},
\end{split}\label{BNL}
\end{align}
for $s \in [0,t]$, and the third bound
\begin{align}
\begin{split}
\left|N(k,s) - N(0,s)\hat{\sigma}(k,s)\right| &\leq C\left\|N(\cdot,s)\right\|_\infty \leq C \left\|\ut(\cdot,s)^3\right\|_1 \leq \frac{C \eta(s)^2}{1+s}.
\end{split}\label{CNL}
\end{align}

We are now in a good position to estimate $\N(k,t)$. First, for $t \geq 1$, we establish with the aid of~\eqref{ANL} and~\eqref{BNL}
\begin{align*}
\|\N(\cdot,t)\|_1 &\leq C \eta(t)\left(\sqrt{\delta} + \eta(t)\right) \left(\int_0^{\frac{t}{2}} \int_\R \frac{|k|e^{-d_1 k^2 (t-s)}}{\sqrt{1+s}\log^{\frac{3}{4}}(1+s)} \de k\de s +
\int_{\frac{t}{2}}^t \int_\R \frac{e^{-d_1 k^2 (t-s)}}{(1+s)\log^{\frac{3}{4}}(1+s)} \de k \de s\right)\\
&\leq C \eta(t)\left(\sqrt{\delta} + \eta(t)\right) \left(\int_0^{\frac{t}{2}} \frac{1}{(1+t) \sqrt{1+s}\log^{\frac{3}{4}}(1+s)} \de s +
\int_{\frac{t}{2}}^t \frac{1}{\sqrt{t-s}(1+t)\log^{\frac{3}{4}}(1+t)} \de s\right)\\
&\leq C\frac{\sqrt{\delta} + \eta(t)^2}{\sqrt{1+t}\log^{\frac{3}{4}}(2+t)},
\end{align*}
and, similarly, we obtain
\begin{align*}
\|\N(\cdot,t)\|_\infty &\leq C \left(\sqrt{\delta} + \eta(t)^2\right) \left(\int_0^{\frac{t}{2}} \frac{\sup_{k \in \R} \left|ke^{-d_1 k^2 (t-s)}\right|}{\sqrt{1+s}\log^{\frac{3}{4}}(1+s)} \de s +
\int_{\frac{t}{2}}^t \frac{\sup_{k \in \R} \left|e^{-d_1 k^2 (t-s)}\right|}{(1+s)\log^{\frac{3}{4}}(1+s)} \de s\right) \leq C\frac{\sqrt{\delta} + \eta(t)^2}{\log^{\frac{3}{4}}(2+t)}.
\end{align*}
On the other hand, for $t \leq 1$,~\eqref{CNL} yields the bounds
\begin{align*}
\|\N(\cdot,t)\|_1 &\leq C\eta(t)^2 \int_0^t \int_\R \frac{e^{-d_1 k^2 (t-s)}}{1+s} \de k\de s \leq C\eta(t)^2,\\
\|\N(\cdot,t)\|_\infty &\leq C\eta(t)^2 \int_0^t \frac{\sup_{k \in \R} e^{-d_1 k^2 (t-s)}}{1+s} \de s \leq C\eta(t)^2.
\end{align*}
We conclude that it holds
\begin{align}
\|\N(\cdot,t)\|_1 &\leq C\frac{\sqrt{\delta} + \eta(t)^2}{\sqrt{1+t}\log^{\frac{3}{4}}(2+t)}, \qquad & \|\N(\cdot,t)\|_\infty &\leq C\frac{\sqrt{\delta} + \eta(t)^2}{\log^{\frac{3}{4}}(2+t)}. \label{NLR1}
\end{align}

We proceed with the estimation of the remaining nonlinear term
\begin{align*}
\B(k,t) := \int_0^t e^{-d_1 k^2 (t-s)} \left(B(k,s) - B(0,s)\hat{\sigma}(k,s) \right)\de s,
\end{align*}
in~\eqref{DuhamelR}. We use~\eqref{mixint} and
\begin{align*}
\left\|\partial_k B(\cdot ,s)\right\|_\infty \leq C\int_\R \left|\xt\ut(\xt,s)\vt(\xt,s)\right| \de \xt \leq C\int_\R \left|\xt\ut(\xt,s)\right| \de \xt \left\|\vt(\cdot,s)\right\|_\infty \leq C \eta(s)^2,
\end{align*}
for $s \in [0,t]$, to bound the integrand in $\B(k,t)$ as
\begin{align}
\begin{split}
&\left|B(k,s) - B(0,s)\hat{\sigma}(k,s)\right| \leq C\|B(\cdot,s)\|_\infty^{3/4} \left(\left|B(k,s) - B(0,s)\right| + \left(1-e^{-d_1 k^2 (1+s)}\right) \left|B(0,s)\right|\right)^{1/4}\\
&\qquad \leq C\|B(\cdot,s)\|_\infty^{3/4} \left(|k| \|\partial_k B(\cdot,s)\|_\infty + |k|\sqrt{1+s} \|B(\cdot,s)\|_\infty\right)^{1/4} \leq C\frac{\eta(s)^2 \sqrt[4]{|k|} }{(1+s)^{15/16}}.
\end{split}\label{BBNL}
\end{align}
Thus, applying~\eqref{BBNL} we bound
\begin{align}
\|\B(\cdot,t)\|_1 \leq C\eta(t)^2 \int_0^t \int_\R \frac{\sqrt[4]{|k|} e^{-d_1k^2 (t-s)}}{(1+s)^{15/16}} \de k \de s \leq  \int_0^t \frac{C\eta(t)^2\de s}{(t-s)^{5/8}(1+s)^{15/16}}\leq \frac{C\eta(t)^2}{(1+t)^{9/16}}, \label{NLR2}
\end{align}
and, similarly, we obtain
\begin{align}
\|\B(\cdot,t)\|_\infty \leq C\eta(t)^2 \int_0^t \frac{\sup_{k \in \R} \sqrt[4]{|k|} e^{-d_1k^2 (t-s)}}{(1+s)^{15/16}} \de s \leq \frac{C\eta(t)^2}{(1+t)^{1/16}}. \label{NLR3}
\end{align}

\paragraph*{Final estimate.} Finally, combining~\eqref{DuhamelR},~\eqref{linRbounds},~\eqref{NLR1},~\eqref{NLR2} and~\eqref{NLR3} we conclude
\begin{align}
\log^{\frac{3}{4}}(t+2) \left(\left\|\hat{R}(\cdot,t)\right\|_\infty  + \sqrt{1+t}\left\|\hat{R}(\cdot,t)\right\|_1\right) \leq C\left(\sqrt{\delta} + \eta(t)^2\right). \label{finalr}
\end{align}

\subsubsection{Analysis of the \texorpdfstring{$\ut$}{ }-equation}
Denote by $\U$ denotes the drifting Gaussian
\begin{align*} \U(\xt,t) := \frac{e^{-\tfrac{\xt^2}{4d_1t}}}{\sqrt{4\pi d_1t}}.\end{align*}
We integrate the $\ut$-equation in~\eqref{CAS32} and obtain the Duhamel formulation
\begin{align}
\begin{split}
\partial_\xt^j \ut(\xt,t) &= \int_{\R} \partial_\xt^j \U(\xt\!-\!y,t) u_0(y)\de y + \int_0^t \int_\R \partial_\xt^j \U(\xt\!-\!y,t\!-\!s) \left(\alpha \ut(y,s)\vt(y,s) \!-\! \mu \ut(y,s)^3\right) \de y \de s,\\
\end{split} \label{Duhamelu}
\end{align}
for $j=0,1$, $\xt \in \R$ and $t \in [0,T)$. We estimate the linear and nonlinear terms in~\eqref{Duhamelu} one by one.

\paragraph*{Linear estimates.} For $(u_0,v_0) \in H^2(\R,\R^2)$ satisfying~\eqref{initcond4}, we proceed as in~\eqref{linexp} and estimate the the linear term in~\eqref{Duhamelu} by
\begin{align}
\left|\int_{\R} \partial_\xt^j \U(\xt-y,t) u_0(y)\de y \right| \leq C \delta \frac{e^{-\tfrac{\xt^2}{M(1+t)}}}{t^{j/2}\sqrt{1+t}}, \label{linestu}
\end{align}
for $j = 0, 1$.

\paragraph*{Nonlinear estimates.} We first consider the cubic term in~\eqref{Duhamelu}. By~\eqref{defR}, the cubic term can be expanded as
\begin{align*}
\ut(y,s)^3 = \sigma(y,s)^3 A(s)^3 + 3\sigma(y,s)^2 A(s)^2 R(y,s) + 2 \sigma(y,s) A(y,s) R(y,s)^2 + R(y,s)^2 \ut(y,s).
\end{align*}
for $s \in [0,t]$ and $y \in \R$. Using~\eqref{Aestfin}, $\|R(\cdot,s)\|_\infty \leq \|\hat{R}(\cdot,s)\|_1$, $M \geq 4d_1$ and $\eta(t) \leq \eta_0 \leq 1$, we establish the following pointwise bounds for each term in the above expansion:
\begin{align*}
\left|\sigma(y,s)^l A(s)^l R(y,s)^{3-l}\right| &\leq \frac{Ce^{-\tfrac{y^2}{M(1+s)}}}{(1+s)^{3/2} \left(\left(\delta + \eta(s)^2\right)^{-2} + \log(1+s)\right)^{l/2} \log^{\frac{3(3-l)}{4}}(2+s)},\\
\left|R(y,s)^2 \ut(y,s)\right| &\leq \frac{C\eta(s)^2 e^{-\tfrac{y^2}{M(1+s)}}}{(1+s)^{3/2} \log^{\frac{3}{2}}(2+s)},
\end{align*}
for $l = 1,2,3$, $y \in \R$ and $s \in [0,t]$. We apply the integral identities~\eqref{integralid1} and
\begin{align*}
\int_0^t \frac{1}{(\alpha^{-2} + r)\sqrt{r}} \de r = 2 \alpha \arctan\left(\alpha \sqrt{t}\right), \qquad \alpha, t \geq 0,
\end{align*}
to conclude
\begin{align}
\begin{split}
&\left|\int_0^t \int_\R \partial_\xt^j \U(\xt-y,t-s)\ut(y,s)^3 \de y \de s\right| \\ &\leq C\int_0^t \int_\R \frac{e^{-\tfrac{(\xt-y)^2}{M(t-s)} - \tfrac{y^2}{M(1+s)}}}{(1+s)^{3/2}(t-s)^{(1+j)/2}} \left(\sum_{l = 1}^3 \frac{\left(\left(\delta + \eta(s)^2\right)^{-2} + \log(1+s)\right)^{-l/2}}{\log^{\frac{3(3-l)}{4}}(2+s)} + \frac{\eta(s)^2}{\log^{\frac{3}{2}}(2+s)}\right) \de y \de s\\
&\leq \frac{Ce^{-\tfrac{\xt^2}{M(1+t)}}}{\sqrt{1+t}} \left(\sum_{l = 2}^3 \int_0^{\log(1+\frac{t}{2})} \frac{\de r}{t^{j/2} \left(\left(\delta + \eta(t)^2\right)^{-2} + r\right)^{l/2} r^{(3-l)/2}} +\int_0^{\frac{t}{2}} \frac{\delta + \eta(t)^2}{t^{j/2} (2+s) \log^{\frac{3}{2}}(2+s)} \de s\right.\\
&\left.\qquad\qquad\qquad\qquad\qquad \qquad\qquad\qquad + \int_{\frac{t}{2}}^t \frac{\delta + \eta(t)^2}{(1+t)(t-s)^{j/2}} \de s\right) \leq C\left(\delta + \eta(t)^2\right)\frac{e^{-\tfrac{\xt^2}{M(1+t)}}}{t^{j/2}\sqrt{1+t}},
\end{split} \label{unl1}
\end{align}
for $j = 0,1$.

Using $M \geq 8d_1$, we estimate the remaining mix-term in~\eqref{Duhamelu} as follows
\begin{align}
\left|\int_0^t\int_\R \partial_\xt^j \U(\xt - y,t-s) \ut(y,s)\vt(y,s) \de y \de s\right| \leq C\eta(t)^2\left(I_j + \I_j\right), \label{estI7}
\end{align}
for $j = 0,1$, with $I_j$ defined in~\eqref{defIj} and
\begin{align*}
\I_j = \int_0^t \int_0^s \int_\R \frac{e^{-\tfrac{2(\xt-y)^2}{M(t-s)} - \tfrac{y^2}{M(1+s)}-\tfrac{(y + c(s - r))^2}{M(1+s)}}}{(1+s)(t-s)^{(1+j)/2}\sqrt{1+r}\sqrt{r}}\left(\frac{\sqrt[8]{1+r}}{r^{3/8}} + \frac{1}{\sqrt{s-r}}\right)\de y \de r \de s.
\end{align*}
An estimate on $I_j$ has already been obtained in~\eqref{estM4}. Using~\eqref{integralid1}, we calculate the inner integral in $\I_j$ and establish
\begin{align*}
\begin{split}
\I_j &\leq C \frac{e^{-\tfrac{\xt^2}{M(1+t)}}}{\sqrt{1+t}} \int_0^t \int_0^s \frac{e^{-\tfrac{(\xt + c(s-r))^2}{M (1+t)}-\tfrac{(s-r)^2 (t-s) c^2}{2 M (1+s) (1+t)}}}{(t-s)^{j/2}\sqrt{1+s}\sqrt{1+r}\sqrt{r}} \left(\frac{\sqrt[8]{1+r}}{r^{3/8}} + \frac{1}{\sqrt{s-r}}\right)\de r \de s\\&\leq C \frac{e^{-\tfrac{\xt^2}{M(1+t)}}}{\sqrt{1+t}}\left(\I_{1,j} + \I_{2,j} + \I_{3,j}\right),
\end{split}
\end{align*}
where, by~\eqref{innerest}, we have
\begin{align*}
\I_{1,j} \!:=\! \int_0^t \int_0^{\frac{s}{2}} \frac{e^{-\tfrac{s^2 (t-s) c^2}{8 M (1+s) (1+t)}}(t-s)^{-j/2}}{\sqrt{1+s}\sqrt{1+r}\sqrt{r}} \left(\frac{\sqrt[8]{1+r}}{r^{3/8}} \!+\! \frac{1}{\sqrt{s-r}}\right) \de r \de s \leq C \int_0^t \frac{e^{-\tfrac{s^2 (t-s) c^2}{8 M (1+s) (1+t)}} \de s}{(t-s)^{j/2} \sqrt{1+s}} \!\leq\! \frac{C}{(1+t)^{j/2}},
\end{align*}
and, using~\eqref{integralID5} and $c \neq 0$, it holds
\begin{align*}
\I_{2,j} &:= \int_0^{t} \int_{\frac{s}{2}}^s \frac{e^{-\tfrac{(s-r)^2 (t-s) c^2}{2 M (1+s) (1+t)}}}{(1+s)^{7/8}s^{7/8}(t-s)^{j/2}} \de r \de s \\ &\leq C\left(\int_0^{\min\{1,t\}} \frac{1}{(1+s)^{3/4}(t-s)^{j/2}} \de s + \int_{\min\{1,t\}}^t \frac{\sqrt{1+t}}{(1+s)^{3/8} s^{7/8} (t-s)^{(j+1)/2}} \de s\right) \leq \frac{C}{(1+t)^{j/2}},\\
\I_{3,j} &:= \int_0^{t} \int_{\frac{s}{2}}^s \frac{e^{-\tfrac{(s-r)^2 (t-s) c^2}{2 M (1+s) (1+t)}}}{(1+s)\sqrt{s}\sqrt{s-r}(t-s)^{j/2}} \de r \de s \\&\leq C\left(\int_0^{\min\{1,t\}}  \frac{1}{(1+s)(t-s)^{j/2}} \de s + \int_{\min\{1,t\}}^t  \frac{\sqrt[4]{1+t}}{(1+s)^{3/4}\sqrt{s}(t-s)^{(1+2j)/4}} \de s\right) \leq \frac{C}{(1+t)^{j/2}}.
\end{align*}
So, we have obtained the estimate
\begin{align*}
 \I_j \leq \frac{e^{-\tfrac{\xt^2}{M(1+t)}}}{(1+t)^{(1+j)/2}}, \qquad j = 0,1.
\end{align*}
Combining the latter with~\eqref{estM4} and~\eqref{estI7} yields the bound
\begin{align}
\left|\int_0^t\int_\R \partial_\xt^j \U(\xt -y,t-s) \ut(y,s)\vt(y,s) \de y \de s\right| \leq C\eta(t)^2 \frac{e^{-\tfrac{\xt^2}{M(1+t)}}}{t^{j/2}\sqrt{1+t}}, \qquad j=0,1.\label{unl2}
\end{align}

\paragraph*{Final estimates.} Applying the estimates~\eqref{linestu},~\eqref{unl1} and~\eqref{unl2} on the linear and nonlinear terms in~\eqref{Duhamelu} we establish
\begin{align}
e^{\tfrac{\xt^2}{M(1+t)}}\sqrt{1+t} \left(|\ut(x,t)| + \sqrt{t}|\ut_\xt(\xt,t)|\right) \leq C\left(\delta + \eta(t)^2\right). \label{finalu}
\end{align}
Finally, recall that we derived~\eqref{finalv} and~\eqref{finalu} for all $t \in [0,T)$ such that $\eta(t) \leq \eta_0$. For each $s \in [0,t]$ it holds $\eta(s) \leq \eta(t) \leq \eta_0$ and, hence, we can use~\eqref{finalv} and~\eqref{finalu} to estimate
\begin{align*}\|(\ut,\vt)(\cdot,s)\|_2 \leq C\frac{\delta + \eta(s)^2}{\sqrt[4]{1+s}},\end{align*}
for all $s \in [0,t]$. Combining the latter with~\eqref{boundderiv} yields
\begin{align}
\|\ut(\cdot,t)\|_{H^2} \leq C\left(\delta + \eta(t)^2\right). \label{finalH}
\end{align}

\subsection{Conclusion}
For initial conditions $(u_0,v_0) \in X_{\rho_E}^\alpha$ satisfying~\eqref{initcond4} and for $\eta_0 > 0$ sufficiently small, the estimates~\eqref{finalv},~\eqref{finalr},~\eqref{finalu} and~\eqref{finalH} yield the key inequality~\eqref{etaest5}, which proves, as explained in~\S\ref{sec:planMR4}, Theorem~\ref{mainresult4}. $\hfill \Box$

\section{Future outlook} \label{sec:discussion}

In this section we comment on open problems, future extensions and possible applications of our results.

\paragraph*{Multiple spatial dimensions or multiple components.} Perhaps the most natural way to extend our results is to increase the number of components or spatial dimensions in~\eqref{GRD}, i.e.~to consider the class of reaction-diffusion-advection systems
\begin{align}
u_t = D \Delta u + \sum_{i = 1}^d C_i u_{x_i} + f(u) + \sum_{i = 1}^d \left(g_i(u)\right)_{x_i}, \quad u(x,t) \in \R^n, x \in \R^d, \label{RD}
\end{align}
where $\Delta$ is the Laplacian, $D$ is a nonnegative diagonal matrix of diffusion coefficients, $C_i$ is a diagonal matrix of velocities and $f,g_i \colon \R^n \to \R^n$ are smooth nonlinearities with $f(0), (Df)(0) = 0$ and $g_i(0), (Dg_i)(0) = 0$. Systems of the form~\eqref{RD} model $n$ reactants, which are subject to species-dependent diffusion and drift, in an unbounded, $d$-dimensional domain.

Increasing the spatial dimension $d$ improves the temporal decay on the linear level. Indeed, localized solutions in $L^1(\R^d,\R^n) \cap L^\infty(\R^d,\R^n)$ to the associated linear system
\begin{align*}u_t = D\Delta u + \sum_{i = 1}^d C_i u_{x_i},\end{align*}
decay with rate $t^{-d/2}$. Thus, we expect that our results on global existence, Theorems~\ref{mainresult1} and~\ref{mainresult2}, have counterparts in higher spatial dimensions (upon adapting the weights to $\rho_E(x) = e^{-\|x\|^2/M}$ and $\rho_A(x) = (1+\|x\|)^r$). In fact, we expect that the proofs simplify: due to the improved temporal decay properties only quadratic nonlinearities are marginal if $d = 2$ and all smooth nonlinearities are irrelevant for $d \geq 3$. This leads to the natural question whether the additional decay can be exploited to relax the localization assumption on the initial data. In the setting of planar traveling waves, it is shown in~\cite{RSAN} that one can allow for non-localized perturbations. We expect that these results transfer to the current setting in this paper.

We expect that an extension to multiple components of Theorem~\ref{mainresult1} is rather straightforward as long as the nonlinear terms in the $i$-th component have a nontrivial contribution of the $i$-th component itself, i.e.~we require that the nonlinear terms in the $i$-th component are of the form $\partial_x^j (u_i h(u))$ with $j = 0,1$ and $h \colon \R^n \to \R$. Indeed, the spatio-temporal weights on the $i$-th component in the proof of Theorem~\ref{mainresult1} (see the definition of the functions $\eta_E$ and $\eta_A$ in~\S\ref{sec:planM}) depend only on the velocity of the $i$-th component itself and not on the characteristics of the other components.

At the moment, it is rather unclear whether Theorem~\ref{mainresult2} can be extended to the multi-component setting, i.e.~whether we can accommodate any irrelevant nonlinearity and any mix-term in the case of multiple components. In the proof of Theorem~\ref{mainresult2}, the spatio-temporal weight on the $i$-th component (see the definition of the function $\eta$ in~\S\ref{sec:planIRR}) depends on the velocities of \emph{all} components due to the presence of terms of the form~\eqref{drag}. Thus, the number of terms in the spatio-temporal weight would increase rapidly with the number of components, which could complicate the analysis. In addition, in systems with more than two components, mix-terms in the $i$-th component can occur which do not have contributions from the $i$-th component itself. Currently, it is still open how to control such mix-terms. For instance, a term $u_2u_3$ in the equation for $u_1$ leads to a bound of the form
\begin{align*}
\int_0^t\int_{\R} \frac{e^{-\tfrac{2(x-y+c_1(t-s))^2}{M(t-s)} - \tfrac{(y+c_2s)^2}{M(1+s)} - \tfrac{(y+c_3s)^2}{M(1+s)}}}{(1+s)\sqrt{t-s}}\de y\de s,
\end{align*}
which can be compared with the bound $I_0$ in~\eqref{defIj}. However, treating this expression in a similar way as $I_0$ in the hope to
obtain a similar bound as~\eqref{estM4} is problematic as the analysis in~\S\ref{estexpM} breaks down when $(c_1-c_2)(c_1-c_3) > 0$.

\paragraph*{Larger classes of initial data.} Another possible future direction is to study to what extend we can relax the localization requirements on the initial data in our main results, Theorem~\ref{mainresult1} and~\ref{mainresult2}. A first attempt to obtain a larger class of initial data is to relax the algebraic localization in Theorem~\ref{mainresult1} by reducing the power $r$. In this light, we remark that in the global existence analysis in the nonlinear heat equation $u_t = u_{xx} + u^p$ with $p > 3$ in~\cite[Section 6]{JUN} one takes $r > 2$ instead of $r \geq 3$.

Second, one could try to prove Theorem~\ref{mainresult2}  for algebraically localized initial conditions. For such initial conditions, irrelevant nonlinear coupling terms which are not of mix-type, lead, as in~\S\ref{sec:irrelevantcoupling}, to bounds of the form
\begin{align}
\int_0^t\int_{\R} \frac{e^{-\tfrac{4(x-y+c_1(t-s))^2}{M(t-s)}}}{\left(1 + |x+c_2s| + \sqrt{s}\right)^{2r} \sqrt{t-s}} \de y\de s, \label{nasty}
\end{align}
with $c_1 \neq c_2$, which can be compared with the integrals $II_0$ and $III_0$ occurring in~\S\ref{sec:nonlalg}. However, the fact that $c_1 \neq c_2$ seems to prohibit a similar treatment as in~\S\ref{sec:nonlalg}. The integrals occurring on pages 344 and 349 in~\cite{HOWZUM} are similar to~\eqref{nasty} and therefore this reference could be of help.

\paragraph*{Full characterization of admissible nonlinearities.} In the scalar setting it is well-known, see Remark~\ref{scalar}, which smooth nonlinearities are controlled by the linear dynamics and which ones might lead to growth or even finite time blow-up. At the moment such a full characterization of admissible smooth nonlinearities is still open for reaction-diffusion-advection \emph{systems}, although some specific cases have been addressed, see for instance~\cite{ESCHER,ESCLEV,BKX}

In this paper we make a first step in the direction of a full characterization by showing that any nonlinearity containing irrelevant terms and mix-terms can be controlled by the linear dynamics when components propagate with different velocities. However, even when components exhibit different velocities, the question which other relevant or marginal nonlinear terms can be included is very subtle. Theorem~\ref{mainresult4} even demonstrates that one needs to consider the \emph{full} nonlinearity, instead of just the leading order (or `most dangerous') term of the nonlinearity, like in the scalar setting.

\paragraph*{Application to the nonlinear stability analysis of wave trains.} The perturbation equations arising in the nonlinear stability analysis of wave-train solutions to reaction-diffusion systems are, in the appropriate co-moving frame, reaction-diffusion-advection systems of the form~\eqref{GRD} where the coefficients are spatially periodic. Thus, techniques to prove global existence and decay of small solutions (like the methods mentioned in~\S\ref{overview}) can be employed in these perturbation equations to prove nonlinear stability, see for instance~\cite{JUN,JONZ,SAN3}.

We expect that the techniques developed in this paper can be applied to prove nonlinear stability of wave trains at the \emph{Eckhaus boundary}, which is believed to play an important role in pattern formation~\cite{ahlers}. At the Eckhaus boundary, wave trains are marginally spectrally stable and at the threshold of destabilization. Yet, nonlinear stability was proven in the case of a sideband destabilization~\cite{ECKSCH}. When the wave train is at the threshold of a Hopf destabilization, the associated critical modes generally exhibit different group velocities. We expect that the methods developed in this paper could be employed to exploit this difference in group velocities and prove nonlinear stability for wave trains at the Eckhaus boundary in the case of a Hopf destabilization.

\bibliographystyle{plain}
\bibliography{mybib}

\end{document}